\documentclass[11pt, reqno]{amsart}

\usepackage{amsmath, amsthm, amssymb}
\usepackage{mathtools,enumerate,thmtools,bm}

\usepackage[margin=1.5in]{geometry}
\usepackage{arydshln}
\usepackage{graphicx}
\usepackage{mathdots}
\usepackage[colorlinks=true,urlcolor=blue,linkcolor=blue,citecolor=blue,psdextra]{hyperref}
\usepackage{cleveref}
\usepackage[utf8]{inputenc}
\usepackage[T1]{fontenc}

\newcommand*{\trn}[2][-3mu]{\ensuremath{\mskip1mu\prescript{\smash{\mathrm t\mkern#1}}{}{\mathstrut#2}}}%

\usepackage{bookmark}
\hypersetup{
  bookmarksopen=true,
  unicode=true             
}

\usepackage{marginnote,letltxmacro}

\LetLtxMacro\mn\marginnote

\theoremstyle{plain}
\newtheorem{theorem}{Theorem}[section]
\newtheorem{proposition}[theorem]{Proposition}
\newtheorem{lemma}[theorem]{Lemma}

\newtheorem{corollary}[theorem]{Corollary}
\newtheorem{claim}{Claim}[section]
\newtheorem{subclaim}{Subclaim}[claim]

\theoremstyle{remark}
\newtheorem{remark}{Remark}[section]

\theoremstyle{definition}
\newtheorem{definition}{Definition}[section]

\newcommand{\DI}{\mathrm{DI}}

\newcommand{\N}{\mathbb{N}}
\newcommand{\Z}{\mathbb{Z}}

\newcommand{\R}{\mathbb{R}}
\newcommand{\Cx}{\mathbb{C}}
\newcommand{\cA}{\mathcal{A}}
\newcommand{\cP}{\mathcal P}
\newcommand{\cH}{\mathcal H}
\newcommand{\cJ}{\mathcal J}
\newcommand{\cN}{\mathcal{N}}
\newcommand{\cT}{\mathcal{T}}
\newcommand{\SL}{\mathrm{SL}}
\newcommand{\GL}{\mathrm{GL}}
\newcommand{\SO}{\mathrm{SO}}
\newcommand{\diag}{\mathrm{diag}}
\newcommand{\Lie}{\mathrm{Lie}}
\newcommand{\Proj}{\mathrm{Pr}}
\newcommand{\Span}{\mathrm{span}}
\newcommand{\inv}{^{-1}}
\newcommand{\cl}{\overline}
\newcommand{\Mat}{\mathrm{Mat}}
\newcommand{\Po}{Q} 
\newcommand{\HA}{H_{\bfxi}}
\newcommand{\Hcmin}{{H_C\text{-}\min}}
\newcommand{\Hmin}{{H\text{-}\min}}
\newcommand{\Hcmax}{{H_C\text{-}\max}}
\newcommand{\Hmax}{{H\text{-}\max}}
\newcommand{\bydef}{\stackrel{\mathrm{def}}{=}}

\DeclarePairedDelimiter{\norm}{\lVert}{\rVert}

\DeclarePairedDelimiter{\abs}{\lvert}{\rvert}

\DeclareMathOperator{\Ad}{Ad}
\DeclareMathOperator{\supp}{supp}

\DeclareMathOperator{\ad}{ad}
\newcommand{\zspn}[1]{\langle #1 \rangle_{\Z}}

\newcommand{\h}{\mathfrak{h}}
\newcommand{\la}[1]{\mathfrak{#1}}

\newcommand{\bfx}{\mathbf{x}}
\newcommand{\bfy}{\mathbf{y}}
\newcommand{\bfe}{\mathbf{e}}
\newcommand{\bfo}{\mathbf{o}}
\newcommand{\bfc}{\mathbf{c}}

\newcommand{\bff}{\mathbf{f}}
\newcommand{\bfkappa}{\bm{\kappa}}
\newcommand{\bfxi}{\bm{\xi}}

\numberwithin{equation}{section}

\title[Weighted Dirichlet]{Equidistribution of non-uniformly stretching translates of shrinking smooth curves and weighted Dirichlet approximation}
\author{Nimish A. Shah}
\author{Pengyu Yang}
\address{The Ohio State University, Columbus, OH 43210; email: shah@math.osu.edu} 
\address{Morningside Center of Mathematics, Chinese Academy of Sciences, Beijing 100190; email:yangpengyu@amss.ac.cn}

\thanks{This material is based upon work supported by the National Science Foundation under Grant No.\ DMS-1700394.}

\begin{document}

\begin{abstract}
Consider the action of $a_t=\mathrm{diag}(e^{nt},e^{-r_1(t)},\ldots,e^{-r_n(t)})\in\mathrm{SL}(n+1,\mathbb{R})$, where $r_i(t)\to\infty$ for each $i$, on the space of unimodular lattices in $\mathbb{R}^{n+1}$. We show that $a_t$-translates of segments of size $e^{-t}$ about all except countably many  points of a nondegenererate smooth horospherical curve get equidistributed in the space as $t\to\infty$. From this, it follows that the weighted Dirichlet approximation theorem cannot be improved for almost all points on any nondegenerate $C^{2n}$ curve in $\mathbb{R}^n$. These results extend the corresponding results for translates of fixed pieces of analytic curves due to Shah (2009), and answer some questions inspired by the work of Davenport and Schmidt (1969) and Kleinbock and Weiss (2008).
\end{abstract}

\subjclass[2010]{Primary 37A17, 22E46; Secondary 11J13}
	\keywords{Homogeneous dynamics, unipotent flow, Dirichlet-improvable vectors, equidistribution}

	\maketitle
	\setcounter{tocdepth}{1}
	\tableofcontents

\section{Introduction} \label{sec:intro}
\subsection{Equidistribution of expanding translates of smooth curves}
Let $n\in\N$. Let $L$ be a Lie group, $\Lambda$ be a lattice in $L$, and suppose that $G=\SL(n+1,\R)$ acts on $L/\Lambda$ via a Lie group homomorphism from $G$ to $L$. Let $\cP(L/\Lambda)$ denote the space of Borel probability measures on $L/\Lambda$ endowed with the weak$^\ast$-topology, and the $L$-action on it is given by $(g\mu)(E):=\mu(g^{-1}E)$ for all Borel measurable $E\subset L/\Lambda$, for all $g\in L$ and $\mu\in \cP(L/\Lambda)$.
	
	Let $x\in L/\Lambda$. For any subgroup $F$ of $G$ generated by unipotent elements, by Ratner's measure classification theorem (see~\cite{Ratner1991-Measure}) there is a unique $F$-invariant probability measure $\mu_{\cl{Fx}}$ on the homogeneous space $\cl{Fx}$.   
	
	 Let $\cT$ be a subnet or a subsequence of the directed set $[0,\infty)$. We will write $t\to\infty$ or $t_i\to\infty$ to mean a subnet or a subsequence of $\cT$. For each $t\in\cT$, let $r_1(t)\geq r_2(t)\geq\cdots\geq r_n(t)\geq 0$ be such that $\sum_{i=1}^{n}r_i(t)=nt$; we will treat each $r_i$ as a functions of $t\in\cT$. 
	 For all $t\in\cT$ and $\bfx=(x_1,\ldots,x_n)\in\R^n$, define
	\begin{align} 
    a_t&=\begin{bsmallmatrix}
     e^{nt}  \\
     & e^{-r_1(t)}  \\
     &            & \ddots  \\   
     &            &        & e^{-r_n(t)}
    \end{bsmallmatrix} \in G, \text{ and }
    u(\bfx)=
	\begin{bsmallmatrix}
	1 & x_1 & \cdots & x_n \\
	  & 1\\
	  & & \ddots\\
	  &&&1
	\end{bsmallmatrix}\in G. \label{def2:ri}
	\end{align}

 We note that $t\leq r_1(t)\leq nt$. To have simpler statements of theorems we will assume that for some $1\leq n_0\leq n$, we have
	 \begin{equation} \label{def:n0}
	 \lim_{t\to\infty} r_{n_0}(t)=\infty \text {, and } \lim_{t\to\infty} r_i(t)=0, \,\forall n_0<i\leq n.
	 \end{equation}
	Let 
	\begin{gather} \label{eq:Fmbyn-1}
  G_{n_0}:=\left\{
  \begin{bsmallmatrix}
  C & D \\
   & I_{n-n_0}
  \end{bsmallmatrix}\in G: C\in \SL(n_0+1,\R) \text{ and } D\in \Mat_{n_0+1,n-n_0}\right\},
  \end{gather}
  where $I_k$ denotes the $k\times k$-identity matrix and $\Mat_{k,l}$ denotes the space of $k\times l$ real matrices. We note that $u(\R^n)\subset G_{n_0}$ and $a_tG_{n_0}\to G_{n_0}$ in $G/G_{n_0}$ as $t\to\infty$. Also note that if $n_0=n$, then $G_{n_0}=G$.

	
	 A map $\phi:(0,1)\to\R^n$ is said to be \emph{$\ell$-nondegenerate\/} at $s\in (0,1)$ for some $\ell\geq n$, if the derivatives $\phi^{(i)}(s)$ exist for all $1\leq i\leq \ell$ and they span $\R^n$; see the paragraph preceding Theorem~A in \cite{Kleinbock1998}. We say that $\phi$ is $\ell$-nondegenerate for some $\ell\geq n$, if $\phi$ is $\ell$-nondegenerate at all $s\in (0,1)$. If $\phi$ is $C^k$ and $\ell$-nondegenerate for some $\ell\leq k$, then $\phi$ is $k$-nondegenerate. 
	 

Let $x\in L/\Lambda$ and $g_t\to e$  in $L$ such that
\begin{equation} \label{eq:gtn0}
\cl{G_{n_0}g_tx}\subset g_t\cl{G_{n_0}x},\,\forall t.
\end{equation}
Note that \eqref{eq:gtn0} is satisfied if $\cl{G_{n_0}x}=L/\Lambda$, or if $\{g_t\}_t\subset N_L(G)$. 
	
	We pick $k\in\N$ such that 
	\begin{equation} \label{def:k}
	\sup_{t\in\cT} (nt+r_1(t)-kt)<\infty.
	\end{equation}
	Since $t\leq r_1(t)\leq nt$, we need $k\geq n+1$, and any $k\geq 2n$ satisfies \eqref{def:k}.
	
	\begin{theorem} \label{thm:smooth} For $k$ as in \eqref{def:k}, let $\phi:(0,1)\to\R^n$ be a $k$-nondegenerate $C^k$-map. Let $\sigma$ be an absolutely continuous Borel probability measure on $(0,1)$. Then for any $f\in C_c(L/\Lambda)$,
	\begin{equation} \label{eq:mainlimit}
	\lim_{t\to\infty} \int f(a_tu(\phi(s))g_tx)\,d\sigma(s) = 
	\int_{L/\Lambda} f\,d\mu_{\cl{G_{n_0}x}}.
	\end{equation}
	\end{theorem}
	
	This extends the main theorem of \cite{Shah:Minkowski} from analytic curves to $C^k$ curves. 

\subsection{Non-improvability of Dirichlet-Minkowski theorem}  Let $\cN$ be an infinite sequence of $\N^n$. Let $0<\mu\leq 1$. Let $\DI(\cN,\mu)$ denote the set of vectors $(\xi_1,\ldots,\xi_n)\in\R^n$ with the property that for all but finitely many $(N_1,\ldots,N_n)\in \cN$, there exist $(q_1,\ldots,q_n)\in \Z^n\setminus\{0\}$ and $p\in \Z$ such that 
\[
\abs{\xi_1q_1+\ldots+\xi_nq_n-p}\leq \mu(N_1\cdots N_n)^{-1} \text{ and } \abs{q_i}\leq N_i, \,\forall 1\leq i\leq n.
\]

Similarly, let $\DI'(\cN,\mu)$ denote the set of vectors $(\xi_1,\ldots,\xi_n)\in\R^n$ with the property that for all but finitely many $(N_1,\ldots,N_n)\in \cN$, there exist $q\in\Z\setminus\{0\}$ and $(p_1,\ldots,p_n)\in \Z^n$ such that 
\[
\abs{\xi_i q-p_i}\leq \mu N_i,\, \forall 1\leq i\leq n \text{, and } \abs{q}\leq N_1\cdots N_n.
\]

Let 
\[
\DI(\cN):=\cup_{0<\mu<1} \DI(\cN,\mu) \text{ and } \DI'(\cN):=\cup_{0<\mu<1} \DI'(\cN,\mu).
\]

By Minkowski's extension of Dirichlet's theorem,  $\DI(\cN,1)=\DI'(\cN,1)=\R^n$. On the other hand, Davenport and Schmidt~\cite{DavenportSchmidt1968-69,Davenport1970DirichletsII} and Kleinbock and Weiss \cite{Kleinbock2008DirichletsFlows} showed that $\DI(\cN)$ and $\DI'(\cN)$ are Lebesgue null. In \cite{Shah-Dirichlet,Shah:Minkowski} it was shown that if $\phi:(0,1)\to\R^n$ is an analytic map whose image is not contained in a proper affine subspace, then  $\phi(s)\not\in \DI(\cN)\cup\DI'(\cN)$ for almost every $s$. In \cite{Shi2017}, Shi and Weiss showed that if $\phi:(0,1)\to\R^2$ is a $2$-nondegenerate map, then $\phi(s)\not\in\DI(\{(N,N):N\in\N\})$ for almost all $s$. 

Given an infinite sequence $\cN\subset \N^n$, let
\[
\bar r_1(\cN):=\limsup_{(N_1,\ldots,N_n)\in\cN} \frac{\log(\max(N_1,\ldots,N_n))}{\log(N_1\cdots N_n)}.
\]
Then $1/n\leq \bar r_1(\cN)\leq 1$. 

\begin{theorem} Let $\phi:(0,1)\to\R^n$ be a $k$-nondegenerate $C^k$-map and $\cN\subset \N^n$ be an infinite sequence. Suppose that $k\geq n(1+\bar r_1(\cN))$. Then the vector $\phi(s)\notin\DI(\cN)\cup \DI'(\cN)$ for Lebesgue almost all $s\in (0,1)$.
\end{theorem}

One can deduce this result from Theorem~\ref{thm:smooth} as in \cite[Section 2]{Shah:Minkowski} using Dani's correspondence, which was developed in \cite{Dani1985}, \cite{Kleinbock2008DirichletsFlows}, and \cite{Shah-Dirichlet}. 

The special case of this theorem when $\cN=\{(N,\ldots,N):N\in\N\}$ was obtained in \cite{Shah2018}. Using the arguments of \cite[Section 4 and 5]{Shah2018}, one can extend \Cref{thm:smooth} to Pyartli-type nondegenerate $C^{2n+1}$ immersions $\phi:(0,1)^d\to \R^n$, see \cite[Definition~1.1]{Shah2018}. Then it  follows that $\phi(s)\notin \DI(\cN)\cup\DI'(\cN)$ for almost all $s\in (0,1)^d$.

\subsection{Equidistribution of expanding translates of shrinking curves}
\Cref{thm:smooth} can be deduced from each of its `shrinking curve' versions described below. We would like to show that if for each $t$ we choose a short pieces of the smooth curve  around a base point and translate that piece by $a_t$, then as $t\to\infty$, the the expanded long piece `converges' to some algebraic measure in the homogeneous space. 

\begin{theorem}[Equidistribution at all but countably many points] \label{thm:co-countable} For $k$ as in \eqref{def:k}, let $\phi:(0,1)\to\R^n$ be a $k$-nondegenerate $C^k$-map. Given $x\in L/\Lambda$, there exists a countable set $E_x\subset (0,1)$ such that the following holds: For any absolutely continuous Borel probability measure $\nu$ on $\R$, $f\in C_c(L/\Lambda)$, $s\in (0,1)\setminus E_x$, a sequence $s_i\to s$, and a sequences $t_i\to\infty$, we have
	\begin{equation} \label{eq:equidistribution}
	\lim_{i\to\infty} \int  f(a_{t_i} u(\phi(s_i+e^{-t_i}\eta))g_ix)\,d\nu(\eta) = \int f\,d\mu_{\cl{G_{n_0}x}},
	\end{equation}
	for any sequence $g_i\to e$ in $L$ satisfying \eqref{eq:gtn0}; that is $\cl{G_{n_0}g_ix}\subset g_i\cl{G_{n_0}x}$ for all $i$.
	\end{theorem}

To describe such a limit distribution at a given $s\in (0,1)$ we need the following definition.

	\begin{definition}[Ordered regular]
	\label{def:ord_reg} 
	Let $\phi:(0,1)\to \R^n$ and $s\in (0,1)$ such that $\phi$ is $C^{n}$ in a neighbourhood of $s$.  To say $\phi$ is \emph{ordered regular\/} at $s$ means that for each $1\leq i\leq n$, the linear span of $\{\phi^{(1)}(s),\ldots,\phi^{(i)}(s)\}$ projects onto $\R^i$ under the map $\pi_i((x_1,\ldots,x_n))=(x_1,\ldots,x_i)$; that is,  $\pi_i\circ\phi$ is $i$-nondegenerate at $s$. 
	\end{definition}
	
	If $\phi$ is $\ell$-nondegenerate for some $\ell\geq n$, then $\phi$ is ordered regular at each  $s\in (0,1)\setminus Z$, where $Z$ is a discrete subset of $(0,1)$, see \Cref{cor:ord-reg}. We note that $0$ or $1$ may be limit points of $Z$ in $[0,1]$.
	
\begin{remark} \label{rem:OR} (Algebraic group interpretation of ordered regularity.) Let $N^+$ (resp. $N^-$) denote the upper (resp. lower) triangular unipotent subgroup of $\GL(n, \R)$. Let $D$ denote the full diagonal subgroup of $\GL(n,\R)$. Then $N^-DN^+$ is 
	a Zariski open dense subset of $\GL(n,\R)$. For $s\in (0,1)$, let
	\[
	M_{\phi}(s)=\begin{bsmallmatrix}
	\phi^{(1)}(s)/1!\\
	    \vdots         \\
	 	\phi^{(n)}(s)/n!
	\end{bsmallmatrix}\in \GL(n,\R)
	\]
	Using Gauss elimination on the columns of $M_\phi(s)$, it is straightforward to verify that $\phi$ is ordered regular at $s$ if and only if $M_{\phi}(s)\in N^-DN^+$. So if $\phi$ is ordered regular at $s$, then there exist unique $B(s)\in N^+$ and $\kappa_i(s)\in\R^\times$ for $1\leq i \leq n$ such that 
	\[
	M_{\phi(s)}\in N^-\diag(\kappa_1(s),\ldots,\kappa_n(s))B(s);
	\]
	that is, for each $1\leq i\leq n$
	\begin{equation} \label{eq:Bs}
(\phi^{(i)}(s)/i!)B(s)\inv \in \kappa_i(s)\bfe_i+\Span\{\bfe_1,\ldots,\bfe_{i-1}\},
	\end{equation}
	where $\bfe_1,\ldots,\bfe_n$ denote the standard basis of $\R^n$.
    \end{remark}

The next result says that we get equidistribution if we shrink `slower' then $e^{-t}$. 

\begin{theorem}[Equidistribution under slower shrinking] \label{thm:slow-shrink} Suppose that $k>n+\limsup_{t\to\infty} t\inv r_1(t)$. Let $\phi:(0,1)\to\R^n$ and $s\in (0,1)$ be such that $\phi$ is $C^k$ in a neighborhood of $s$ and ordered regular at $s$. For each $t\geq 0$, let $\beta_t\to\infty$ such that $\beta_te^{-t}\to0$ as $t\to\infty$. Then for any absolutely continuous Borel probability measure $\nu$ on $\R$ and any $f\in C_c(L/\Lambda)$,
\begin{equation} \label{eq:slow-shrink-equi}
	\lim_{t\to\infty} \int  f(a_t u(\phi(s+\beta_t e^{-t}\eta))x)\,d\nu(\eta) = \int f\,d\mu_{\cl{G_{n_0}x}}.
	\end{equation}
	\end{theorem}
	
The next result shows that if we shrink at the `optimal' rate of $e^{-t}$ around a point $s$, then the corresponding limiting distribution turns out to be an integral of a translates of a homogeneous measure. 

We say that $\{a_t\}_t$ is {\em uniform\/} if  $\limsup_{t\to\infty} r_1(t)-r_n(t)<\infty$; in other words, $\{a_t\inv \diag(e^{nt},e^{-t},\ldots,e^{-t}):t\geq 0\}$ is contained in a compact subset of $G$. We say that $\{a_t\}$ is {\em non-uniform\/} if $\liminf_{t\to\infty} r_1(t)-r_n(t)=\infty$. 

\begin{theorem} \label{thm:main}
Suppose that $\{a_t\}_t$ is non-uniform. Let $\phi:(0,1)\to\R^n$ and $s\in (0,1)$ be such that $\phi$ is $C^k$ in a neighborhood of $s$ and ordered regular at $s$. Suppose that $\supp\nu\subset [0,\infty)$. Then for any $x\in L/\Lambda$ and any $f\in C_c(L/\Lambda)$, 
	\begin{align} 
	    &\lim_{t\to\infty}\int_0^\infty f(a_t u(\phi(s+e^{-t}\eta))x)\,d\nu(\eta)
	    \label{eq:MainLimitLeft}\\
	    &=
	\int_0^\infty\left(\int_{L/\Lambda} f(\exp((\log\eta) H_{n_0})v(s)_{n-n_0}\inv w(\kappa_n(s))y)\,) d\mu_{\cl{\Po_{n_0} x_s}}(y)\right)\,d\nu(\eta), 
	\label{eq:MainLimitRight}
	\end{align}
	where $x_s=v(s)u(\phi(s))x$, $\kappa_n(s)\in\R^{\times}$ and $B(s)\in N^+$ are as in \eqref{eq:Bs},  
	\begin{align} \label{eq:vs}
	        v(s)&=\begin{bsmallmatrix} 1  \\ &B(s) \end{bsmallmatrix}
	        \text{ and } 
	        v(s)_{n-n_0}= \begin{bsmallmatrix}
	        1\\&I_{n_0}\\&&B(s)_{n-n_0}
	        \end{bsmallmatrix}, 
	 \end{align}
	 where $B(s)_{n-n_0}$ is the lower right $(n-n_0)\times(n-n_0)$ block of $B(s)$,
	 \begin{align}
	\Po_{n_0}&=\Po\cap G_{n_0} \text{, where } \Po:=\begin{bsmallmatrix} \SL(n,\R) \\ \R^n &1 \end{bsmallmatrix}, \label{eq:Q}\\
	\label{eq:wxi}
	    w(\kappa)&=\begin{cases} 
	    \sigma(\kappa):=\begin{bsmallmatrix} &&  \kappa \\ &I_{n-1}&\\ -\kappa^{-1}&&\end{bsmallmatrix},
	    & \text{if } \kappa\neq 0 \text{ and } n_0=n \\
	    u((0,\ldots,0,\kappa))=\begin{bsmallmatrix}
	    1 & &\kappa\\& I_{n-1}\\&&1
	    \end{bsmallmatrix}, & 
	    \text{if } n_0<n,
	    \end{cases}\\
	    H_{n_0}&=({n}/{n_0})\diag(n_0,-1,\ldots,-1,0,\ldots,0)\in \mathfrak{sl}(n+1,\R).
	    \label{eq:Hn0}
	    \end{align}
	\end{theorem}

\begin{remark} (1) In \Cref{thm:main}, if we replace $s+e^{-t}\eta$ by $s-e^{-t}\eta$ in \eqref{eq:MainLimitLeft}, then \eqref{eq:MainLimitRight} holds with $\kappa_n(s)$ replaced by $(-1)^n\kappa_n(s)$; see \Cref{rem:h-negative}.
So, by combining the results for the cases of $s\pm e^{-t}\eta$, we obtain the analogous limiting distribution result for all probability measures $\nu$ on $\R$. 

(2) For uniformly expanding $\{a_t\}_t$, the analogue of \Cref{thm:main} was obtained in \cite[Theorem 3.4]{Shah2018}, where the role of $\Po_{n_0}=Q$ is played by a connected Lie group whose Lie algebra is spanned by $\{Z,Z^2,\ldots,Z^n\}$ with \(
Z=\begin{bsmallmatrix}
	&0\\I_n
	\end{bsmallmatrix}
\).
This result was proved via a very different approach involving equidistribution of polynomial trajectories on homogeneous spaces \cite[Theorem~1.1]{Shah-poly-Duke}.  That approach does not yield \Cref{thm:slow-shrink} in the uniform case.
\end{remark}

\begin{remark}[Faster shrinking] \label{rem:shrink fast}
For Theorems~\ref{thm:co-countable} and \ref{thm:main}, the shrinking rate of $e^{-t}$ can be considered optimal, as shrinking any faster may fail to yield a similar limiting equidistribution results in general. For example, let $n=1$, $L=G=\SL(2,\R)$, $a_t=\diag(e^t, e^{-t})$, and $x\in L/\Lambda$ such that $x$ is fixed by $\begin{bsmallmatrix} 1 & 0\\ 1 & 1\end{bsmallmatrix}$. If we shrink at the rate of $o(e^{-t})$, then the limiting measure will be zero on $L/\Lambda$; that is, the translated measures will escape to infinity. 
\end{remark}

\subsection{Organization of the paper} We will study the limiting distribution of translations of parameter measures on curves on the homogeneous space $L/\Lambda$ via the following standard scheme: (1) Verify non-divergence criterion due to Dani and Margulis and show that the measures do not escape to infinity; (2) show that any limiting measure is invariant under a non-trivial unipotent subgroup of $G$; (3) apply Ratner's theorem to conclude that any limiting measure is concentrated on a union of certain types of algebraic subvarieties of $L$ projected to $L/\Lambda$; (4) apply linearization technique to show that the limit measure will be zero on the images of the algebraic subvarieties, except for some specific subvariety, and provide precise description of the limiting measures. 

This scheme was followed in many papers, for example \cite{Shah-geodesic1,Shah:Minkowski,Yang2016, Yang2020}, for translates of measures on analytic curves. The Dani-Margulis non-divergence criterion and the linearization techniques are available only for the $(C,\alpha)$-good functions, which were introduced by Kleinbock and Margulis \cite{Kleinbock1998}. The orbits of the translates of an analytic curve in a finite dimensional representation of $G$ are $(C,\alpha)$-good, but the same need not hold for an differentiable curve. Therefore we need to approximate $\phi$ using a fixed degree polynomial associated to its Taylor expansion around a point $s$. Since the errors are also expanded by $a_t$, we will work with pieces of curves around $s$ which shrink at a suitable rate as $t\to\infty$. This approach was carried out for $G=\SO(n,1)$ in \cite{Shah-geodesic-smooth09}, but its generalization to a higher rank group turned out to be difficult. 

Indeed, in \Cref{sect:linear_dynamics} we prove an expansion result (\Cref{thm:basic_lemma}) about dynamical interactions of certain diagonal elements and unipotent elements on a finite dimensional representation of $G$. That result will allow us to extend the "basic lemmas" of the above mentioned papers (for fixed pieces of analytic curves) to shrinking pieces of smooth curves in \Cref{sec:curves}. Using the extended basic lemmas (\Cref{prop:main-basic} and \Cref{prop:main-bdd}), in \Cref{sec:Ratner} we will accomplish all the steps of the above scheme, and derive technical versions of the theorems stated in the introduction, and complete the proof of \Cref{thm:main}. In \Cref{sec:slow} we provide the proofs of \Cref{thm:slow-shrink} and \Cref{thm:smooth}.  In \Cref{sec:butcountable} we prove discreteness of $s\in (0,1)$ with certain exceptional properties (\Cref{prop:discrete}) in the non-uniform case, and complete the proof of \Cref{thm:co-countable}. 

\subsubsection*{Acknowledgement} We would like to thank Manfred Einsiedler, Dmitry Kleinbock, Elon Lindenstrauss, and Lei Yang for valuable discussions on this topic.  

\section{Expansion in linear representations} \label{sect:linear_dynamics}
	Let $\rho\colon G\to\GL(V)$ be a finite dimensional linear representation of $G$. It induces a representation $d\rho\colon\mathfrak{g}\to\mathrm{End}(V)$ of the Lie algebra $\mathfrak{g}=\mathfrak{sl}(n+1,\R)$. For any element $g$ in $G$, $X$ in $\mathfrak{g}$ and any vector $v$ in $V$, we write $g\cdot v=\rho(g)v$ and $X\cdot v = d\rho(X)v$ for simplicity.
	Let $\h$ denote the Cartan subalgebra of $\mathfrak{g}$ consisting of diagonal matrices. Then we have a weight space decomposition
	\[
	V=\bigoplus_{\lambda\in\h^*}V_\lambda,
	\]
	where $V_{\lambda}=\{v\in V\colon H\cdot v=\lambda(H)v,\; \forall H\in\h \}$. Then $V_\lambda\neq0$ for only finitely many $\lambda\in\h^\ast$. For any nonzero vector $v$ in $V$, we express  $v=\sum_{\lambda\in\h^\ast} v_\lambda$, where $v_\lambda\in V_\lambda$ and define 
	\[
	\Lambda_v=\{\lambda\in \h^*\colon v_\lambda\neq 0 \},
	\]
and for	any $S\subset \h^\ast$, define 
\begin{equation} \label{eq:subF}
v_S=\sum_{\lambda\in S} v_{\lambda}.
\end{equation} 
Let $g\in G$ such that $\Ad g$ preserves $\h$; that is, $g\in N_G(\h)$. Then $\Ad g$ acts on $\h$ as a Weyl group element, and its action on  $\h^\ast$ is given by: $\forall\,\lambda\in\h^\ast$ and $H\in\h$,
\begin{equation} \label{eq:Weylonh}
(g\lambda)(H)=\lambda(g\inv Hg).
\end{equation} 
Therefore for any $v\in V$ and $S\subset \Lambda_v$,
\begin{equation} \label{eq:WeylonV}
    \Lambda_{g v}=g\Lambda_v \text{ and } (gv)_{gS}=g(v_S).
\end{equation}

For the standard action of $G=\SL(n+1,\R)$ on $\R^{n+1}$ we take $\{e_0,e_1,\ldots,e_n\}$ as the standard basis of $\R^{n+1}$ and we express an element $g\in G$ in the matrix form $(g_{i,j})$, the $i,j\in\{0,\ldots,n\}$. 

	Consider the following elements in $\h$: Define
	\begin{align}
	\label{eq:HC}
	H_C&=\diag(\frac{n}{2},\frac{n}{2}-1,\ldots,\frac{n}{2}-n) \\
	\label{eq:Hi}
	H_i&=\diag(n,-\frac{n}{i},\cdots,-\frac{n}{i},0,\cdots,0) \text{ for $1\leq i\leq n$.}
	\end{align}
The reason for defining $H_C$ as above is the following algebraic relation: For any $h>0$ and $(x_1,\ldots,x_n)\in\R^n$,
\begin{equation} \label{def:HC}
e^{H_C\log h}u(x_1,\ldots,x_n)e^{-H_C\log h}=u(h x_1,h^2 x_2,\ldots,h^n x_n).
\end{equation}

\subsubsection{Translating diagonals} \label{sec:ri} 
For all $t\in\cT$ and $1\leq i\leq n$, let $\xi_i(t)\geq 0$ be such that
  \begin{gather}
  \label{eq:convex_combination_of_H_i}
	\HA=\sum_{i=1}^{n}\xi_i H_i=\diag(nt,r_1(t),\ldots,r_n(t))\text{, where  $\bfxi:=(\xi_1(t),\ldots,\xi_n(t))$}.
	\end{gather}  
Then $\sum_{i=1}^n \xi_i=t$, where $\xi_i$'s are treated as functions of $t\in\cT$. We recall that by \eqref{def2:ri}, $r_1(t)\geq \cdots\geq r_n(t)\geq 0$ and  $nt=r_1(t)+\cdots+ r_n(t)$. So 
	\begin{equation} \label{eq:xi_i}
	    \xi_n(t)=r_n(t) \text{, and } \xi_i=(i/n)(r_i-r_{i+1}),\ 
\forall\,	n-1\geq i\geq 1.
	\end{equation}

By \eqref{def2:ri}, $a_t=\exp(\HA)$ for all $t\in\cT$. By \eqref{def:n0}, 
\begin{equation} \label{eq:n0xi}
\lim_{t\to\infty}\xi_i(t)=0,\,\forall n_0<i\leq n \text{, and } \lim_{t\to\infty} \xi_{n_0}(t)=\infty.
\end{equation}

We note that $\{a_t\}_t$ is non-uniform if and only if 
\begin{equation} \label{eq:non-uni}
\text{there exists } 1\leq j<n \text{ such that } \lim_{t\to\infty} \xi_j(t)=\infty. 
\end{equation}

\subsubsection*{Notation for unipotent elements} Let $\{\bfe_i\}_{1\leq i\leq n}$ be the standard basis of $\R^n$, where the $i$-th coordinate of $\bfe_i$ is 1 and all other coordinates are 0. For $\bfx\in\R^n\backslash \{0\}$, we write $\bfx=\sum_{i=1}^{n}x_i\bfe_i$. Let $u_i$ denote the unipotent element $u(x_i\bfe_i)$. Then $u(\bfx)=u_nu_{n-1}\cdots u_1$.

\subsubsection{Motivation for the main proposition} In order to prove our limiting distribution results using Dani-Margulis non-divergence, Ratner's classification of ergodic invariant measures, and the linearization techniques, we will need to show the following: Suppose that for all small $h:=e^{-t}\eta >0$, we have 
\[
\phi(s+h)-\phi(s)=R(h):=(h\kappa_1,\ldots,h^n\kappa_n)
\]
where $\kappa_i\neq 0$ for all $i$. Let $V$ be a finite dimensional representation of $G$. Our basic goal is to show that there exists $C>0$ such that for any $v\in V$, 
\[
\sup_{\eta\in (0,1)} \norm{a_tu(R(e^{-t}\eta))v}\geq C\norm{v}.
\]
By \eqref{def:HC} we have that for $\bfkappa:=(\kappa_1,\ldots,\kappa_n)$ and $\log h=-(t-\log\eta)$,
\[
u(R(h))=e^{-(t-\log\eta)H_C}u(\bfkappa)e^{(t-\log\eta)H_C}.
\]
With respect to the corresponding action of the Lie algebra of $G$ on $V$, let $B$ be the set of eigenvalues of $H_C$ on $V$. Then $v=\sum_{b\in B} v(b)$, where $H_Cv(b)=bv(b)$. Since $a_t=\exp(\HA)$, 
\begin{align*}
a_tu(R(h))v&=\sum_{b\in B}a_tu(R(h))v(b)\\
&=e^{\HA-(t-\log\eta)H_C}u(\bfkappa)e^{(t-\log\eta)H_C}v(b)
\\
&=\sum_{b\in B} e^{(t-\log\eta)b} e^{\HA-(t-\log\eta)H_C}u(\bfkappa)v(b).
\end{align*}
So for any $\mu\in \Lambda_{u(\bfkappa)v(b)}$, we have
\begin{align} \label{eq:atumu}
   [a_tu(R(h))v(b)]_\mu&=e^{[\mu(\HA)-t(\mu(H_C)-b)]}\eta^{\mu(H_C)-b}[u(\bfkappa)v(b)]_\mu.
\end{align}
Since $\HA=\sum_{k=1}^n \xi_k(t)H_k$ and $\sum_k \xi_k(t)=t$,
we have 
\[
\mu(\HA)-t(\mu(H_C)-b)=\sum_{k=1}^n  \xi_k(t)(\mu(H_k)-(\mu(H_C)-b)).
\]
Moreover $\xi_k(t)\geq 0$ for all $k$. In view of these observations, we aim to show that 
\[
\forall\, b\in B \text{ and } \forall\, 1\leq k\leq n,  \ \exists\, \mu\in\Lambda_{u(\bfkappa)v(b)} \text{ such that }  \mu(H_k)-(\mu(H_C)-b)\geq 0.
\]
As we are taking supremum over $\eta$ in an interval, coefficients $\eta^{\mu(H_C)-b}$ will prevent cancellations when we sum the terms \eqref{eq:atumu} over $b\in B$; we will see that $\mu(H_C)-b\in \N\cup\{0\}$. The formal result in the general will be proved later as \Cref{prop:main-basic}. 

The following result involves some of the main new ideas developed in this article. 

	\begin{proposition} \label{thm:basic_lemma}
		Let $\rho\colon G\to\GL(V)$ be a finite dimensional linear representation of $G$. Let $v\in V\setminus\{0\}$ such that $H_C\cdot v:=d\rho(H_C)v=bv$ for some $b\in\R$. 
		Fix $\bfx=(x_1,\cdots,x_n)\in(\R\setminus\{0\})^n$. For $1\leq k\leq n$, define
	\begin{align}
	    S_k=S_k(v)=S_k(\bfx,v)=\{\mu\in \Lambda_{u(\bfx)v}:\mu(H_k)-(\mu(H_C)-b)
	    \geq 0\}.  \label{eq:Sk}
	\end{align}
	Then the following statements hold:
	\begin{enumerate}
	    \item  \label{itm:Basic-1}
	    $\mu(H_C)-b\geq 0$ for all $\mu\in \Lambda_{u(\bfx)v}$.
	   
		\item \label{itm:r1rn}
		\begin{enumerate} 
		\item \label{itm:Sn} We have that $S_n\neq\emptyset$.
	    \item \label{itm:vGfixed}
	    Suppose that 
	    \[
	    \mu(H_n)-(\mu(H_C)-b)=0 \text{ and } \mu(H_C)-b=0,\,  \forall \mu\in S_n.
	    \]
		Then $v$ is $G$-fixed. 
	\end{enumerate}
	\item \label{itm:S} Let $S=S(v)=S(\bfx,v):=\cap_{k=1}^{n} S_k$. 
	Then $S\neq \emptyset$. 
	\item \label{itm:S0} Suppose $1\leq j< n$ and $j\leq n_0\leq n$ are such that
	\begin{equation} \label{eq:S0}
	\forall \mu\in S,\,  \mu(H_k)-(\mu(H_C)-b)=0	
	\text{ for } k=j,n_0.
	\end{equation}
	\begin{enumerate} 
	\item \label{itm:Qn0fixed} 
	Then $v$ is fixed by $\Po_{n_0}=\Po\cap G_{n_0}$. (see~\eqref{eq:Fmbyn-1}, \eqref{eq:Q}). 
\item \label{itm:Gn0fixed}
Moreover if $\mu(H_C)-b=0$ for all $\mu\in S$, then $v$ is fixed by $G_{n_0}$. 
	\end{enumerate}
	\end{enumerate}
	\end{proposition}

The following result motivated \Cref{thm:basic_lemma}, and it will be crucially used in its proof.
	
	\begin{lemma}[{\cite[Lemma~4.1]{Shah2020}}]\label{lem:SL_2_basic_lemma}
		Let $V$ be a finite dimensional representation of $\mathrm{SL}(2,\mathbb{R})$. Let
		$$
		A=
		{ \begin{bsmallmatrix}
		1 & \\
		& -1
		\end{bsmallmatrix}}\in \mathfrak{sl}(2,\R),\text{ and }
	    u(r)=
		{ \begin{bsmallmatrix}
		1 & r\\
		0 & 1
		\end{bsmallmatrix}},\ 
		\forall\,r\in\R.
		$$
		Express $V$ as the direct sum of eigen-spaces with respect to the action of $A$:
		$$V=\bigoplus_{\lambda\in\mathbb{R}}V_\lambda(A) \text{, where } V_\lambda(A):=\{v\in V:A v=\lambda v\}. $$
		For any $v\in V\backslash\{0\}$ and $\lambda\in\mathbb{R}$,
		let $v_\lambda=v_\lambda(A)$ denote the $V_\lambda(A)$-component of $v$, and define
		\begin{equation} \label{eq:lambda-max}
		  \lambda^{\max}(v)=\max\{\lambda:v_\lambda\neq 0\} \text{, and } v_{\max}=v_{\lambda^{\max}(v)}.  
		\end{equation}
		Let $r\neq 0$. Then the following statements hold:
	\begin{enumerate} 
	\item \label{itm:sl2-basic} $\lambda^{\max}(u(r)v)+\lambda^{\max}(v)\geq 0$.
	
	\item \label{itm:sl2-equality}
	There is equality in (\ref{itm:sl2-basic}),  if and only if
	\begin{equation}
	    \label{eq:SL2-equal}
	    v=u^{-}(r)\cdot v_{\max} \text{, and } (u(r)v)_{\max}=\sigma_1(r)\cdot v_{\max},
	\end{equation}
	where 
	\begin{equation} \label{eq:sigma1}
	u^{-}(r)={ \begin{bsmallmatrix} 1 & 0\\ r&1\end{bsmallmatrix}}  \text {, and } \sigma_1(r)={ \begin{bsmallmatrix} 0 & r \\ -r^{-1} & 0 \end{bsmallmatrix}}.
	\end{equation}
\item \label{itm:sl2-eigen}
Suppose that $v$ is an eigenvector of $A$. 
\begin{enumerate} 
\item  \label{itm:sl2-eigen+} There is equality in (\ref{itm:sl2-basic}), if and only if $v$ is fixed by $u^-(\R)$, if and only if  $(u(r)v)_{\max}=\sigma_1(r)v$.
\item \label{itm:sl2-eigen=}  $\lambda^{\max}(u(r)v)=\lambda^{\max}(v)$, if and only if $v$ is fixed by $u(\R)$.
\item \label{itm:sl2-eigen0}
$\lambda^{\max}(u(r)v)=\lambda^{\max}(v)=0$ if and only if, $v$ fixed by $\SL(2, \R)$.  
\end{enumerate}
\end{enumerate}
\end{lemma}

\subsection{Proof of Proposition~\ref{thm:basic_lemma}(\ref{itm:Basic-1})}

Let $\mu\in\Lambda_{u(\bfx)v}$. Then $\mu=\lambda+\sum_{i=1}^{n}m_i\beta_i$ for some $\lambda\in\Lambda(v)$ and $m_i\in\Z_{\geq 0}$. Since $\lambda(H_C)=b$ and $\beta_i(H_C)>0$, we have $\mu(H_C)\geq b$. 

\subsection{Proof of Proposition~\ref{thm:basic_lemma}(\ref{itm:r1rn})} The proof given here was motivated by \cite[Proposition~A.0.1]{Pengyu-thesis}.  First we suppose that $\bfx=(1,\dots,1)$. The case of general $\bfx$ will be deduced from this case in \Cref{sec:bfx}.  
 \label{sec:x=111} 
\subsubsection*{Key observations} 
Let
	\begin{equation} \label{eq:defu-}
	u^-:=
	{
	\begin{bsmallmatrix}
	1 \\
	1      & 1 \\
	\vdots & \vdots   &\ddots &\\
	1      &  1   & \dots & 1  \\
	\end{bsmallmatrix}
	} \text{ and }
	\sigma =
	\begin{bsmallmatrix} & 1 \\ -I_n \end{bsmallmatrix}.
	\end{equation}
Let $u^+$ denote the transpose of $u^-$. 
The first key observation is: 
	\begin{align} 
	u^+\sigma^{-1}u(\bfx) &=
	{\begin{bsmallmatrix}
	1 & 1 & \dots & 1  \\
	  & 1 & \dots & 1  \\
	  &   & \ddots & \vdots  \\
	  &   &        &        1  \\
	\end{bsmallmatrix}
	\begin{bsmallmatrix}
	& -1\\
	& &\ddots \\
	&&&-1\\
	1&
	\end{bsmallmatrix}
	}u(\bfx)
	\nonumber \\
	&= {
	\begin{bsmallmatrix}
	     1  &-1     &-1      &\ldots &-1\\
	     1  & 0      &-1     &\dots &-1\\
	  \vdots& \vdots &       &\ddots &\vdots \\
	     1  &  0     &       &\ldots &-1\\
	  1     &0       & \ldots &      & 0
	\end{bsmallmatrix}
	\begin{bsmallmatrix}
	     1 & 1 & 1      &\ldots & 1\\
	       & 1 & 0       & \ldots   &  0    \\
	       &   & \ddots &    & \vdots\\
	       &   &        & 1  &  0\\
	       &   &        &    & 1
	\end{bsmallmatrix}
	}
     \nonumber\\
	&=
	{
	\begin{bsmallmatrix}
	1 \\
	1      & 1 \\
	\vdots & \vdots   &\ddots &\\
	1      &  1   & \dots & 1  \\
	\end{bsmallmatrix} }
	= u^-. 
	\end{align}
	In other words,
	\begin{equation}
	    \label{eq:u+-sigmaux}
	    u(\bfx)=\sigma\cdot (u^+)\inv u^-.
	\end{equation}
	And the second key observation is
		\begin{gather}
		    	 H_n-H_C = -\sigma H_C \sigma^{-1}.
		    	 \label{eq:HnHc}
    \end{gather}

\subsubsection{Min-Max notation}
For any $H\in \h$ and $w\in V$, we define 
		\begin{equation} \label{eq:Hminw}
		    \begin{array}{rl}
		\Hmin(w)&:=\min\{\lambda(H):\lambda\in\Lambda_w\},\\
		\Lambda_w^{\Hmin}&:=\{\lambda\in \Lambda_w: \lambda(H)=\Hmin(w)\} \text{, and}	\\
		w_{\Hmin}&:=w_{\Lambda_w^{\Hmin(w)}}.
		\end{array}
        \end{equation}
We also define $\Hmax(w)$, $\Lambda^{\Hmax}_w$, and $w_{\Hmax}$ in a similar manner. 
Then
\begin{equation} 
\label{eq:w-max-min}
    \begin{array}{rl}
    (-H)\text{-}\max(w)&=-[\Hmin(w)], \\
     \Lambda^{(-H)\text{-}\max}_w&=\Lambda^{\Hmin}_w
    \text{, and }\\
w_{(-H)\text{-}\max}&=w_{\Hmin}. 
\end{array}
\end{equation}
We have $\Hmin(w)\leq \Hmax(w)$, and the equality holds if and only if $w$ is an eigen-vector of $H$. By \eqref{eq:Weylonh} and \eqref{eq:WeylonV} for any $\gamma\in N_G(\h)$,  
\begin{equation} \label{eq:g-min}
    \begin{array}{rl}
(\gamma H\gamma\inv)\text{-}\min(\gamma w)&=\Hmin(w),\\
\Lambda^{(\gamma H\gamma\inv)\text{-}\min}_{\gamma w}&=\gamma \Lambda^{\Hmin}_{w}, \text{ and } \\ 
(\gamma w)_{\gamma H\gamma\inv\text{-}\min}&=\gamma (w)_{\Hmin}.
\end{array}
\end{equation}

\begin{claim} \label{claim:HC-min}
Let $w\in V$.  Let $g^+$ (resp.\ $g^-$) be an upper (resp.\ lower) triangular unipotent matrix in $G=\SL(n+1,\R)$. then
\begin{gather}
    \label{eq:u+w}
   \Hcmin(g^+w)=\Hcmin(w), \text{ and } \Lambda_{g^+w}^{\Hcmin}=\Lambda_{w}^{\Hcmin};\\
   \label{eq:u-w}
   \Hcmax(g^-w)=\Hcmax(w), \text{ and } \Lambda_{g^-w}^{\Hcmax}=\Lambda_{w}^{\Hcmax}. 
\end{gather}
Moreover,
\begin{equation}
    \label{eq:u+w2}
    (g^+w)_{\Hcmin}=w_{\Hcmin} \text{ and } (g^-w)_{\Hcmax}=w_{\Hcmax}.
\end{equation}
\end{claim}
\begin{proof}
Let $E_{i,j}\in\mathfrak{sl}(n,\R)$ denote the matrix whose $(i,j)$-th entry is $1$ and other entries are $0$. Let $\lambda\in \Lambda_w$. Since $[H_C,E_{i,j}]=(j-i)E_{i,j}$, we have
\begin{equation} \label{eq:Eij}
d\rho(H_C)d\rho(E_{i,j})w_\lambda=(\lambda(H_C)+(j-i))d\rho(E_{i,j})w_\lambda.
\end{equation}
Therefore if $d\rho(E_{i,j})w\neq 0$, then
\begin{align*}
    \Hcmin(d\rho(E_{i,j}) w)>\Hcmin(w) & \text{, if $i<j$, and} \\
    \Hcmax(d\rho(E_{i,j}) w)<\Hcmax(w) &\text{, if $i>j$.}
\end{align*}

Let $X^+$ (resp.\ $X^-$) be a strictly upper (resp.\ lower) triangular nilpotent matrix in $\mathfrak{sl}(n+1,\R)$ such that $g^\pm=\exp(X^\pm)$.
Then for any $k\geq 0$:
\begin{align}
d\rho(X^+)^{k+1}w\neq 0 &\Rightarrow \Hcmin(d\rho(X^+)^{k+1}w)>\Hcmin(d\rho(X^+)^kw); \nonumber\\
d\rho(X^-)^{k+1}w\neq 0&\Rightarrow  \Hcmax(d\rho(X^-)^{k+1}w)<\Hcmax(d\rho(X^-)^k w).  \label{eq:drhoX-w}
\end{align}
Since $\rho(g^{\pm})=I_V+\sum_{k=1}^{\dim V} d\rho(X^{\pm})^k/k!$,
\eqref{eq:u+w}, \eqref{eq:u-w}, and \eqref{eq:u+w2} follow.
\end{proof}

\begin{claim}
\label{claim:g-w}
Let $X^-\in \mathfrak{sl}(n+1,\R)$ be a strict lower triangular nilpotent matrix and $g^-=\exp(X^-)\in G$.  
Suppose that $w$ and $g^-w$ are eigen-vectors of $H_C$. Then $\rho(X^-)w=0$, and hence $g^-w=w$.  
\end{claim}

\begin{proof}  
Suppose $\rho(X^-)w\neq 0$. Then by \eqref{eq:drhoX-w} the following hold:
\begin{equation*}
\Hcmax(d\rho(X^-)w)<\Hcmax(w)=\Hcmin(w),
\end{equation*}
because $w$ is an eigen-vector of $H_C$, and 
for any $k\geq 2$, if $d\rho(X^-)^kw\neq 0$, then 
\[
\Hcmax(d\rho(X^-)^kw)<\Hcmax(d\rho(X^-)w).
\]
Therefore
\begin{align*}
\Hcmin(g^-w)=\Hcmin(\exp(d\rho(X^-))w)<\Hcmin(w).
\end{align*}
Since $\Hcmax(g^-w)=\Hcmin(g^-w)$, we get $\Hcmax(g^-w)<\Hcmax(w)$, which contradicts \eqref{eq:u-w}.
\end{proof}

\subsubsection{Proof of \Cref{thm:basic_lemma}(\ref{itm:Sn}) for $\bfx=(1.,\ldots,1)$} \label{subsec:MainCalculation}
        \begin{align}
          (H_n-H_C)\text{-}\max(u(\bfx)v) 
           &= (-\sigma H_C \sigma\inv)\text{-}\max(u(\bfx)v) \quad \text{ by \eqref{eq:HnHc}} \label{eq:a0}\\
           &=-[(\sigma H_C \sigma\inv)\text{-}\min(u(\bfx)v)] 
           \quad \text{ by \eqref{eq:w-max-min}} \notag\\
           &=-[(\sigma H_C \sigma\inv)\text{-}\min(\sigma (u^+)\inv u^-v) \quad \text{ by \eqref{eq:u+-sigmaux}} \notag\\
           &=-[\Hcmin((u^+)\inv u^-v)] \quad \text{ by \eqref{eq:g-min}}\notag\\
           &=-[\Hcmin(u^-v)] \quad \text{ by \eqref{eq:u+w}}\notag\\
           &\geq  -[\Hcmax(u^-v)] \label{eq:ineqality1}\\
           &=-[\Hcmax(v)] \quad \text{ by \eqref{eq:u-w}} \notag\\
           &=-b,\quad \text{ because $H_C\cdot v=bv$}. \notag
        \end{align}     
Therefore 
\begin{equation} \label{eq:Sn-max} S_n\supset\Lambda^{(H_n-H_C)\text{-}\max}_{u(\bfx)v}\neq\emptyset. 
\end{equation}
\qed
\subsubsection{Proof of Proposition~\ref{thm:basic_lemma}(\ref{itm:vGfixed}) for $\bfx=(1,\ldots,1)$}
\label{subsec:a0c0}	
	    First suppose that
	    \begin{equation} \label{eq:Sn01}
	    \mu(H_n)-(\mu(H_C)-b)=0,\, \forall \mu\in S_n.
	    \end{equation}
	    Therefore by \eqref{eq:Sn-max}, we have $S_n=\Lambda^{(H_n-H_C)\text{-}\max}_{u(\bfx)v}$, and we have the equality in \eqref{eq:ineqality1}. So $u^-v$ is an eigenvector of $H_C$. Let $X^-=\log(u^-)$. Then by Claim~\ref{claim:g-w}, 
	    \begin{equation} \label{eq:vBfixed}
	    d\rho(X^-)v=0 \text{, and hence } u^-\cdot v=v.
	    \end{equation}

By the calculation as in  \Cref{subsec:MainCalculation},
\begin{align*}
    [u(\bfx)v]_{S_n}
    &=[u(\bfx)v]_{(H_n-H_C)\text{-}\max} \quad
    \text{by \eqref{eq:Hminw}} \notag\\
    &=[u(\bfx)v]_{(-\sigma H_C \sigma\inv)\text{-}\max} \quad \text{ by \eqref{eq:HnHc}} \notag \\
    &=[u(\bfx)v]_{(\sigma H_C \sigma\inv)\text{-}\min} \quad \text{ by \eqref{eq:w-max-min}}\\
    &=[\sigma (u^+)\inv u^-v]_{(\sigma H_C \sigma\inv)\text{-}\min} \quad \text{ by \eqref{eq:u+-sigmaux}} \notag\\
    &=\sigma [(u^+)\inv u^-v]_{\Hcmin} \quad \text{ by \eqref{eq:g-min}}\\
    &=\sigma [u^-v]_{\Hcmin}\quad \text{ by \eqref{eq:u+w2}}\notag\\
    &=\sigma v_{\Hcmin} \quad \text{ because $u^-v=v$.}\\
    &=\sigma v, \quad \text{ because $v$ is an eigenvector of $H_C$}. 
\end{align*}

As a consequence, we get 
\begin{equation} \label{eq:vS}
S_n=\Lambda_{\sigma v}=\sigma \Lambda_v.
\end{equation}

We further suppose that $\mu(H_C)-b=0$ for all $\mu\in S_n$. For $\lambda\in\Lambda_v$, $\sigma\lambda\in S_n$, so,  by \eqref{eq:Weylonh},
\[
b=(\sigma \lambda)(H_C)=\lambda(\sigma^{-1} H_C \sigma).
\]
That is, $\sigma^{-1} H_C \sigma \cdot v = bv$. By \eqref{eq:HnHc}, we have 
\[
\sigma^{-1} H_C\sigma-H_C=\sigma\inv(H_C-\sigma H_C\sigma\inv)\sigma=\sigma\inv H_n\sigma,
\]
and
\[
\tilde H_n:=\sigma\inv H_n\sigma=\diag(1,\ldots,1,-n)\in\la{h}.
\]
Therefore, since $H_C v= bv$, we get 
$\tilde H_n\cdot v=0$. So by \eqref{eq:vBfixed}
\[
\tilde H_n,\, X^-\in \la{g}_v:=\{X\in \la{g}:d\rho(X)v=0\}.
\]

\begin{claim} 
\label{claim:epimorphic}
The Lie algebra $\la{f}$ generated by $\tilde H_n$ and $X^-$ contains the span of strictly negative eigenspaces of $\ad(\tilde H_n)$ in $\la{g}$.
\end{claim}

By \cite[Lemma 5.2]{Shah1996a}, for any finite dimensional Lie-algebra representation of $\la{g}$, if a vector is annihilated by $\tilde H_n$ and the span of strictly negative eigenspaces of $\ad(\tilde H_n)$ in $\la{g}$, then it is annihilated by $\la{g}$. Therefore by \Cref{claim:epimorphic}, $v$ is fixed by $G$, which is the goal of \Cref{subsec:a0c0}. 

In other words, we will prove that the subalgebra generated by $\tilde H_n$ and $X^-$ is {\em epimorphic\/} in $\la{g}$, (see~\cite{Mozes1995EpimorphicMeasures}).

\begin{proof}[Proof of \Cref{claim:epimorphic}]
By \eqref{eq:defu-}, one verifies that $(u^-)\inv =I-Z$, where
\[
Z={
	\begin{bsmallmatrix}
	0 \\
	1 & 0 \\
	   &  \ddots &\ddots  \\
	   &   & 1 & 0  \\
	\end{bsmallmatrix} } .
\]
Therefore 
\[
X^-=-\log((u^-)\inv)=-\log(1-Z)=\sum_{k=1}^n  Z^k/k=
{
	\begin{bsmallmatrix}
	0 \\
	1           & 0 \\
	\frac{1}{2}& 1       &   0\\
	\vdots      &  \ddots & \ddots & \ddots \\
   \frac{1}{n} &  \hdots & \frac{1}{2}& 1 & 0  \\
	\end{bsmallmatrix} }.
\]
as $Z^k$ consists of $1$'s in all the $(k+i,i)$-th entries for $1\leq k\leq n$ and $0\leq i\leq n-k$, and the $0$'s elsewhere.

Then 
\[
\la{f} \ni [\tilde H_n,X^-]=-(n+1)Y_1 \text{, where }Y_1:=
{
	\begin{bsmallmatrix}
	0 \\
	\vdots      &        & \ddots    \\
	0          &         &   \hdots    &   0            \\
   \frac{1}{n} &  \hdots  & \frac{1}{2}& 1    & 0  \\
	\end{bsmallmatrix} }.
\]
For $2\leq k\leq n$, inductively define 
\[
Y_{k}=[Y_{k-1},X^{-}]\in \la{f}\text{, so } 
Y_k={
	\begin{bsmallmatrix}
	0 \\
     \vdots                 &  \ddots       & 
     \\
     0          &  \hdots&0                 \\
   \ast \hdots \ast    & 1  &       \underbrace{0 \hdots    0}_\text{$k$-terms}&  \\
	\end{bsmallmatrix} },
\]
which is inductively verified for all $2\leq k\leq n$ as follows: 
(1) $X^{-}Y_{k-1}=0$, so $Y_{k}=Y_{k-1}X^{-}$; (2) all rows except the last row of $Y_k$ are $0$; (3) the last $k$ columns of $Y_k$ are $0$, and (4) the $(n,n-k)$-th entry of $Y_k$ equals $1$.

Therefore $\la{f}$ contains $\R\text{-}\Span(\{Y_k:1\leq k\leq n\})$, which is the span of all strictly negative eigenspaces of $\ad(\tilde H_n)$ in $\la{g}$.
\end{proof}

\subsubsection{Completion of Proof of \Cref{thm:basic_lemma}(\ref{itm:r1rn})} \label{sec:bfx}
Now suppose $\bfx=(x_1,\ldots,x_n)$ with $x_i\neq 0$ for all $i$. Let 
\[
a_{\bfx}=\diag(1,x_1\inv,\ldots,x_n\inv)\in \GL(n,\R).
\]
Let $z_0\in\Cx$ be such that $z_0^{n+1}=\prod_{i=1}^n x_i$. Let
\[
d_{\bfx}=\diag(z_0,z_0x_1^{-1},\ldots,z_0x_n^{-1})=z_0a_{\bfx}\in \SL_{n+1}(\Cx).
\]
Then
\begin{equation} \label{eq:dxax}
    d(\bfx)gd(\bfx)^{-1}=a_{\bfx} g a_{\bfx}^{-1},\, \forall g\in \GL(n,\Cx).
\end{equation}
In particular, $d_{\bfx}$ normalizes $G$,
\begin{gather*}
\Ad(d_\bfx)(H)=H,\,\forall H\in \la{h} 
 \text{, and } d_{\bfx}u((1,\ldots,1))d_{\bfx}^{-1}=u(\bfx).
\end{gather*}

The representation of $G=\SL(n,\R)$ on $V$ canonically extends to a representation of $\SL(n,\Cx)$ on $V_\Cx:=V\otimes \Cx$. We treat $V_\Cx$ as a representation of $G$ over $\R$. 

Let $v_\bfx=d(\bfx)^{-1}v\in V_\Cx$. Then
\begin{gather*}
u(\bfx)v=d_{\bfx}\cdot u(1,\ldots,1)v_\bfx, \,
\Lambda_{v}=\Lambda_{v_{\bfx}},\, \Lambda_{u(\bfx)v}=\Lambda_{u(1,\ldots,1)v_\bfx},
\end{gather*}
and hence $S_n(\bfx,v)=S_n(u((1,\ldots,1)),v_\bfx)$.
Now it is straightforward to deduce   \Cref{thm:basic_lemma}(\ref{itm:r1rn}) for the actions of $u(\bfx)$, from the special case of the actions of $u((1,\dots,1))$ and the vector $v_{\bfx}\in V_\Cx$ for the representation of $G$ on $V_\Cx$.
\qed

\subsection{Proof of Proposition~\ref{thm:basic_lemma}(\ref{itm:S},\ref{itm:S0})}  Pick any $1\leq j\leq n$. We will use Theorem~\ref{thm:basic_lemma}(\ref{itm:r1rn}) for the case of $n=j$, and recursively apply \Cref{lem:SL_2_basic_lemma} $(n-j)$ times. 

\subsubsection{Notation} For $1\leq i\leq n$, let $\beta_i\in\h^*$ be defined by 
\begin{equation} \label{eq:betaiH}
    \exp(H)u(\bfe_i)\exp(-H)=u(e^{\beta_i(H)}\bfe_i),\,\forall H\in\h;
\end{equation}
in other words,
	\begin{equation} \label{eq:beta_i}
	\beta_i(\diag(a_0,\cdots,a_{n}))=a_0-a_{i},
	\end{equation}
	Then $\{\beta_i:1\leq i\leq n\}$ is a basis of $\h^\ast$.
By \eqref{eq:HC} and \eqref{eq:Hi},
	\begin{equation}
	\label{eq:beta-HC}
	    \beta_i(H_C)=i \text{ and } 
	    \beta_i(H_k)=\begin{cases} 
	                     n+k/n & \text{ if $i\leq k$}\\
	                       n   & \text{ if $i>k$}. 
	                 \end{cases}
	\end{equation}

A crucial concept used in our proof is the  
\begin{definition}[Lexicographical total order on $\h^\ast$] \label{def:lexico}
    For any $\lambda,\mu\in \h^\ast$, we say that $\lambda>\mu$, 
	if there exists some $1\leq k\leq n$ such that $\lambda-\mu=\sum_{i=1}^k m_i\beta_i$ and 
	$m_k>0$. This is a total order on $\la{h}^\ast$. 
\end{definition}

For $1\leq i\leq n$, let $A_i=\diag(1,-\delta_1,\ldots,-\delta_n)\in\h$, where $\delta_{i}=1$ and $\delta_{j}=0$ if $i\neq j$. 
Let $L_i\cong \SL(2,\R)$ be the subgroup of $G$ which is generated by $u(\R\bfe_i)$, $\exp(\R A_i)$, and $\trn{u(\R \bfe_i)}$, where $\trn{X}$ denotes the transpose of $X$.
	
	Let $1\leq j\leq n$ be given. let  
	\begin{equation} \label{def:Gj}
	    G^j=\left\{{{\begin{bsmallmatrix} M & \\ & I_{n-j} \end{bsmallmatrix}}}\in \SL(n+1,\R): M\in \SL(j+1,\R)\right\},
	\end{equation}
where $I_m$ denotes the $m\times m$-identity matrix. In view of \eqref{eq:HC} and \eqref{eq:Hi}, define
    \begin{align}\label{eq:Hjj}
         H^j_j&=A_1+\cdots+A_j=(j/n)H_j\in \Lie(G^j)  \quad \text{ and }\\
         \label{eq:HCj}
    H_C^j&=\diag(j/2,j/2-1,\ldots,j/2-j,0,\ldots,0)\in \Lie(G^j).
    \end{align}

\subsubsection{First step - action of $G^j$} \label{sec:step1}
Let $\lambda_0$ be the unique maximal element in $\Lambda_v$ in the lexicographical order on $\h^\ast$. In particular, $\lambda_0(H_C)=b$.

	  Consider the $G^j$ submodule
	\[
	W_j=\bigoplus_{(m_1,\ldots,m_j)\in \Z^j} V_{\lambda_0+\sum_{i=1}^{j}m_i\beta_i}.
	\]
	Let $w_0=\Proj_{W_j}(v)$, where the projection $\Proj_{W_j}:V\to W_j$ is $G^j$-equivariant.  

    We want to apply Theorem~\ref{thm:basic_lemma}(\ref{itm:r1rn}) for the action of $G^j\cong \SL(j+1,\R)$ on $W_j$ for $j$ in place of $n$, where $H_j^j$ plays the role of $H_n$, and $H_C^j$ plays the role of $H_C$. Note that \begin{equation} \label{eq:HCjHC}
        \beta_i(H_C^j)=i=\beta_i(H_C), \quad\forall 1\leq i\leq j.
    \end{equation}
    For any $w\in W_j$ and $\lambda\in \Lambda_w$, we have $\lambda-\lambda_0\in \langle\beta_1,\ldots,\beta_j\rangle_{\Z}$, so
    \begin{equation}\label{eq:w0-eigen}
        (\lambda-\lambda_0)(H_C^j)=(\lambda-\lambda_0)(H_C)=\lambda(H_C)-b. 
    \end{equation}
    Since $v$ is an eigen-vector of $H_C$, we have that $w_0$ is an eigenvector of $H^j_C$ with eigenvalue $b_j:=\lambda_0(H^j_C)$. 
    
    We write $u_i=u(x_i\bfe_i)$, where $x_i\neq 0$, for $1\leq i\leq n$. So $u(\bfx)=u_n\cdots u_1$. Let 
    \begin{align}
        \label{eq:defSj}
        S_j^j=S_j^j(w_0):=\{\lambda\in \Lambda_{(u_j\ldots u_1)w_0}:\lambda(H^j_j-H^j_C)-\lambda_0(H^j_C)\geq 0\}.
    \end{align}
    By Theorem~\ref{thm:basic_lemma}(\ref{itm:Basic-1},\ref{itm:Sn}), $S_j^j\neq\emptyset$, and given any $\lambda_j\in S_j^j$, 
      \begin{align} 
          0\leq c_j&:=\lambda_j(H_C^j)-b_j=(\lambda_j-\lambda_0)(H_C^j) \notag\\
          &=(\lambda_j-\lambda_0)(H_C)=\lambda_j(H_C)-b \text{, by \eqref{eq:w0-eigen}} \label{eq:bj-2}\\
      \label{eq:bj-1}
        0\leq a_j&:=\lambda_j(H^j_j)-(\lambda_j(H^j_C)-b_j)=\lambda_j(H_j^j)-c_j.
      \end{align}
      
\subsubsection{Successive actions of $L_i\cong\SL(2,\R)$ for $i=j+1,\ldots,n$} \label{sec:NextStep}

	 Let $v_0=v$ and $v_i=u_{i}u_{i-1}\cdots u_1v$ for each $1\leq i\leq n$. Then for each $i$, $\Lambda_{v_i}\subset \Lambda_{v_{i-1}}+\Z \beta_i$, and for any $\lambda\in \Lambda_{v_{i-1}}$, there exists $m\in\Z$ such that $\lambda+m\beta_i\in\Lambda_{v_i}$. We have picked $\lambda_j\in S_j^j(w_0)\subset \Lambda_{(u_j\cdots u_1)w_0}$. Since $\Proj_{W_j}$ is $G^j$-equivariant, $u_j\ldots u_1w_0=\Proj_{W_j}(v_j)$, and since $\Proj_{W_j}$ is $\la{h}$-equivariant, $\lambda_j\in \Lambda_{v_j}$. For $j+1\leq i\leq n$, we inductively define $\lambda_i\in\Lambda_{v_i}$ as follows: Let
	\begin{align}  
	    m_i:=\max\{m\in\Z\colon\lambda_{i-1}+m\beta_i\in\Lambda_{v_i}\}  
	    \text{ and }
	    \lambda_i:=\lambda_{i-1}+m_i\beta_i. \label{eq:lambda_i}
	\end{align}
   We have  $u(\bfx)=u_nu_{n-1}\cdots u_1$ and $\lambda_n\in \Lambda_{u(\bfx)v}$. 
   \subsubsection{Proof of Theorem~\ref{thm:basic_lemma}(\ref{itm:S})} We will show that $\lambda_n\in S_k$ for each $j\leq k\leq n$. 
   
		Let $j+1\leq i\leq n$. Consider the following $L_i\cong \SL(2,\R)$-submodule of $V$:
		\begin{equation} \label{eq:Vi}
		V_i=\bigoplus_{k\in\Z}V_{\lambda_{i-1}+k\beta_i}.
		\end{equation}
		Let $\Proj_{V_i}:V\to V_i$ denote the $\la{h}$-equivariant projection. Let
		\begin{equation} \label{eq:wi-1}
		    w_{i-1}=\Proj_{V_i}(v_{i-1}).
		\end{equation} 
		
		We want to apply \Cref{lem:SL_2_basic_lemma} to $L_i$, $V_i$, $w_{i-1}$, $A_i$ and $u(\bfe_i)$. Now 
		\begin{align}
		    \lambda^{\max}(w_{i-1})
		    &=\max\{\mu(A_i):\mu\in \Lambda_{w_{i-1}}\} \nonumber\\
		    &=\max\{\lambda_{i-1}(A_i)+k\beta_i(A_i):k\in\Z,\, \lambda_{i-1}+k\beta_i\in \Lambda_{v_{i-1}}\}\nonumber\\
		    &=\lambda_{i-1}(A_i)+2\cdot \max\{k\in\Z:\lambda_{i-1}+k\beta_i\in \Lambda_{v_{i-1}}\}.
		    \label{eq:wmax}
		\end{align}

\begin{claim} \label{claim:kleq0}
 Suppose $k\in\Z$ is such that $\lambda_{i-1}+k\beta_i\in \Lambda_{v_{i-1}}$. Then $k\leq 0$. 
\end{claim}
\begin{proof} Note that
\[
\Lambda_{v_{i-1}}=\Lambda_{u_{i-1}\cdots u_1v}\subset \Lambda_v+\zspn{\beta_1,\ldots,\beta_{i-1}}.
\]
So we pick $\lambda\in \Lambda_{v}$ such that 
\[
\lambda_{i-1}+k\beta_i\in \lambda+\zspn{\beta_1,\ldots,\beta_{i-1}}.
\]
By definition, see \eqref{eq:lambda_i}, $\lambda_{i-1}= \lambda_0+\sum_{\ell=1}^{i-1}m_\ell\beta_\ell$. Therefore
\[
\lambda_0-\lambda\in (-k)\beta_i+\zspn{\beta_1,\ldots,\beta_{i-1}}.
\]
Since $\lambda_0\geq \lambda$ with respect to the lexicographic order on $\h^\ast$, we get $-k\geq 0$.
\end{proof}

By definition, $\lambda_{i-1}\in \Lambda_{v_{i-1}}$. So by \eqref{eq:wmax} and Claim~\ref{claim:kleq0}, 
\begin{equation} \label{eq:w_imax}
    \lambda^{\max}(w_{i-1})=\lambda_{i-1}(A_i).
\end{equation}

We note that $\Proj_{V_i}:V\to V_i$ is an $L_i$-equivariant map. Therefore 
\begin{equation} \label{eq:ProjVi}
    u_iw_{i-1}=u_i\Proj_{V_i}(v_{i-1})=\Proj_{V_i}(u_iv_{i-1})=\Proj_{V_i}(v_i). 
\end{equation}
Therefore by the definition of $\lambda_i$ in \eqref{eq:lambda_i},  since $\beta_i(A_i)=2$, we get
\begin{equation} \label{eq:lambdaiAi}
  \lambda^{\max}(u_iw_{i-1})=\lambda_i(A_i).
\end{equation} 

Hence by \Cref{lem:SL_2_basic_lemma} 
		\begin{align}\label{eq:temp1-sum}
		0\leq \lambda^{\max}(u_iw_{i-1})+\lambda^{\max}(w_{i-1})
		=	\lambda_i(A_i)+\lambda_{i-1}(A_i)
		=2\lambda_i(A_i)-2m_i,
		\end{align}
		because $\lambda_{i-1}=\lambda_i-m_i\beta_i$ and $\beta_i(A_i)=2$. So we define
		\begin{equation}\label{eq:first_inequality}
			\delta_i= \lambda_i(A_i) - m_i\geq 0.
		\end{equation}
	
	Now let $j< i\leq n$. 
		Since $\lambda_i\in \Lambda_{v_i}=
		\Lambda_{u_i\cdots u_1v}$, by \eqref{eq:u+w} of \Cref{claim:HC-min},
		\begin{equation}\label{eq:defci}
		c_i\bydef\lambda_i(H_C)-\lambda_0(H_C)\geq 0, 
			\quad\forall j<i\leq n. 
		\end{equation}
	Since $\beta_i(H_C)=i$, 
		\begin{equation} \label{eq:mi-ci}
		c_i-c_{i-1}=(\lambda_i-\lambda_{i-1})(H_C)=m_i\beta_i(H_C)=im_i.
		\end{equation}
  
 For $j< i\leq n$, we have
 \begin{align}
     a_i
     &\bydef \lambda_i(H^i_i)- c_i \label{eq:ak}\\
     &=\lambda_i(H_{i-1}^{i-1})+\lambda_i(A_i) - c_{i-1} - (c_i-c_{i-1})  \text{, \tiny{by \eqref{eq:Hjj}}}\notag\\
     &=(\lambda_{i-1}(H_{i-1}^{i-1})-c_{i-1})+(m_i\beta_i(H_{i-1}^{i-1})+\lambda_i(A_i))-im_i 
     \text{, \tiny {by \eqref{eq:lambda_i}, \eqref{eq:mi-ci}}} \notag\\
     &=a_{i-1}+m_i(i-1)+\lambda_i(A_i)-im_i 
     \text{, \tiny{by \eqref{eq:bj-1}, \eqref{eq:ak}, and as $\beta_i(H_{i-1}^{i-1})=i-1$ by \eqref{eq:beta_i}}} \notag\\
     &=a_{i-1}+\delta_i 
     \text{, \tiny{by \eqref{eq:first_inequality}}} \label{eq:ak-1}\\
     &=a_j+\sum_{k=j+1}^{i} \delta_k 
     \text{, \tiny{by \eqref{eq:ak-1}} \label{eq:ak+1}}\\
     &\geq 0 
     \text{, \tiny{by \eqref{eq:bj-1} and \eqref{eq:first_inequality}}}. \label{eq:ak+}
 \end{align}

 Let $j\leq k\leq n$. Then
 \begin{align}
     d_k&\bydef\lambda_n(H_k-H_C)+\lambda_0(H_C) \label{def:dk}\\
     &=\sum_{i=k+1}^n (\lambda_i-\lambda_{i-1})(H_k-H_C)+\lambda_k(H_k-H_C)+\lambda_0(H_C) \notag\\
     &=\sum_{i=k+1}^n m_i\beta_i(H_k-H_C)+\lambda_k\bigl(\frac{n}{k}H_k^k\bigr)-(\lambda_k-\lambda_0)(H_C) \text{, {\tiny by \eqref{eq:lambda_i} and  $H_k=\frac{n}{k}H_k^k$}}\notag\\
     &=\sum_{i=k+1}^n \frac{c_i-c_{i-1}}{i}(n-i)+\frac{n}{k}(a_k+c_k)-c_k, \text{\tiny{as $m_i=\frac{c_i-c_{i-1}}{i}$ by \eqref{eq:mi-ci},}} \notag\\
     & \quad \text{\tiny{$\beta_i(H_n)=n$ by \eqref{eq:beta-HC}, $\beta_i(H_C)=i$, $\lambda_k(H_k^k)=a_k+c_k$ by \eqref{eq:bj-1} and \eqref{eq:ak}, and $c_k$ as in \eqref{eq:defci}}}\notag\\
     &=\frac{n}{k}a_k+\sum_{i=k}^{n-1} \bigl(\frac{n}{i}-\frac{n}{i+1}\bigr)c_i \geq 0
     \text{, \tiny{by \eqref{eq:bj-2}, \eqref{eq:bj-1}, \eqref{eq:ak+}, and \eqref{eq:defci}}}. \label{eq:dk+}
 \end{align}
So $\lambda_n\in S_k$ for all $j\leq k\leq n$ for any given $1\leq j\leq n$. By choosing $j=1$, we get $\lambda_n\in S$. So $S\neq\emptyset$. This completes the proof of Theorem~\ref{thm:basic_lemma}(\ref{itm:S}). 

\subsubsection{Proof of \Cref{thm:basic_lemma}(\ref{itm:S0})} We are given $1\leq j<n$ and $j\leq n_0\leq n$ such that \eqref{eq:S0} holds. Since $\lambda_n\in S$, by \eqref{def:dk} we get $d_j=0$ and $d_{n_0}=0$. So by \eqref{eq:dk+} for $k=j$, we get 
\begin{equation} \label{eq:ck0}
  a_j=0, \text{ and }  c_k=0, \, \forall j\leq k\leq n-1.
\end{equation}
And by \eqref{eq:dk+} for $k=n_0$, we get $a_{n_0}=0$. So by \eqref{eq:ak+1} for $i=n_0$, we get 
\begin{equation} \label{eq:ak0}
    \delta_k=0,\, \forall j< k\leq n_0.
\end{equation}
By \eqref{eq:mi-ci} and \eqref{eq:ck0}, for all $j<i\leq n-1$, we have
\begin{equation} \label{eq:mi0} 
  m_i=(c_i-c_{i-1})/i=0 \text{, and hence }  \lambda_i=\lambda_{i-1}+m_i\beta_i=\lambda_{i-1}.
\end{equation}

We recall the action of $G^j$ on $W_j$ as in \Cref{sec:step1}. Let $w_0=\Proj_{W_j}(v)$.
\begin{claim}
\label{claim:w0-fixed} $v_{\lambda_0}$ is fixed by $G^j$ and $\lambda_0\in S^j_j(w_0)$ (see \eqref{eq:defSj}).
\end{claim}

\begin{proof}
 By \eqref{eq:w0-eigen} $w_0$ is an eigenvector of $H_C^j$ with eigenvalue $b_j$. 
We have $a_j=0$. Having picked any $\lambda_j\in S_j^j(w_0)$, we obtained $\lambda_k$ for all $j< k\leq n$ as in the beginning of \Cref{sec:NextStep}. Therefore by \eqref{eq:bj-2} and \eqref{eq:ck0}, 
\[
0=c_j=(\lambda_j-\lambda_0)(H_C)=(\lambda_j-\lambda_0)(H_C^j),\,\forall \lambda_j\in S_j.
\]
Now we apply \Cref{thm:basic_lemma}(\ref{itm:vGfixed}) to $G^j=\SL(j+1,\R)$ action on $W_j$ for $w_0$ in place of $v$, where $H_j^j$ plays the role of $H_n$ and $H_C^j$ plays the role of $H_C$, to conclude that $w_0$ is $G^j$-fixed. As a consequence $S^j_j(w_0)=\Lambda_{w_0}$. 

Since $\exp(\la{h})$ normalizes $G^j$ and preserves $W_j$, the set of $G^j$-fixed vectors in $W_j$ is $\la{h}$-invariant. Since $w_0$ is $G^j$-fixed, we have that $(w_0)_\mu$ is also $G^j$-fixed for any $\mu\in \Lambda_{w_0}$.

Since $\Proj_{W_j}$ is equivariant under the actions of $G^j$ and $\la{h}$, for any $g\in G^j$,
\[
\Proj_{W_j}(v)=w_0=gw_0=\Proj_{W_j}(gv),
\]
and since $V_{\lambda_0}\subset W_j$, 
\[
(gv)_{\lambda_0}=(\Proj_{W_j}(gv))_{\lambda_0}=(\Proj_{W_j}(v))_{\lambda_0}=(w_0)_{\lambda_0}.
\]
By putting $g=e$ and $g=u_j\cdots u_1\in G^j$, we get 
\[
0\neq v_{\lambda_0}=(w_0)_{\lambda_0}=(v_j)_{\lambda_0}.
\]
Therefore $\lambda_0\in \Lambda_{w_0}=S_j^j(w_0)$, and $v_{\lambda_0}$ is $G^j$-fixed. This proves \Cref{claim:w0-fixed}.
\end{proof}

Now we fix $\lambda_j=\lambda_0\in S_j$. So by \eqref{eq:mi0} we get 
\begin{equation} \label{eq:lambdai0}
\lambda_i=\lambda_j=\lambda_0,\,\forall j\leq i\leq n-1.
\end{equation}

Let $i\in \{j+1,\ldots,n\}$. As in \Cref{sec:NextStep}, consider the $L_i\cong \SL_2(\R)$ submodule
\begin{equation} \label{def:Vi}
V_i=\bigoplus_{k\in\Z} V_{\lambda_{i-1}+k\beta_i}=\bigoplus_{k\in\Z} V_{\lambda_{0}+k\beta_i},
\end{equation}
see~\eqref{eq:Vi}, and let $w_{i-1}:=\Proj_{V_i}(v_{i-1})$ as in  \eqref{eq:wi-1}.

\begin{claim} \label{claim:wi-1} 
Then $w_{i-1}=(v_{i-1})_{\lambda_0}$.
\end{claim}

\begin{proof}
Suppose that $\lambda_0+k\beta_i\in \Lambda_{v_{i-1}}$ for some $k\in\Z$. By \eqref{eq:u+w} of Claim~\ref{claim:HC-min}, we have $\Hcmin(v_{i-1})\geq \Hcmin(v)$.
Therefore $(\lambda_0+k\beta_i)(H_C)\geq\lambda_0(H_C)$. As $\beta_i(H_C)=i\geq 1$, we have $k\geq 0$. Since $\lambda_0=\lambda_{i-1}$, we get $k\leq 0$ by Claim~\ref{claim:kleq0}. Therefore $k=0$.
\end{proof}

By \eqref{eq:w_imax}, \eqref{eq:lambdaiAi}, \eqref{eq:temp1-sum}, and \eqref{eq:first_inequality}
\begin{gather}
    \lambda^{\max}(w_{i-1})=\lambda_{i-1}(A_i) \text{ and } \lambda^{\max}(u_iw_{i-1})=\lambda_{i}(A_i)
\label{eq:min-max-i}\\
\lambda^{\max}(w_{i-1})+\lambda^{\max}(u_iw_{i-1})=2\lambda_i(A_i)-2m_i=2\delta_i. \label{eq:sum_2delta}
\end{gather}

\begin{claim} \label{claim:cases} Let $j<i\leq n$. 
\begin{enumerate}
    \item \label{itm:in-1}  If $m_i=0$, then $w_{i-1}$ is fixed by $u_i$.
    \item \label{itm:im} If $m_i=0$ and $\delta_i=0$, then $w_{i-1}$ is fixed by $L_i$.
\end{enumerate}
\end{claim}

\begin{proof} We will apply Lemma~\ref{lem:SL_2_basic_lemma} for the $L_i$ action on $V_i$, for $u_i,A_i\in L_i$ and the vector $w_{i-1}\in V_i$. By Claim~\ref{claim:wi-1}, $w_{i-1}$ is an eigenvector of $A_i$. 

First suppose $m_i=0$. Then $\lambda_i=\lambda_{i-1}+m_i\beta_i=\lambda_{i-1}$. So  by \eqref{eq:min-max-i},
\begin{equation} \label{eq:max-min-same}
\lambda^{\max}(w_{i-1})=\lambda^{\max}(u_iw_{i-1}).
\end{equation}
So by Lemma~\ref{lem:SL_2_basic_lemma}(\ref{itm:sl2-eigen=}), $w_{i-1}$ is fixed by $u_i$. This proves (\ref{itm:in-1}). 

Further suppose $\delta_i=0$. So by \eqref{eq:sum_2delta} and \eqref{eq:max-min-same},
\begin{equation*} 
\lambda^{\max}(w_{i-1})=\lambda^{\max}(u_iw_{i-1})=0.
\end{equation*}
So by Lemma~\ref{lem:SL_2_basic_lemma}(\ref{itm:sl2-eigen0}), $w_{i-1}$ is fixed by $L_i$. This proves (\ref{itm:im}).
\end{proof}

\begin{claim} \label{claim:wiw0} 
$w_{k}=v_{\lambda_0}$ for all $j\leq k\leq n-1$.
\end{claim}

\begin{proof}
Let $i\in \{j+1,\ldots,n-1\}$. Then $m_i=0$ by \eqref{eq:mi0}, and hence $u_iw_{i-1}=w_{i-1}$ by \Cref{claim:cases}(\ref{itm:in-1}). So by \eqref{eq:ProjVi}, $\Proj_{V_i}(v_i)=u_iw_{i-1}=w_{i-1}$. And by \Cref{claim:wi-1}, $w_{i-1}\in V_{\lambda_0}\subset V_i$. Therefore $(v_i)_{\lambda_0}=w_{i-1}$. So by \Cref{claim:wi-1}, $w_i=(v_i)_{\lambda_0}=w_{i-1}$. 

Thus $w_k=w_j$ for all $j<k\leq n-1$, and by \Cref{claim:w0-fixed} and \Cref{claim:wi-1}, we have $w_j=(v_j)_{\lambda_0}=v_{\lambda_0}$.
\end{proof}

\begin{claim} \label{claim:Svlam0}
We have $u(\bfx)v_{\lambda_0}=u_nv_{\lambda_0}=\Proj_{V_n}(u(\bfx)v)$. In particular,
\[
S(v_{\lambda_0})=S(v)\cap \Lambda_{V_n}\ni \{\lambda_n\}.
\]
Moreover, if $\lambda_n(H_C)=b$, then $m_n=0$.
\end{claim}

\begin{proof}
By \eqref{eq:mi0}, \Cref{claim:cases}(\ref{itm:in-1}), and \Cref{claim:wiw0}, we have that $u_iv_{\lambda_0}=v_{\lambda_0}$ for all $j\leq i\leq n-1$. So
\(
u(\bfx)v_{\lambda_0}=u_n(u_{n-1}\cdots u_1v_{\lambda_0})=u_nv_{\lambda_0}.
\)
We consider the standard description of the $L_n\cong \SL(2,\R)$ representation on $V_n$, see \eqref{eq:Vi} and \eqref{def:Vi}. 
Now $v_{\lambda_0}$ is an eigen-vector of $A_n$, $\lambda_{n-1}=\lambda_0$, and by definition of $m_n$ as in \eqref{eq:lambda_i}, we get $m_n\geq 0$ and
\begin{equation} \label{eq:unvlam0}
\Lambda_{u_nv_{\lambda_0}}=\{\mu_m:=\lambda_0+m\beta_n:0\leq m\leq m_n\}.
\end{equation}
By definition $v_n=u(\bfx)v$ and by \eqref{eq:ProjVi},
\[
\Proj_{V_n}(v_n)=\Proj_{V_n}(u_nv_{n-1})=u_n\Proj_{V_n}(v_{n-1})=u_nw_{n-1}=u_nv_{\lambda_0}=u(\bfx)v_{\lambda_0}.
\]

We have $\lambda_0(H_C)=b$, $\beta_n(H_C)=n$, and $\lambda_n=\lambda_{n-1}+\beta_n=\lambda_0+m_n\beta_n$. So $\lambda_n(H_C)=b+nm_n$. Now if $\lambda_n(H_C)=b$, then $m_n=0$. 
\end{proof}

\begin{claim} \label{claim:Gk-fixed}
$v_{\lambda_0}$ is fixed by $G^{k}:={\begin{bsmallmatrix} \SL(k+1,\R)& \\ & I_{n-k} \end{bsmallmatrix}}$, where  $k=\min\{n-1,n_0\}$. 
\end{claim}

\begin{proof}
In view of \eqref{eq:ak0} and \eqref{eq:mi0}, by combining \Cref{claim:w0-fixed}, \Cref{claim:cases}(\ref{itm:im}) and \Cref{claim:wiw0}, we conclude that $v_{\lambda_0}$ is fixed by the subgroup generated by $G^j$ and $L_i$ for $j+1\leq i\leq \min\{n-1,n_0\}$ which equals $G^{k}$.
\end{proof}

\subsubsection{Induction on maximal element of $\Lambda_{v}$} We will prove \Cref{thm:basic_lemma}(\ref{itm:Qn0fixed},\ref{itm:Gn0fixed}) by induction on the maximum element of $\Lambda_{v}$ with respect to the lexicographic total order on $\la{h}^\ast$ (see \Cref{def:lexico}).

Suppose $v':=v-v_{\lambda_0}\neq 0$. Then the maximal element of $\Lambda_{v'}$ is strictly smaller than $\lambda_0$. Now $u(\bfx)v'=u(\bfx)v-u(\bfx)v_{\lambda_0}$.  By \Cref{claim:Svlam0},
\[
u(\bfx)v_{\lambda_0}=\Proj_{V_n}(u(\bfx)v) 
\text{, so } u(\bfx)v'=\Proj_{V_n^\perp}(u(\bfx)v),
\]
where $V_n^\perp=\oplus\{V_{\mu}:\mu\in\la{h}^\ast\setminus \Lambda_{V_n}\}$
and $V=V_n\oplus V_n^\perp$. Therefore 
\[
S(v_{\lambda_0})=S(v)\cap \Lambda_{V_n},\, S(v')=S(v)\cap \Lambda_{V_n^{\perp}}, \text{ and }
S(v)=S(v_{\lambda_0})\sqcup S(v').
\]
Hence
 \begin{equation*} \label{eq:Svprime}
    (u(\bfx)v)_{S(v)}=(u(\bfx)v)_{S(v_{\lambda_0})}+(u(\bfx)v)_{S(v')}
    =(u(\bfx)(v_{\lambda_0}))_{S(v_{\lambda_0})}+(u(\bfx)v')_{S(v')}.
\end{equation*}

Therefore if we know that any of
\Cref{thm:basic_lemma}(\ref{itm:Qn0fixed} or \ref{itm:Gn0fixed}) is holds for $v_{\lambda_0}$, as well as $v'$, in place of
$v$, then the corresponding statement is holds for $v$. Therefore by using induction on the maximal element in $\Lambda_v$, it is sufficient to prove \Cref{thm:basic_lemma}(\ref{itm:S0}) for the case of $v=v_{\lambda_0}$.

\subsubsection{Proof of \Cref{thm:basic_lemma}(\ref{itm:S0}) for $n_0=n$} 
Let 
\(
U_n^-=\begin{bsmallmatrix} 1 & & \\ & I_{n-1} & \\ \R & & 1 \end{bsmallmatrix}\subset L_n
\).

\begin{claim} \label{claim:m=n} 
We have that $v_{\lambda_0}$ is fixed by $U_n^-$. Moreover, if $m_n=0$, then $v_{\lambda_0}$ is fixed by $L_n$.
\end{claim}

\begin{proof} 
Since $n_0=n$, $\delta_n=0$ by \eqref{eq:ak0}. So by \eqref{eq:sum_2delta},
\begin{equation} \label{eq:imndelta}
\lambda^{\max}(w_{n-1})+\lambda^{\max}(u_nw_{n-1})=2\delta_n=0.
\end{equation}
By \Cref{claim:wiw0}, $w_{n-1}=v_{\lambda_0}$ is an eigen-vector of $A_n$. We apply \Cref{lem:SL_2_basic_lemma}(\ref{itm:sl2-eigen0}) to the representation $V_n=\sum_{k\in\Z} V_{\lambda_0+k\beta_n}$ of $L_n\cong \SL(2,\R)$ for the actions of its diagonal element $A_n$ and the unipotent element $u_n=u(0,\ldots,0,x_n)$ on the vector $w_{n-1}=v_{\lambda_0}$ to conclude the following: $v_{\lambda_0}$ is fixed by $U_n^-$. Moreover, if $m_n=0$, then by \Cref{claim:cases}(\ref{itm:im}), $v_{\lambda_0}=w_{n-1}$ is fixed by $L_n$.
\end{proof}

So by \Cref{claim:Gk-fixed} and \Cref{claim:m=n}, $v_{\lambda_0}$ is fixed by the subgroup generated by $G^{n-1}$ and $U_n^-$ which equals
$\Po={
	\begin{bsmallmatrix}
	\SL(n,\R) & \mathbf{0}_{n\times 1}\\ \R^n & 1
	\end{bsmallmatrix}}$. 
Therefore \Cref{thm:basic_lemma}(\ref{itm:Qn0fixed}) holds for $v=v_{\lambda_0}$ and $n_0=n$.

Now suppose $\lambda_n(H_C)-b=0$. Then by \Cref{claim:Svlam0},  $m_n=0$. So by \Cref{claim:m=n}, $v_{\lambda_0}$ is also fixed by $L_n$. Therefore $v_{\lambda_0}$ is fixed by $G$, which is generated by $\Po$ and $L_n$. Therefore  \Cref{thm:basic_lemma}(\ref{itm:Gn0fixed}) holds for $v=v_{\lambda_0}$ and $n_0=n$. 

\subsubsection{Proof of \Cref{thm:basic_lemma}(\ref{itm:S0}) for $n_0<n$}  By \Cref{claim:Gk-fixed}, $v_{\lambda_0}$ is fixed by $G^{n_0}$. Now by \Cref{claim:cases}(\ref{itm:in-1}) and \Cref{claim:wi-1}, we get that $v_{\lambda_0}$ is fixed by $u_i$ for all $n_0< i\leq n-1$. The smallest subgroup containing $\{u_i:n_0< i\leq n-1\}$ and normalized by $G^{n_0}$ is $\begin{bsmallmatrix} I_{n_0+1}& M_{(n_0+1)\times(n-n_0-1)} \\ &I_{n-n_0-1} \\ & &1 \end{bsmallmatrix}$. 
Therefore $v_{\lambda_0}$ is fixed by 
\[
\begin{bsmallmatrix} \SL(n_0+1,\R)& M_{(n_0+1)\times(n-n_0-1)} \\ &I_{n-n_0-1} \\ & &1 \end{bsmallmatrix}=G_{n_0}\cap Q,
\text{ where } 
G_{n_0}=\begin{bsmallmatrix}
	\SL(m+1,\R)&M_{m+1,n-m}\\&I_{n-m}
	\end{bsmallmatrix}.
	\]
Moreover by \Cref{claim:Svlam0}, since $n_0<n$, we get \[
(u(\bfx)v_{\lambda_0})_{S(v_{\lambda_0})}=u(\bfx)v_{\lambda_0}=u_nv_{\lambda_0}=u((0,\ldots,0,x_n))v_{\lambda_0}.
\]
Therefore \Cref{thm:basic_lemma}(\ref{itm:Qn0fixed}) holds of $v=v_{\lambda_0}$ and $n_0<n$. 

Finally, suppose that $\mu(H_C)=b$ for all $\mu\in S(v_{\lambda_0})$.  By \Cref{claim:Svlam0}, since $\lambda_n\in S(v_{\lambda_0})$, we get $m_n=0$. Then by \Cref{claim:cases},  $v_{\lambda_0}$ is fixed by $u_n$. Therefore $v_{\lambda_0}$ is fixed by $G_{n_0}$, which is generated by $G_{n_0}\cap \Po$ and $u_n$. Therefore \Cref{thm:basic_lemma}(\ref{itm:Gn0fixed}) holds of $v=v_{\lambda_0}$ and $n_0<n $. 

This completes the proof of \Cref{thm:basic_lemma}. \qed

\section{Translates of smooth curves}
\label{sec:curves}

We begin this section by showing that if a map $\phi:(0,1)\to \R^n$ is $\ell$-nondegenerate for some $\ell\geq n$, then $\phi$ is ordered regular at each $s$ outside a discrete subset of $(0,1)$.

\begin{proposition} \label{prop:nondeg-reg}
Let $n\in\N$, $\ell\geq n$, and suppose that $\phi:(0,1)\to \R^n$ is $\ell$-nondegenerate at some $s\in (0,1)$. Then for 
$\psi:=\phi^{(1)}\wedge\cdots\wedge \phi^{(n)}:(0,1)\to\wedge^n \R^n$, 
we have $\psi^{(m)}(s)\neq 0$ for some $0\leq m\leq n(\ell-n)$.
\end{proposition}

\begin{proof}
For $m\in\N\cup\{0\}$, let
\[
S_m=\{(j_1,\dots,j_n):1\leq j_1<j_2<\ldots<j_n,\, \sum_{k=1}^n (j_k-k)=m\}. \]
Then we can inductively verify that
\begin{equation} \label{eq:wedge-m}
    \psi^{(m)}=\sum_{J=(j_1,\ldots,j_n)\in S_m} c_J \cdot \phi^{(j_1)}\wedge \cdots \wedge \phi^{(j_n)} \text{, where } 0\neq c_J\in\N,\,\forall J\in S_m.
\end{equation}

For $1\leq k\leq n$, let $l_k(s)\in\N$ be the smallest $l$ such that $\{\phi^{(i)}(s):1\leq i\leq l\}$ contains $k$ linearly independent vectors. Then 
\begin{gather} \label{eq:lks}
k\leq l_k(s)\leq \ell-(n-k),\\
 \label{eq:lis}
\phi^{(l_1(s))}(s)\wedge \cdots \wedge \phi^{(l_k(s))}(s)\neq 0\text{, and }\\
\phi^{(j_1)}(s)\wedge \cdots \wedge \phi^{(j_k)}(s)=0\text{, for any } 1\leq j_1<j_2<\cdots<j_k<l_k(s).
\label{eq:li}
\end{gather}

Choose $m=\sum_{k=1}^n (l_k(s)-k)$. Then by \eqref{eq:lks}, $0\leq m\leq n(\ell-n)$. Suppose $J=(j_1,\ldots,j_n)\in S_m$ is such that 
\[
\phi^{(j_1)}(s)\wedge \cdots \wedge \phi^{(j_n)}(s)\neq 0.
\]
Then for each $1\leq k\leq n$, the term 
\[
\phi^{(j_1)}(s)\wedge \cdots \wedge \phi^{(j_k)}(s)\neq 0;
\]
and hence $j_k\geq l_k(s)$ by \eqref{eq:li}. Since 
\[
\sum_{k=1}^n (j_k-k)=m=\sum_{k=1}^n (l_k(s)-k),
\]
we get $j_k=l_k(s)$ for all $1\leq k\leq n$. Thus we conclude that 
\[
\psi^{(m)}(s)=c_{J}\cdot \phi^{(l_1(s))}(s)\wedge \cdots \wedge \phi^{(l_k(s))}(s)\neq 0,
\]
where $J=(l_1(s),\ldots,l_n(s))\in S_m$, by \eqref{eq:wedge-m} and \eqref{eq:lis}.
\end{proof}

\begin{proposition} \label{prop:Rolls}
Let $\phi:(0,1)\to\R$ be such that for each $s\in (0,1)$, $\phi^{(\ell)}(s)\neq 0$ for some $\ell\in\N\cup\{0\}$. Then $Z=\{s\in (0,1):\phi(s)=0\}$ is a discrete subset of $(0,1)$; here $0$ or $1$ may be limit points of $Z$ in $[0,1]$.
\end{proposition}

\begin{proof}
Let $s\in Z$. Then $l=\inf\{i\in\N:\phi^{(i)}(s)\neq 0\}\in \N$. By Taylor's expansion, 
$\phi(s+h)=h^l\phi^i(s)/(l!)+o(h^{l})\neq 0$,
for all small $h\neq 0$.
\end{proof}

\begin{corollary} \label{cor:ord-reg}
Let $\phi:(0,1)\to \R^n$ such that for each $s\in (0,1)$ there exists some $\ell\geq n$ such that $\phi$ is $\ell$-nondegenerate at $s$. Then there exists a discrete subset $Z$ of $(0,1)$ such that $\phi$ is ordered regular at all points of $(0,1)\setminus Z$; see \Cref{def:ord_reg}.
\end{corollary}

\begin{proof}
Let $1\leq i\leq n$. Let $\pi_i:\R^n\to \R^i$ denote the projection on the first $i$-coordinates. Let $s\in (0,1)$. Then there exists $\ell\geq n$ such that $\phi$ is $\ell$-nondegenerate. Then $\phi_i:=\pi\circ\phi:(0,1)\to \R^i$ is $\ell$-nondegenerate at $s$, because $\phi_i^{(l)}=\pi_i\circ \phi^{(l)}$ for any $l$. Let 
\[
\psi_i=\phi_i^{(1)}\wedge \cdots \wedge \phi_i^{(i)}:(0,1)\to \wedge^i\R^i\cong \R.
\]
Then by \Cref{prop:nondeg-reg},  $\psi_i^{(m)}(s)\neq 0$ for some $m\geq 0$. Hence by \Cref{prop:Rolls}, $Z_i:=\{s\in (0,1):\psi_i(s)=0\}$ is a discrete subset of $(0,1)$. Then $\phi_i$ is $i$-nondegenerate at all $s\in (0,1)\setminus Z_i$. 
Now we put $Z=\cup_{i=1}^n Z_i$. Then $Z$ is discrete in $(0,1)$, and $\phi$ is ordered regular at all $s\in (0,1)\setminus Z$.
\end{proof}

\subsubsection{Approximation by a polynomial map} \label{subsec:approx} Let $s\in (0,1)$ and suppose that $\phi$ is ordered regular at $s$. In view of \Cref{rem:OR}, let $B(s)=(b_{i,j}(s))$ be the $n\times n$ upper triangular unipotent matrix  such that 
\begin{align} \label{eq:Bsbis}
    (\phi^{(i)}(s)/i!)B(s)\inv =\kappa_i(s)\bfe_i+\sum_{1\leq j<i} \kappa_{j,i}(s)\bfe_j, 
\end{align}
where $\kappa_i(s)\neq 0$ and $\kappa_{j,i}(s)\in\R$ for each $1\leq i\leq n$ and $1\leq j<i$. We note that $\kappa_i(s)$ and $\kappa_{i,j}(s)$ are continuous functions of $s$.

For any small enough $h\in \R$, by Taylor expansion
\begin{align*}
    \phi(s+h)-\phi(s)=\sum_{i=1}^k (\phi^{(i)}(s)/i!) h^i + \bfo(h^k),
\end{align*}
where $\bfo(h^k)/h^k\to 0$ as $h\to 0$. Then
\begin{align}
    (\phi(s+h)-\phi(s))B(s)\inv& = \sum_{i=1}^n 
\bigl(\kappa_i(s) h^i+\sum_{j=i+1}^k \kappa_{i,j}(s) h^j\bigr)\bfe_i+\bfo(h^k),\label{eq:Bsinv1}\\
    &= R_s(h)+\bfo(h^k), \label{eq:Bsinv}
\end{align}
where $\kappa_{i,j}(s)\in\R \text{ for all $i<j\leq k$}$,
\begin{gather}
    R_s(h)=\sum_{i=1}^n (\kappa_i(s) +\epsilon_i(s,h)) h^i\bfe_i, \notag\\
    \epsilon_i(s,h) := \sum_{j=i+1}^k \kappa_{i,j}(s)h^{j-i} \text{, so } \epsilon_i(s,h)=o(h).  
    \label{eq:epsiloni}
\end{gather}
We note that $h\mapsto R_s(h)$ is a polynomial map in $h$ of degree at most $k$. 
\begin{remark} \label{rem:drop-s} Suppose that $\phi$ is ordered regular at $s_0\in (0,1)$. 
For each $t\geq 0$, we pick $s_t\in (0,1)$ such that $s_t\to s_0$ as $t\to\infty$. Hence $\phi$ is ordered regular at $s_t$ for all large $t$. We will consider the above expressions for $s_t$ in place of $s$ for all large $t$, and to keep our notation manageable, we will not display dependence on $t$ and $s_0$ in various quantities: We shall write $s$ for $s_t$,  $\kappa_i$ for $\kappa_i(s)$ and $\kappa_{i,j}$ for $\kappa_{i,j}(s)$, $R(h)$ in place of $R_{s}(h)$, and $\epsilon_i(h)$ in place of $\epsilon_i(s,h)$.
\end{remark}

\begin{remark} \label{rem:h-negative}
For simplicity of notation, we will only consider the case of $h\geq 0$. If $h<0$, we will work with $\tilde R$ in place of $R$, where 
\[
\tilde R(-h):=R(h)=\sum_{i=1}^n (-1)^i\left(\kappa_i(s)+\sum_{j=i+1}(-1)^{j-i}k_{i,j}(s)(-h)^{j-i}\right)(-h)^i. 
\]
\end{remark}

\subsubsection{Translate by diagonals}  \label{subsec:xi}

Recalling \eqref{def2:ri} and \Cref{sec:ri}, for $t\in\cT$,  let
\[
a_t=\diag(e^{nt},e^{-r_1(t)},\ldots,e^{-r_n(t)})=\exp(\HA),
\]
where $\HA=\sum_{i=1}^n \xi_i(t)H_i$ is a convex combination of $tH_i$'s. For notational convenience, we may write $r_i$ to mean $r_i(t)$ and $\xi_i$ to mean $\xi_i(t)$. 

We pick any $k\in\N$ such that 
\begin{equation} \label{def:choose-k}
	\limsup_{t\to\infty} nt+r_1(t)-kt<\infty.
\end{equation}
For example, we can pick $k$ such that $k\geq n+t\inv r_1(t)$ for all large $t\in\cT$. 
For all $t\in\cT$, let $h_t>0$ be such that 
\begin{equation} \label{def:ht}
M:=\limsup_{t\to\infty} h_t^ke^{nt+r_1(t)}<\infty.
\end{equation}

Suppose that $\phi$ is $C^k$ and ordered regular in a neighborhood of $s$. 
Let $v(s)=\begin{bsmallmatrix} 1 & \\ & B(s)  \end{bsmallmatrix}$. For all $\bfx\in\R^n$, 
\begin{equation}
    v(s)u(\bfx)v(s)\inv=u(\bfx B(s)\inv). 
    \label{eq:vs-conj}
\end{equation}
For $t\in\cT$, let
\begin{equation} \label{eq:vinfty}
    v_t(s):=a_tv(s) a_t^{-1}=\begin{bsmallmatrix} 1 & \\ & \left(b_{i,j}(s)e^{-(r_i(t)-r_j(t))}\right)
    \end{bsmallmatrix},
\end{equation}
where $B(s)=(b_{i,j}(s))$ is an upper triangular unipotent matrix, so $b_{i,j}(s)=0$ for 
$i>j$, and $r_i(t)-r_j(t)\geq 0$ for all $i\leq j$. So $\{v_t(s)\}_{t\in\cT}$ is relatively compact in $G$.

Therefore
\begin{align}
    a_tu(\phi(s+h_t)) 
    &=a_tu(\phi(s+h_t)-\phi(s))u(\phi(s)) \nonumber\\
    &=a_t v(s)\inv v(s)u(\phi(s+h_t)-\phi(s))v(s)\inv v(s)u(\phi(s))\nonumber\\
    &=a_t v(s)\inv u((\phi(s+h_t)-\phi(s))B(s)\inv)v(s)u(\phi(s)) \text{, by \eqref{eq:vs-conj}}\nonumber\\
    &=v_t(s)\inv a_tu(\bfo(h_t^k))u(R(h_t))v(s)u(\phi(s)) \text{, by \eqref{eq:Bsinv}}\nonumber\\
    &=v_t(s)\inv u(e^{nt+r_1(t)}\bfo(h_t^k))\cdot a_tu(R(h_t))v(s)u(\phi(s)) \nonumber\\
    &=(I_{n+1}+\bfo_t(1))v_t(s)\inv a_tu(R(h_t))v(s)u(\phi(s)) 
    \text{, by \eqref{def:ht},}
    \label{eq:poly-approx}
\end{align}
where $\bfo_t(1)\stackrel{t\to\infty}{\longrightarrow} 0$ uniformly for any $\{h_t\}$ satisfying \eqref{def:ht}.

\subsubsection{Critical and slow shrinking}
For each $t$, let $\alpha_t\geq 1$ be such that 
\begin{equation} \label{eq:alphat}
   \limsup_{t\to\infty} \alpha_t^{k} e^{-tk+nt+r_1(t)}<\infty.
\end{equation}
In particular, since $nt+r_1(t)\to\infty$, 
\begin{align}
\label{eq:logalpha}
  \lim_{t\to\infty} \alpha_t e^{-t}=0\text{, or equivalently } \lim_{t\to\infty}  (t-\log\alpha_t) = \infty.
\end{align}

By \eqref{eq:epsiloni}, for each $t$, the map
\(
\eta\mapsto R(\alpha_t e^{-t}\eta)\in\R^n
\)
is a polynomial map of degree at most $k$ (in the variable $\eta$). 

Let $J\subset (0,\infty)$ be a finite closed interval. Let $\eta\in J$. Let $h_t:=\alpha_t e^{-t}\eta$, then by \eqref{eq:alphat}, $\limsup_{t\to\infty} h_t^ke^{nt+r_1(t)}\leq M$ for all $\eta\in J$ some $M>0$; in particular, \eqref{def:ht} holds. So by \eqref{eq:poly-approx}, for any finite interval $J$, asymptotically as $t\to\infty$, the expanded curve $\{a_tu(\phi(s+\alpha_t e^{-t}\eta))x:\eta\in J\}$ is `almost parallel' to the curve $\{a_{t}u(R(\alpha_t e^{-t}\eta))y:\eta\in J\}$ for $y=v(s)u(\phi(s))x$; so up to a bounded translate by $v_t(s)$, and hence their asymptotic behaviors are same as $t\to\infty$.

The case of bounded $\{\alpha_t\}_t$ corresponds to {\em critical shrinking\/} and the case of $\alpha_t\to\infty$ corresponds to {\em slow shrinking\/}.

\subsubsection{Notation}
\label{not:b}
Let $G=\SL(n+1,\R)$ and $\rho\colon G\to\GL(V)$ be a representation of $G$ on a finite dimensional vector space $V$. So $\la{g}=\mathfrak{sl}(n+1,\R)$ acts on $V$ via $d\rho$. We continue with notations in \Cref{sect:linear_dynamics}. Let 
\begin{gather*}
    \Delta=\{\lambda\in\h^\ast:V_\lambda\neq 0\}\\
    B=\{\mu(H_C):\mu\in \Delta\}.
\end{gather*}
Then $\Delta$ and $B$ are finite sets.

\begin{remark} \label{rem:int}
Suppose $V$ is an irreducible representation of $G$. Then for any $\mu_1,\mu_2\in \Delta$, we have $(\mu_1-\mu_2)\in \langle \beta_1,\ldots,\beta_n\rangle_{\Z}$. Since $\beta_i(H_C)\in\Z$ for all $i$, we get $\mu_1(H_C)-\mu_2(H_C)\in\Z$.
\end{remark}

For any $b\in B$, define
\begin{align*}
    V(b)&=\{v\in V: H_Cv=bv\}\\
    \Delta(b)&=\{\mu\in \Delta: \mu(H_C)=b\} \\
    \Delta^{0+}(b)&=\{\mu\in\Delta:\mu(H_C)-b\geq 0 \text{ and }\mu(H_k)-(\mu(H_C)-b)\geq 0,\, \forall 1\leq k\leq n\}
    \\
    S_n(b)&=\{\mu\in\Delta:\mu(H_C)-b\geq 0 \text{ and } \mu(H_n)-(\mu(H_C)-b)\geq 0\}.
\end{align*}

For any $v\in V$ and $S\subset \Delta$, we express 
\begin{gather*}
    v=\sum_{\lambda\in\Delta} v_{\lambda}\text{, where $v_{\lambda}\in V_\lambda$, and } v_S=\sum_{\lambda\in S} v_\lambda. 
\end{gather*}
In particular, for any $b\in B$, we write
\[
    v(b)=v_{\Delta(b)}=\sum_{\lambda\in \Delta(b)} v_\lambda. 
\]

\subsubsection{Choice of norm on $V$} \label{subsec:supnorm} We fix a basis of $V$ consisting of eigen vectors of the Lie algebra $\h$, and fix the sup-norm $\norm{\cdot}$ on $V$ in terms of coordinates  corresponding to the chosen basis. 

We note that for any $v\in V$ and any coordinate linear functional $\ell:V\to\R$, we have $\abs{\ell(v)}\leq \norm{v}$; where by a coordinate linear functional we shall mean that it takes value $\pm 1$ on one of the basis elements and it vanishes on the rest of the basis elements. 
Also for any $S\subset \Delta$ and any $v\in V$, we have $\norm{v_S}\leq \norm{v}$, and there exists a coordinate linear functional $\ell$ on $V$ such that $\ell(v)=\norm{v_S}$.

\begin{proposition}
\label{cor:basic}
Let $E$ be a compact subset of $(\R\setminus\{0\})^{n}$. Then there exists $D_1>0$ such that for all $b\in B$,  $\bfx\in E$ and $v\in V$, we have
\[
\norm{[u(\bfx)v(b)]_{\Delta^{0+}(b)}}\geq D_1 \norm{v(b)}
\]
\end{proposition}

\begin{proof} Let $b\in B$. Let $V^1(b)=\{v\in V(b):\norm{v}=1\}$. The function $f:E\times V^1(b)\to [0,\infty)$ given by 
\[
f(\bfx,v)=\norm{[u(\bfx)v]_{\Delta^{0+}(b)}}, \ \forall \bfx\in E,\, v\in V^1(b),
\]
is continuous. Let $x\in E$ and $v\in V^1(b)$. Since $H_Cv=bv$, $v\neq 0$, and $\bfx\in (\R\setminus\{0\})^n$,  by \Cref{thm:basic_lemma}(\ref{itm:Basic-1},\ref{itm:S}),  $\Delta^{0+}(b)=S(\bfx,v)\neq \emptyset$. Since $S(\bfx,v)\subset \Lambda_{u(\bfx)v}$, we get that $[u(\bfx)v(b)]_{\Delta^{0+}(b)}\neq 0$. Hence $f(\bfx,v)>0$.
Therefore 
\[
D(b):=\inf\{f(\bfx,v):\bfx\in E,\,v\in V^1(b)\}>0. 
\]

Since $B$ is finite, $D_1=\min\{D(b):b\in B\}>0$, and the conclusion of the proposition holds. 
\end{proof}

We will need the following simple property of polynomials.

\begin{lemma}[{\cite[Lemma~4.1]{EMS-GAFA97}}] \label{prop:Calpha}
Let $d\in\N$ and $J\subset (0,\infty)$ be an interval of finite positive length. Then there exists $C_{d,J}>0$ such that if $f(\eta)=\sum_{i=0}^d  c_i\eta^i$, where $c_i\in\R$, then 
\begin{equation} \label{eq:CJ}
\sup_{\eta\in J} \abs{f(\eta)}\geq C_{d,J} \max (\abs{c_0},\ldots,\abs{c_d}).
\end{equation}
\end{lemma}

\begin{proof}
Suppose $J=[\eta_0,\eta_d]\subset (0,\infty)$. Let
$\eta_i=\eta_0+i\abs{J}/d$ for $i=1,\ldots,d-1$, where $\abs{J}=\eta_d-\eta_0>0$. Let $V$ be the $(d+1)\times(d+1)$ Vandermonde matrix whose $(i,j)$ entry is $\eta_{j-1}^{i-1}$, where $i,j\in\{1\ldots,d+1\}$. Let $\bfc$ denote the column vector whose $i$-th coordinate is $c_{i-1}$ and $\bff$ be the column vector whose $i$-th coordinate is $f(\eta_{i-1})$. Then $\bff=V\bfc$, and hence $\bfc=V^{-1} \bff$. 
By a direct calculation using Vandermonde determinants (see~\cite{Gautschi1962}),
\[
\norm{V^{-1}}\leq (1+\eta_d)(\abs{J}/d)^{-d},
\]
where $\norm{\cdot}$ denotes the sup-norm.  Therefore, $\norm{\bfc}\leq d\norm{V^{-1}}\norm{\bff}$. So 
\[
\norm{\bff}\geq d^{-1} \norm{V^{-1}}^{-1} \norm{\bfc}.
\]
Then \eqref{eq:CJ} holds for
\[
C_{d,J}= d^{-1} \norm{V^{-1}}^{-1}\geq \abs{J}^d/[d^{d+1}(1+\eta_d)]>0.
\]
\end{proof}

\subsubsection{The importance of the choice of $H_C$}
By the definition of $\beta_i$ in \eqref{eq:beta_i}, for any $H\in\h$ and $\sum_{i=1}^n x_i\bfe_i\in \R^n$, we have
	    \[
	    \exp(H)u\left(\sum_{i=1}^n x_i\bfe_i\right)\exp(-H)=u\left(\sum_{i=1}^n x_ie^{\beta_i(H)}\bfe_i\right).
	    \]
Let $h>0$, $s=-\log h$, and $H=sH_C$. Let 
$R(h)=\sum_{i=1}^n c_ih^i \bfe_i$, where $c_i\in\R$.  
By definition, $\beta_i(H_C)=i$, so $h^ie^{\beta_i(sH_C)}=1$ for all $1\leq i\leq n$. Therefore 
\begin{align} 
\exp(sH_C)u(R(h))\exp(-sH_C) 
&=u\left(\sum_{i=1}^n c_i \bfe_i\right). 
\label{eq:HCuQ}
\end{align}

Now we are ready to prove our following main consequence of Theorem~\ref{thm:basic_lemma}. 

\subsubsection{Notation} \label{not:alphat}
For the rest of the section,
each $t\geq 0$, let $\alpha_t\geq 1$ such that either $\alpha_t=1$ for all $t$, or $\lim_{t\to\infty} 
\alpha_t=\infty$ and $\lim_{t\to\infty}\alpha_te^{-t}=0$.

We fix $s_0\in (0,1)$ and assume that $\phi$ is ordered regular at $s_0$. Let $s_t\in (0,1)$ such that $s_t\to s_0$ as $t\to\infty$. If $\alpha_t=1$ for all $t$, we will assume that $s_t= s_0$ for all $t$. 

Let $h\geq 0$, in view of \eqref{eq:epsiloni} and \Cref{rem:drop-s}, we write $s=s_t$ and
\begin{equation} \label{eq:Rh}
    R(h)=\sum_{i=0}^n (\kappa_i+\epsilon_i(h))h^i\bfe_i,
\end{equation}
where $\kappa_i=\kappa_i(s_t)$ and $\epsilon_i(h)=\epsilon_i(s_t,h)$.

\begin{proposition} \label{prop:main-basic}
Let $J\subset (0,\infty)$ be a compact interval of positive length. There exists $D_2>0$ and $T\geq 0$ such that for any $v\in V$ and $t\geq T$, we have
\begin{equation}
   M_t:=\sup_{\eta\in J}\norm{a_tu(R(\alpha_te^{-t}\eta))v}\geq D_2\norm{v}.
\end{equation}
\end{proposition}

\begin{proof}
Expressing $V$ as a direct sum of irreducible representations of $G$, and considering the projection of $v$ on each of them, without loss of generality, we may assume that $V$ is an irreducible representation of $G$.

For $t\geq 0$ and $\eta\in J$, let $h=\alpha_te^{-t}\eta$ and $\bfy=\sum_{i=1}^n (\kappa_i+\epsilon_i(h))\bfe_i$. Let $v\in V$, $b\in B$ and $\mu\in \Delta$. Let $s=-\log h$. Then by \eqref{eq:HCuQ}, 
\begin{align}
    &[a_tu(R(h))v(b)]_\mu \notag\\
    &=[a_t\exp(-sH_C) (\exp(sH_C)u(R(h))\exp(-sH_C))\exp(sH_C)v(b)]_\mu \notag\\
    &=e^{\mu(\HA)}e^{-s\mu(H_C)}[u(\bfy)\exp(sH_C)v(b)]_{\mu} \notag\\
    &=e^{\mu(\HA)}e^{-s(\mu(H_C)-b)}[u(\bfy)v(b)]_{\mu}
    \notag\\
    &=e^{\mu(\HA)-t(\mu(H_C)-b))}\alpha_t^{\mu(H_C)-b}\eta^{\mu(H_C)-b}[u(\bfy)v(b)]_{\mu}. \label{eq:expansion1}
\end{align}
If $[u(\bfy)v(b)]_{\mu}\neq 0$, then $\mu(H_C)-b\geq 0$ by \Cref{thm:basic_lemma}(\ref{itm:Basic-1}). So by \Cref{rem:int}, 
$\mu(H_C)-b\in \{0,\ldots,d\}$ for $d=\max\{b_1-b_2:b_i\in B\}\in \{0\}\cup\N$.

Let $v\in V$. Let $b\in B$. Let $\ell:V\to\R$ be a coordinate linear functional on $V$, see \Cref{subsec:supnorm}. Let $\mu\in \Delta$. Then by \eqref{eq:expansion1}, for each $\eta\in J$,
\begin{align*}
    M_t
    &\geq\norm{[a_{t}R(h)v]_{\mu}}\\
    &\geq \ell\bigl(\sum_{b'\in B} [a_{t}R(h)v(b')]_{\mu}\bigr)\text{, by definition of the sup-norm}\\
    &=\sum_{b'\in B} e^{\mu(\HA)-t(\mu(H_C)-b')}\alpha_t^{\mu(H_C)-b'}\eta^{\mu(H_C)-b'}\ell\bigl([u(\bfy)v(b')]_{\mu}\bigr),
\end{align*}
which is a polynomial of degree at most $d$ in the variable $\eta$. So by \Cref{prop:Calpha}, considering the coefficient of $\eta^{\mu(H_C)-b}$, we get
\begin{align}
M_t&\geq C_{d,J} e^{\mu(\HA)-t(\mu(H_C)-b)}\alpha_t^{\mu(H_C)-b}\ell\bigl([u(\bfy)v(b)]_{\mu}\bigr). \label{eq:imp-bound}
\end{align}

Since $\HA=\sum_{\ell=1}^n \xi_i H_i$, $\sum_{i=1}^n \xi_i=t$ and $\xi_i\geq 0$, 
\begin{equation} \label{eq:HAxi}
\mu(\HA)-t(\mu(H_C)-b))=\sum_{i=1}^n \xi_i \bigl(\mu(H_i)-(\mu(H_C)-b)\bigr).
\end{equation}
Therefore for any $\mu\in \Delta^{0+}(b)$, $\mu(H_C)-b\geq 0$ and 
\[
\mu(\HA)-t(\mu(H_C)-b)=\sum_{i=1}^n \xi_i\bigl(\mu(H_i)-(\mu(H_C)-b)\bigr)\geq 0.
\]
Therefore by \eqref{eq:imp-bound}, since $\alpha_t\geq 1$,
\[
M_t\geq C_{d,J} \ell\bigl([u(\bfy)v(b)]_{\mu}\bigr), \,\forall \mu\in \Delta^{0+}(b).
\]
Since $\ell$ is a coordinate linear functional, $\ell\bigl([u(\bfy)v(b)]_{\mu}\bigr)=0$ for all but a single $\mu\in \Delta^{0+}(b)$. Therefore 
\begin{equation} \label{eq:interbound}
M_t\geq C_{d,J}\ell\bigl([u(\bfy)v(b)]_{\Delta^{0+}(b)}\bigr).
\end{equation}

Let $\bfkappa:=\sum_{i=1}^n\kappa_i(s_0)\bfe_i$. Here $\kappa_i(s_0)\neq 0$ for $1\leq i\leq n$.  
Let $D_1>0$ as in \Cref{cor:basic} for $E=\{\bfkappa\}$. Then
\begin{equation} \label{eq:D1}
    \norm{[u(\bfkappa)v(b)]_{\Delta^{0+}(b)}}\geq D_1\norm{v(b)}.
\end{equation}
Since $s_t\to s_0$, $\alpha_te^{-t}\to 0$, and $J$ is compact, $\bfy=\sum_{i=1}^n (\kappa_i+\epsilon_i(h))\bfe_i\to \bfkappa$ uniformly for $\eta\in J$. We pick $T\geq 0$ such that for all $t\geq T$ and all $\eta\in J$,
\begin{equation} \label{eq:approx1}
\norm{u(\bfy)-u(\bfkappa)}\leq D_1/2,
\end{equation}
the operator norm $\norm{u(\cdot)}$ is with respect to the sup-norm on $V$. 

So for all $t\geq T$, in view of \Cref{subsec:supnorm},
\begin{align*}
&\abs{\ell\bigl([u(\bfy)v(b)-u(\bfkappa)v(b)]_{\Delta^{0+}(b)}\bigr)}
\leq \norm{u(\bfy)v(b)-u(\bfkappa)v(b)}
\leq (D_1/2)\norm{v(b)}.
\end{align*}

Now we choose a coordinate linear functional $\ell$ as above so that 
\[
\ell\bigl([u(\bfkappa)v(b)]_{\Delta^{0+}(b)}\bigr)=\norm{[u(\bfkappa)v(
b)]_{\Delta^{0+}(b)}}\geq D_1\norm{v(b)},
\]
by \eqref{eq:D1}. Therefore
\begin{align*}
\ell\bigl([u(\bfy)v(b)]_{\Delta^{0+}(b)}\bigr)
 \geq \ell\bigl([u(\bfkappa)v(b)]_{\Delta^{0+}(b)}\bigr) - (D_1/2)\norm{v(b)}
\geq (D_1/2)\norm{v(b)}.
\end{align*}

Now we choose $b\in B$ such that $\norm{v}=\norm{v(b)}$. Then by \eqref{eq:interbound} we conclude that for all $t\geq T$,
\begin{align*}
M_t \geq C_{d,J}(D_1/2)\norm{v}.
\end{align*}
\end{proof}

\begin{proposition} \label{prop:main-bdd}
Let the notation be as in \Cref{prop:main-basic}.
Suppose there exists a sequence $t_i\to\infty$ and 
a sequence $v_i\to v$ in $V$ and $M>0$ such that 
\[
\sup_{\eta\in J} \norm{a_{t_i}u(R(\alpha_{t_i}e^{-t_i}\eta))v_i}\leq M, \,\forall i.
\]
Then the following statements hold:
\begin{enumerate} 
\item \label{itm:main-infty} If $\lim_{i\to\infty}\alpha_{t_i}=\infty$, then $v$ is fixed by $G_{n_0}$, where $n_0$ is as in \eqref{eq:n0xi}.  
 \item \label{itm:main-1} Suppose that $\{a_t\}_t$ is non-uniform, see~\eqref{eq:non-uni}. 
Then $v$ is fixed by $\Po_{n_0}$, see~\eqref{eq:Q}.
\end{enumerate}
\end{proposition}

\begin{proof}
We may assume that $v\neq 0$. 
Since $v=\sum_{b\in B} v(b)$, it is enough to prove that for any $b\in B$ such that $v(b)\neq 0$, the following statements hold:
\begin{enumerate} 
\item If $\lim_{i\to\infty}\alpha_{t_i}=\infty$, then $v(b)$ is fixed by $G_{n_0}$. 
\item If $j<n$ and $\lim_{i\to\infty} \xi_j(t_i)=\infty$, then $v(b)$ is fixed by $G_{n_0}\cap\Po$. 
\end{enumerate} 

Let $\bfkappa=(\kappa_1(s_0),\ldots,\kappa_n(s_0))\in (\R\setminus\{0\})^n$. Let $\mu\in \Lambda_{u(\bfkappa)v}$; that is, $\mu\in \Delta$ and $[u(\bfkappa)v(b)]_\mu\neq 0$. Let $\ell$ be a coordinate linear functional on $V$ such that 
\[
\ell\bigl([u(\bfkappa)v(b)]_\mu\bigr)=\norm{[u(\bfkappa)v(b)]_{\mu}}>0.
\]
Let the notation be as in the proof of \Cref{prop:main-basic}. Then 
\[
\lim_{i\to\infty} [u(\bfy)v_i(b)]_\mu= [u(\bfkappa)v(b)]_\mu.
\]
Hence for all large $i$, we have 
\[
\ell\bigl([u(\bfy)v_i(b)]_\mu\bigr)\geq (1/2)\ell\bigl([u(\bfkappa)v(b)]_\mu\bigr)=(1/2)\norm{[u(\bfkappa)v(b)]_{\mu}}.
\]
By the condition of the proposition and \eqref{eq:imp-bound}, we get
\begin{align*}  
    M\geq C_{d,J}\alpha_{t_i}^{\mu(H_C)-b}e^{\mu(\HA)-(\mu(H_C)-b))t_i}(1/2) \norm{[u(\bfkappa)v(b)]_{\mu}}.
\end{align*}
Therefore
\begin{equation} \label{eq:Mtimu}
\limsup_{i\to\infty}  \alpha_{t_i}^{\mu(H_C)-b}e^{\mu(\HA)-(\mu(H_C)-b)t_i}<\infty,\, \forall \mu\in \Lambda_{u(\bfkappa)v(b)}. 
\end{equation}

\subsubsection*{Uniform Case}
We suppose that $\alpha_{t_i}\to\infty$ and $\{\HA-tH_n:t\geq 0\}$ is bounded in $\la{h}$. Let $\mu\in S_n(\bfkappa,v(b))$. Then by \eqref{eq:Mtimu}, 
\[
\limsup_{i\to\infty}  \alpha_{t_i}^{\mu(H_C)-b}e^{(\mu(H_n)-(\mu(H_C)-b))t_i}<\infty.
\]
We have $\mu(H_n)-(\mu(H_C)-b))\geq 0$ and $\mu(H_C)-b\geq 0$, and $\alpha_{t_i}\to\infty$.  
Therefore $\mu(H_n)-(\mu(H_C)-b)=0$ and $\mu(H_C)-b=0$. Therefore by \Cref{thm:basic_lemma}(\ref{itm:vGfixed}), $v(b)$ is fixed by $G$. This completes the proof of the proposition when $\{a_t\}$ is uniform.

\subsubsection*{Non-uniform case} Now suppose that $\{a_t\}_t$ is non-uniform. Then by \eqref{eq:non-uni},
there exists $1\leq j<n$ such that 
\[
\limsup_{i\to\infty} \xi_j(t_i)=\infty.
\]
Then $j\leq n_0\leq n$. Let $\mu\in S(\bfkappa,v(b))$. Then $\mu(H_C)-b\geq 0$ and
\begin{equation} \label{eq:HAmu+}
 \mu(H_\ell)-(\mu(H_C)-b)\geq 0, \,\forall 1\leq \ell\leq n.
\end{equation}
Therefore by \eqref{eq:Mtimu}, since $\alpha_{t_i}\geq 1$, we get
\begin{gather} \label{eq:Mtimu2a}
  \limsup_{i\to\infty} \mu(\HA)-t_i(\mu(H_C)-b) <\infty.
\end{gather}

Hence by \eqref{eq:HAxi}, \eqref{eq:HAmu+}, and \eqref{eq:Mtimu2a}
\begin{equation}
\label{eq:Mtimu2}
0\leq \limsup_{i\to\infty} \xi_\ell(t_i)\bigl(\mu(H_\ell)-(\mu(H_C)-b)\bigr)<\infty,\, \forall 1\leq \ell\leq n.
\end{equation}
Therefore, for $k=j,n_0$, since we have $\limsup_{i\to\infty} \xi_k(t_i)=\infty$, we get $\mu(H_k)-(\mu(H_C)-b)=0$; and this holds for all $\mu\in S(\bfkappa,v(b))$. Therefore by \Cref{thm:basic_lemma}(\ref{itm:Qn0fixed}), $v(b)$ is fixed by $\Po_{n_0}$. 

By  \eqref{eq:HAxi} and \ref{eq:Mtimu2}, $\mu(\HA)-t(\mu(H_C)-b)\geq 0$ for all $t\geq 0$ and $\mu\in S(\bfkappa,v(b))$. Therefore by \eqref{eq:Mtimu}, we have 
\begin{equation} \label{eq:alphamu}
\limsup_{i\to\infty} \alpha_{t_i}^{\mu(H_C)-b}<\infty.
\end{equation}
Since $\mu(H_C)-b\geq 0$, if $\alpha_{t_i}\to\infty$, then $\mu(H_C)-b=0$. Therefore by \Cref{thm:basic_lemma}(\ref{itm:Gn0fixed}), we have that $v(b)$ is fixed by $G_{n_0}$. 
\end{proof}

\section{Description of limiting distributions of expanding translates} \label{sec:Ratner}
 For a topological space $X$, let $\cP(X)$ denote the space of Borel probability measures on $X$ with weak-$\ast$ topology. If a topological group $L$ acts continuously on $X$, then for any $g\in L$ and any $\mu\in \cP(X)$, we define $g\mu\in\cP(X)$ by 
 \[
 \int_X f\, d(g\mu):=\int_X f(gy)\,d\mu(y)\text{ for all $f\in C_c(X)$}.
 \]
 For any $y\in X$, we denote the Dirac mass at $y$ by $\delta_y$. Then $g\delta_y=\delta_{gy}$ for all $g\in L$. 
 
We continue with the notation of \Cref{sec:curves}, especially the notation~\ref{not:alphat}, where we suppose that $\phi$ is ordered regular at $s_0$, and choose $s_t\to s_0$ as $t\to\infty$, etc. 

For convenience of notation, in view of \eqref{eq:n0xi}, we will further assume that $\xi_i=0$ for all $n_0<i\leq n$; that is, $\{a_t\}_t\subset G_{n_0}$. 

Let $x\in L/\Lambda$. For each $t\geq 0$, let $g_t\in L$ be such that $g_t\to e$ as $t\to\infty$; later we will impose additional condition \eqref{eq:gt} on $\{g_t\}$.  Let $J\subset (0,\infty)$ be a compact interval of length $\abs{J}>0$. For any $t\geq 0$, we define
 \begin{equation} \label{eq:mut1}
     \mu_t=\abs{J}^{-1}\int_{\eta\in J} a_tu(R_{s_t}(\alpha_te^{-t}\eta))g_t\delta_{x}\,d\eta\in \cP(L/\Lambda).
 \end{equation}
For notational convenience, as done earlier, in this section also we will drop the subscript $s_t$ from $R_{s_t}(\cdot)$. 
 
\subsubsection{Definition and countability of $\cH_{\Lambda}$} \label{subsec:cH} Let $x_0=e\Lambda\in L/\Lambda$. Let $\cH=\cH_{\Lambda}$ denote the collection of all closed connected Lie subgroups $F$ of $L$ such that $Fx_0\cong F/(F\cap\Lambda)$ is closed and admits an $F$-invariant probability measure, say $\lambda_F$, and $\lambda_F$ is ergodic for the action of some subgroup generated by $\Ad_L$-unipotent subgroup of $F$. We observe that $F\in\cH$ if and only if $\gamma F\gamma\inv\in \cH$ for some $\gamma\in\Lambda$. 
We recall an important result that the collection $\cH$ is countable, see \cite{Shah-MathAnn91,Ratner1991-Measure,Dani+Margulis+1993}.

 Since $G_{n_0}$ is generated by unipotent subgroups of $G$, by Ratner's orbit closure theorem~\cite{Ratner1991-Closure}, $\cl{G_{n_0}x}$ is a finite volume homogeneous space, admitting a unique $G_{n_0}$-invariant probability measure denoted by $\mu_{\cl{G_{n_0}x}}$. In other words, if $x=gx_0$ for some $g\in L$, then there exists $H\in\cH$ such that $\cl{G_{n_0}x}=gHx_0$ and $\mu_{\cl{G_{n_0}x}}=g\cdot\lambda_H$. 
 
We will further assume that the family $\{g_t\}$ is such that
\begin{equation} \label{eq:gt}
\cl{G_{n_0}g_tx}\subset g_t\cl{G_{n_0}x},\, \forall t.
\end{equation}

 We note that $\{a_t\}_t\cup u(\R^n)\subset G_{n_0}$. So for each $t$, $\mu_t$ is concentrated on $\cl{G_{n_0}g_tx}\subset g_t\cdot \cl{G_{n_0}x}$, and since $g_t\to e$, every weak-$\ast$ limit of $\mu_{t}$ in $\cP(L/\Lambda)$ is concentrated on $\cl{G_{n_0}x}$.

\begin{remark} \label{rem:uniform} 
If $\alpha_t=1$, $s_t=s_0$, $g_t=e$, and $\HA=tH_n$ for all $t$, then the description of $\lim_{t\to\infty}\mu_t$ was obtained in \cite{Shah2018} by a different method. In the remaining cases, we will describe the weak-$\ast$ limit of $\mu_t$ as $t\to\infty$ under various conditions. 
\end{remark}

\begin{proposition} \label{prop:NonGen-s}
 Suppose that $\{a_t\}_t$ is non-uniform. Let $g\in L$ be such that $x=gx_0$. Suppose that for any $F\in \cH$, if $ g\inv \Po_{n_0}g\subset N_L(F)$, then $g\inv G_{n_0}g\subset N_L(F)$. Then  $\lim_{t\to\infty} \mu_t=\mu_{\cl{G_{n_0}x}}$. 
\end{proposition}

We note that if the condition of the \Cref{prop:NonGen-s} is satisfied for some $g\in L$ such that $x=gx_0$, then for any $\gamma\in\Lambda$, the same condition holds for $g\gamma$ in place of $g$. 

\begin{proposition} \label{prop:main-1}
Suppose that $\lim_{t\to\infty}\alpha_{t}=\infty$. Then 
   $\lim_{t\to\infty} \mu_t=\mu_{\cl{G_{n_0}x}}$.
\end{proposition}

\begin{proposition} \label{prop:main-2}
Suppose that $\{a_t\}_t$ is non-uniform, $\alpha_t=1$, $s_t=s_0$, and $g_t=e$ for all $t$. Then $\lim_{t\to\infty}\mu_t=\mu$, and 
\[
\mu=\abs{J}\inv\int_{\eta\in J} 
\exp((\log \eta)H_{n_0})w(\kappa_n(s_0))\mu_{\cl{\Po_{n_0}x}}\, d\eta,
\]
where $\kappa_n(s_0)$ is as in \eqref{eq:epsiloni} and $w(\kappa)$ is defined in \eqref{eq:wxi}.
\end{proposition}
 
The rest of the section is devoted to proving the above three propositions. In the course of the proof we will first derive \Cref{prop:NonGen-s} and \Cref{prop:main-1}. The \Cref{prop:NonGen-s} will lead to \Cref{thm:co-countable} and \Cref{thm:smooth}.
The \Cref{prop:main-1} will yield \Cref{thm:slow-shrink}, and also for provide an alternative proof of \Cref{thm:smooth} when $\phi$ is $C^k$ for some integer $k>n+\limsup_{t\to\infty} t\inv r_1(t)$.

The proof of \Cref{prop:main-2} is technically more involved. It will be used only in the proof of \Cref{thm:main}. So the \Cref{subsec:main-2}, which provides proofs of these two results, can be skipped on the first reading of this article.

\subsubsection{Change of base point from $x_0$ to $x$} \label{subsec:basepoint} Let $g\in L$ such that $x=gx_0$. Then the stabilizer of $x$ in $L$ is $\Lambda_x:=g\Lambda g\inv$. So replacing $\Lambda$ by $\Lambda_x$, without loss of generality we will now assume that $x=x_0\in L/\Lambda$. 

\subsection{Polynomiality and non-divergence}
Consider any finite dimensional linear representation $V$ of $G=\SL(n+1,\R)$. For any $v\in V$ and any linear functional $\ell$ on $V$, and any $t\in\cT$, the map 
\begin{equation} \label{eq:polynom}
     \eta\mapsto \ell(a_tu(R(\alpha_t e^{-t}\eta))g_tv)
 \end{equation} 
 is a polynomial map (in $\eta$) of a degree which is bounded independent of $t$ and $v$ and $\ell$.

 Therefore in view of the Dani-Margulis non-divergence criterion~\cite[Theorem~2.2]{Shah1996a} combined with \Cref{prop:main-basic}, we conclude that $\{\mu_t:t\geq 1\}$ is relatively compact in $\cP(L/\Lambda)$. Therefore for any sequence $t_i\to \infty$, after passing to a subsequence we have that $\mu_{t_i}$ converges to some $\mu$ in $\cP(L/\Lambda)$. Our goal is to describe the limit measure $\mu$. 

\subsection{Invariance under a unipotent subgroup} 
Let $\{\bfe_1,\ldots,\bfe_n\}$ denote the standard basis of $\R^n$. Let $U_1=\{u(\zeta\bfe_1):\zeta\in\R\}$.

We pick a sequence $t_i\to\infty$ such that $\mu_{t_i}\to \mu$ in $\cP(L/\Lambda)$ as $i\to\infty$. 

\begin{claim} \label{claim:u1-inv}
$\mu$ is $U_1$-invariant.
\end{claim}

\begin{proof}
 Since $R(0)=0$, for $h\in[-1,1]$, we express 
\begin{equation} \label{eq:Rh1}
  R(h)=hR'(0)+h^2R_1(h)=h\kappa_1\bfe_1+h^2R_1(h), 
\end{equation}
where in view of \eqref{eq:epsiloni}, $R_1(h)$ is a polynomial in $h$. 

Fix $\zeta\in \R$. For $t\geq 0$, let $\tau=\alpha_t\inv e^{t}$, and for $\eta\in J$, let $h=\alpha_te^{-t}\eta=\tau\inv\eta$. Then
\begin{align}
    &u(\zeta\bfe_1)a_tu(R(h))=a_tu(\zeta e^{-(nt+r_1(t))}\bfe_1)u(R(h)) \notag\\
    &=a_t u(\zeta e^{-(nt+r_1(t))}\bfe_1+\kappa_1 h\bfe_1)u(h^2 R_1(h))\notag\\
    &=a_tu(\kappa_1 \tilde h \bfe_1)u(h^2 R_1(h))
    \text{, where $\tilde h:=h+\kappa_1\inv \zeta e^{-(nt+r_1(t))}$}\label{eq:htilde}\\
    &=a_tu(\Delta)u(R(\tilde h))
    \text{, where $\Delta:=h^2R_1(h)-\tilde{h}^2R_1(\tilde h)$}\label{def:Delta}\\
    &=[a_tu(\Delta)a_t^{-1}]\cdot a_tu(R(\tilde h))\notag\\
    &=:u(\Delta_t)\cdot a_tu(R(\tilde h))
    \text{, where $\norm{\Delta_t}\leq e^{(nt+r_1(t))}\norm{\Delta}$,} \label{eq:uzetamu}
\end{align}
where we think of $\Delta$ and $\Delta_t$ as functions of $t$ and $\eta$.  

To estimate $\norm{\Delta}$, we apply the mean value theorem to each coordinate function of $h^2R_1(h)$ separately. As we are considering the sup-norm on $\R^n$, 
\[
\norm{\Delta}=\norm{h^2R_1(h)-\tilde{h}^2R_1(\tilde h)}
=\abs{\tilde{h}-h}\cdot \norm{2\bar{h}R_1(\bar h)+{\bar h}^2R_1'(\bar{h})}
\]
for some $\bar{h}$ between $h$ and $\tilde{h}$. We note that $R_1(\xi)$ is a polynomial in $\xi$, the coefficients of $R_1$ continuously depend on $s_t$ and $s_t\to s_0$ as $t\to\infty$. Therefore by \eqref{eq:htilde} and \eqref{def:Delta}, uniformly for $\eta$ in a compact neighbourhood of $J$,
\[
\norm{\Delta}=O(\abs{h-\tilde{h}}h)=O(e^{-(nt+r_1(t))}\tau\inv),
\]
and hence by \eqref{eq:uzetamu},
\begin{equation} \label{eq:Deltat}
\norm{\Delta_t}= e^{(nt+r_1(t))}\norm{\Delta}=O(\tau\inv)
\end{equation}
as $t\to\infty$. Also $\tau\inv=\alpha_t e^{-t}\to 0$ as $t\to\infty$ by the definition of $\alpha_t$.  We express 
\[
\tilde h=\tau\inv\eta+\kappa_1\inv \zeta e^{-(nt+r_1(t))}=\tau\inv \tilde\eta,
\]
where 
\[
\tilde\eta:=\eta+\kappa_1\inv \zeta e^{-(nt+r_1(t))}\tau=\eta+O(e^{-(nt+r_1(t)-t)})
\]
as $\tau=\alpha_t\inv e^{t}$ and $\alpha_t\geq 1$. 

Therefore by
\eqref{eq:mut1} and \eqref{eq:uzetamu}, for any $f\in C_c(L/\Lambda)$
\begin{align*}
&\int f(u(\zeta)y)\,d\mu_t(y)\notag\\
&=(1/\abs{J})\int_{\eta\in J} f(u(\zeta)a_tu(R(\tau\inv \eta))g_tx)\, d\eta \notag\\
&=(1/\abs{J})\int_{\tilde\eta\in J+\kappa_1\inv \zeta e^{-(nt+r_1(t))}\tau} f(u(\Delta_t)a_tu(R(\tau\inv\tilde\eta))g_t x)\,d\tilde\eta \notag\\
&=(1/\abs{J})\int_{\tilde\eta\in J} f(u(\Delta_t)a_tu(R(\tau\inv\tilde\eta))g_tx)\,d\tilde\eta+O(e^{-(nt+r_1(t)-t)}\norm{f}_\infty).
\end{align*}
By \eqref{eq:Deltat}, $u(\Delta_t)\to e$, uniformly for $\tilde\eta\in J$, as $t\to\infty$. Taking the limits for $t=t_i$ as $i\to\infty$ on both sides, since $\mu_{t_i}\to\mu$ and $f$ is bounded and uniformly continuous on $L/\Lambda$, we get $u(\zeta)\mu=\mu$.
\end{proof}

\subsection{Consequences of Ratner's theorem}  \label{subsec:SUF}
In view of \Cref{claim:u1-inv}, let $U$ be any subgroup of $G_{n_0}$ containing $U_1$ such that $U$ is generated by unipotent one-parameter subgroups of $G$ and the action of $U$ on $L/\Lambda$ preserves $\mu$.

Let $F\in \cH=\cH_{\Lambda}$, see \Cref{subsec:cH}.  
\begin{equation} \label{def:NUF}
 N(U,F)=\{g\in L: UgF\subset gF\}=\{g\in L: g^{-1} Ug\subset F\},   
\end{equation}
and we treat $U$ as a subgroup of $L$ via its embedding under the homomorphism from $G$ to $L$. Note that
\begin{equation} \label{eq:Nufnorf}
N(U,F)N_L(F)=N(U,F) \text{ and } N_L(U)N(U,F)=N(U,F).
\end{equation}
where $N_L(F)$ denotes the normalizer of $F$ in $L$. 
Define
\[
S(U,F):=\bigcup_{F_1\in \cH,\, F_1\subset F,\,\dim(F_1)<\dim(F)} N(F_1,U).
\]
We note that for any $\gamma\in\Lambda$, 
\begin{equation} \label{eq:NUFgamma}
N(U,F)\gamma\inv = N(U,\gamma F\gamma\inv). 
\end{equation}
And by \cite[(2.3) Theorem]{Shah-MathAnn91} there exists $F_1\in\cH$ such that  $F_1\subset F\cap \gamma F\gamma\inv$ and contains all unipotent one-parameter subgroups of $L$ contained in $F\cap \gamma F\gamma\inv$. 
In particular if $\gamma\in \Lambda\setminus N_L(F)$, then $\dim F_1<\dim F$, so
\begin{equation} \label{eq:Nufsuf}
    N(U,F)\gamma\inv \cap N(U,F)=N(U,\gamma F\gamma\inv\cap F)=N(U,F_1)\subset S(U,F).
\end{equation}
We also note that 
\begin{equation} \label{eq:nsx}
N(U,F)x_0\setminus S(U,F)x_0=(N(U,F)\setminus S(U,F))x_0.
\end{equation}

\subsubsection{Choice of $F\in\cH$} \label{subsec:ChoiceF}
Due to Ratner's theorem~\cite{Ratner1991-Measure} describing $U$-ergodic invariant measures in $\cP(L/\Lambda)$, almost every $U$-ergodic component of $\mu$ is of the form $g\lambda_F$ for some $F\in \cH$ and $g\in N(U,F)$. Hence $\mu$ is concentrated on the union of $N(U,F)x_0$ over $F\in\cH$. Since $\cH$ is countable, see \Cref{subsec:cH}, we can pick $F\in\cH$ to be of the smallest possible dimension such that $\mu(N(U,F)x_0)>0$. Then by the minimality of $\dim F$ and the countability of $\cH$, we conclude that $\mu(S(U,F)x_0)=0$.
Therefore, by \eqref{eq:nsx}, almost every $U$-ergodic component of $\mu$ restricted to $N(U,F)x_0$ is concentrated on $(N(U,F)\setminus S(U,F))x_0$, and hence it is of the form $g\lambda_F$ for some $g\in N(U,F)\setminus S(U,F)$, and in fact, such a $g$ is unique modulo $F(N_L(F)\cap\Lambda)$. 

\subsection{Linearization technique} \label{subsec:lintech}
Consider the finite dimensional vector space $V=\bigoplus_{d=1}^{\dim\Lie(L)}\wedge^d \Lie(L)$ and the action 
of $G$ on $V$ via the $\bigoplus_d\wedge^d(\Ad_L)$-action of $L$ on $V$. For any $F\in \cH$, pick and fix a nonzero $p_F\in \wedge^{\dim F} \Lie(F)\subset V$. Then the stabilizer of $p_F$ in $L$ is 
\begin{equation} \label{eq:stabil}
N^1(F):=\{g\in N_L(F): \det(\Ad g\mid_{\Lie(F)})=1\}.
\end{equation}
Let $d=\dim F$. Define the linear subspace 
\[
\cA=\{v\in \wedge^d\Lie(L)\subset V : v\wedge Z=0 \in \wedge^{d+1}\Lie(L), \,\forall Z\in \Lie(U)\}.
\]
Then 
\begin{equation} \label{eq:V1}
    \{g\in L: gp_F\in \cA\}=\{g\in L: \Lie(U)\subset \Ad(g)(\Lie(F))\}=N(U,F).
\end{equation}

\subsubsection{Discreteness of $\Lambda p_F$} \label{subsubsec:discrete} By \cite[Theorem~3.4]{Dani+Margulis+1993}, $\Lambda p_F$ is a discrete subset of $V$.
In particular, the inverse image of $\Lambda p_F$ under the continuous map 
$L\ni g\mapsto g\cdot p_F\in V$ is closed; meaning the set $\Lambda N^1(F)$ is closed in $L$. So  $[\Lambda N^1(F)]^{-1}=N^1(F)\Lambda$ is closed in $L$. So $N^1(F)x_0$ is closed in $L/\Lambda$. 

Since $\mu(N(U,F)x_0)>0$ and $\mu(S(U,F)x_0)=0$, by \eqref{eq:nsx} there exists a compact set $C\subset N(U,F)\setminus S(U,F)$ such that $\mu(Cx_0)>0$. 
For any neighbourhood $\Omega$ of $Cx_0$ in $L/\Lambda$, let \[
J(\Omega,i)=\{\eta\in J:a_{t_i}u(R(\tau_i\inv\eta))g_{i}x_0\in \Omega\}\text{, where $\tau_i:=\alpha_{t_i}^{-1} e^{t_i}$ and $g_i=g_{t_i}$.}
\]
Then by \eqref{eq:mut1}, we get 
\[\liminf_{i\to\infty} \abs{J}\inv\abs{(J(\Omega,i))}=\liminf_{i\to\infty}\mu_{t_i}(\Omega)\geq \mu(Cx_0)>0.
\]
Hence
\[
\abs{J(\Omega,i)}\geq \abs{J}\mu(Cx_0)/2>0 
\text{, for all large $i$.}
\]
Therefore, due to the bounded degree polynomiality of the maps \eqref{eq:polynom},  via the linearization technique \cite{Dani+Margulis+1993}
(see \cite[Proposition~4.7]{Shah1996a}), we obtain a compact set $D\subset \cA$ such that for any decreasing sequence $\{\Psi_i\}$ of relatively compact open subsets in $V$ such that 
\begin{equation} \label{eq:Psii}
\cap_{i=1}^\infty \Psi_i=D,
\end{equation}
and after passing to a subsequence, for each $i\in\N$, there exists $\gamma_i\in \Lambda$ such that
\begin{equation} \label{eq:iPsi}
  a_{t_i}u(R(\tau_i\inv J))g_{i}\gamma_ip_F\subset \Psi_i\subset \Psi_1.
\end{equation}

\subsection{Algebraic consequences of the dynamical boundedness}

Since $\cl{\Psi_1}$ is compact, by \Cref{prop:main-basic} we conclude that $\{g_i\gamma_ip_F:i\in\N\}$ is a bounded subset of $V$. Since $g_i\to e$ and $\Lambda p_F$ is a discrete subset of $V$, by passing to a subsequence, we may assume that there exists $\gamma\in\Lambda$ such that $\gamma_ip_F=\gamma p_F$ for all $i$. Now by \eqref{eq:NUFgamma},
\begin{equation} \label{eq:gammaF}
N(U,\gamma F\gamma\inv)\gamma p_F=N(U,F)p_F, \text{ and } \gamma p_F=p_{\gamma F\gamma\inv}.
\end{equation}
So without loss of generality, replacing $F$ by $\gamma F\gamma\inv$,  we may assume that for all $i$,
\begin{equation}  \label{eq:fixed-bounded}
    a_{t_i}u(R(\tau_i^{-1}J))g_{i}p_F\subset \Psi_i\subset\Psi_1. 
\end{equation}

Since $\Psi_1$ is bounded, by \Cref{prop:main-bdd} applied to $v_{t_i}=g_ip_F$, we deduce that $p_F$ is fixed by $\Po_{n_0}=G_{n_0}\cap \Po$. Thus
\begin{equation} \label{eq:QinNF}
    \Po_{n_0}p_F=p_F \text{ and }\Po_{n_0}\subset N^1(F)\subset N_L(F).
\end{equation}

\begin{claim} \label{claim:Gn0-fixed}
If $G_{n_0}$ fixes $p_F$, then $\mu=\mu_{\cl{G_{n_0}x_0}}$.
\end{claim} 

\begin{proof}
Suppose $G_{n_0}$ fixes $p_F$. Then $G_{n_0}\subset N^1(F)$. Since $N^1(F)x_0$ is closed in $L/\Lambda$, $\cl{G_{n_0}x_0}\subset N^1(F)x_0$. So by \eqref{eq:gt}, since $x=x_0$ and $g_i:=g_{t_i}$,
\[
N^1(F)x_0\supset \cl{G_{n_0}x_0}\supset g_i\inv\cl{G_{n_0}g_ix_0}=\cl{(g_i\inv G_{n_0}g_i)x_0}.
\]
Hence $g_i\inv G_{n_0}g_i\subset N^1(F)$. Therefore we get $g_ip_F$ is fixed by $G_{n_0}$. Therefore by \eqref{eq:fixed-bounded}, $p_F\in \Psi_i$ for all $i$. So by \eqref{eq:Psii} $p_F\in D\subset\cA$. Therefore by \eqref{eq:V1} $e\in N(U,F)$. So $U\subset F$. So $U_1\subset F$. So $F\cap G_{n_0}$ is normalized by $G_{n_0}$ and contains $U_1$. The smallest normal subgroup of $G_{n_0}$ containing $U_1$ is $G_{n_0}$. Hence $G_{n_0}\subset F$.  Therefore, for all $t\geq 0$, since $a_tu(\R^n)\subset G_{n_0}$, $Fx_0$ is closed, by \eqref{eq:gt},
\[
a_tu(R(\tau\inv J))g_tx_0\subset \cl{G_{n_0}g_tx_0}\subset g_t\cl{G_{n_0}x_0}\subset g_t Fx_0.
\]
Therefore $\mu_t$ is concentrated on $g_tFx_0$ for all $t$. Since $g_t\to e$, $\mu$ is concentrated on $Fx_0$. So by \Cref{subsec:ChoiceF} every $U$-ergodic component of $\mu$ is of the form $g\lambda_F$ for some $g\in N(U,F)$. Therefore we conclude that $\mu=\lambda_F$. On the other hand, we know that $\mu$ in concentrated on $\cl{G_{n_0}x_0}$, and $G_{n_0}\subset F$. This proves that $\cl{G_{n_0}x_0}=Fx_0$. Therefore $\mu=\lambda_F=\mu_{\cl{G_{n_0}x_0}}$. 
\end{proof}

\subsection{Completion of proof of Proposition~\ref{prop:NonGen-s}} Since we assumed that $x=x_0$, the condition of \Cref{prop:NonGen-s} is satisfied for $g=e$; see the note after the statement of the proposition. By \eqref{eq:QinNF}, we have $\Po_{n_0}\subset N_L(F)$. So by the condition of the proposition, $G_{n_0}\subset N_L(F)$. Since $G_{n_0}$ is generated by unipotent one parameter subgroups, we conclude that $G_{n_0}\subset N^1(F)$. So $G_{n_0}$ fixes $p_F$. So by \Cref{claim:Gn0-fixed}, $\mu=\mu_{\cl{G_{n_0}x_0}}$. This completes the proof of \Cref{prop:NonGen-s}. \qed

\subsection{Completion of proof of Proposition~\ref{prop:main-1}} 
Now suppose $\alpha_t\to\infty$, then by \eqref{eq:fixed-bounded} and \Cref{prop:main-bdd}(\ref{itm:main-infty}) we obtain that $p_F$ is fixed by $G_{n_0}$. So due to \Cref{claim:Gn0-fixed}, the proof of \Cref{prop:main-1} is complete. \qed

The remaining section is devoted to completing the proof of \Cref{prop:main-2}, which is needed only for the proof of \Cref{thm:main} and not used in the proof of any of the other theorems stated in the introduction. So one may choose to skip the rest of this section on the initial reading, and jump straight to \Cref{sec:slow}.

\subsection{Completion of proof of Proposition~\ref{prop:main-2}} \label{subsec:main-2} In this case we assume that $\{a_t\}_t$ is  non-uniform, $\alpha_t=1$, $s_{t}=s_0$, and $g_t=e$ for all $t$.

\begin{lemma} \label{lem:Qfixed} Let $\eta>0$. As $t\to\infty$, 
\begin{equation} \label{eq:limit}
	    \lim_{t\to\infty} a_tu(R(e^{-t}\eta))\Po_{n_0}
	    =\exp((\log\eta)H_{n_0})w(\kappa_n(s_0))\Po_{n_0}, 
	\end{equation}
	in $G/\Po_{n_0}$, where  $w(\kappa)$ is defined in $\eqref{eq:wxi}$. 
\end{lemma}

\begin{proof} 
Let $\bfe_n=\trn{(0,\ldots,0,1)}\in\R^{n+1}$. Then $\Po_{n_0}=\{g\in G_{n_0}:g\bfe_n=\bfe_n\}$. Since all elements involved in the limit are in $G_{n_0}$, it is enough to show that 
\begin{equation} \label{eq:limit-en}
	    \lim_{t\to\infty} a_tu(R(e^{-t}\eta))\bfe_n
	    =\exp((\log\eta)H_{n_0})w(\kappa_n(s_0))\bfe_n.
	\end{equation}
	
Given $t$ and $\eta\in J$, let $h=\eta e^{-t}$. Then by definition of $R(\cdot)$,
\[
u(R(e^{-t}\eta))\bfe_n=(\kappa_n(s_0)+\epsilon_n(h))h^n \bfe_0+\bfe_n.
\]
Therefore
\[
a_t u(R(e^{-t}\eta))\bfe_n=(\kappa_n(s_0)+\epsilon_n(h))h^n e^{nt}\bfe_0+ 
    e^{-r_n(t)}\bfe_n,
    \]
where $r_n(t)=0$ if $n_0<n$ and $\lim_{t\to\infty} r_n(t)\to\infty$ if $n_0=n$. 

We have $w(\kappa_n)=u(\kappa_n\bfe_n)$ if $n_0<n$, and $w(\kappa_n)=\sigma(\kappa_n)$ if $n_0=n$. So
\[
\exp((\log \eta)H_n)w(\kappa_n(s_0))\bfe_n=
\kappa_n(s_0)\eta^n \bfe_0+
\begin{cases}
\bfe_n & \text{ if } n<n_0\\
0 & \text{ if } n=n_0.
\end{cases}
\]
Now \eqref{eq:limit-en} follows because $h^ne^{nt}=\eta^n$ and $\epsilon_n(h)=o(h)=o(e^{-t})$. 
\end{proof}

 Coming back to the proof of \Cref{prop:main-2}, since $\Po_{n_0}$ fixes $p_F$ (see~\eqref{eq:QinNF}), by  \eqref{eq:limit}, for all $\eta>0$ we have
\begin{align}
\label{eq:imp-limit}
    \lim_{t\to\infty} a_tu(R(e^{-t}\eta))p_F&=\exp((\log \eta)H_{n_0})w(\kappa_n(s_0))p_F\text{ in $V$}.
    \end{align}
We will write $\kappa_n=\kappa_n(s_0)$. Hence by \eqref{eq:Psii} and \eqref{eq:fixed-bounded},
$\exp((\log \eta)H_{n_0})w(\kappa_n)p_F\in D\subset\cA$ for each $\eta\in J$. Since the
map $\eta\mapsto \exp((\log \eta)H_{n_0})w(\kappa_n)p_F$ is analytic in $\eta$, we get $\exp((\log \eta)H_{n_0})w(\kappa_n)p_F\subset\cA$ for all $\eta>0$. So by \eqref{eq:V1}, 
we have
\begin{equation} \label{eq:sigmaNUF}
\exp((\log \eta)H_{n_0})w(\kappa_n)\in N(U,F), \quad \forall \eta>0.
\end{equation}
By definition, for any $g\in N(U,F)$, we have $F\supset g\inv Ug$. And by \eqref{eq:sigmaNUF} for $\eta=1$, we get $w(\kappa_n)\in N(U,F)$. Hence $F\supset w(\kappa_n)\inv U w(\kappa_n)$. Thus we gather the following algebraic information:
\begin{equation} \label{eq:inF}
    U_1\subset U,\, \Po_{n_0}\subset N^1(F) \text{, and } w(\kappa_n)\inv U w(\kappa_n)\subset F.
\end{equation}

\subsubsection{Analyzing the algebraic information}
Since $\Po_{n_0}\subset N^1(F)$, uniformly for all $\eta\in J$,
\[
\lim_{t\to\infty} a_{t}u(R(e^{-t}\eta))N^1(F)=\exp((\log \eta)H_{n_0})w(\kappa_n)N^1(F)
\]
in $L/N^1(F)$.  Since $N^1(F)x_0$ is closed we deduce that $\mu_t$ is concentrated on the set $a_{t}u(R(e^{-t}J))N^1(F)x_0$. Hence $\mu$ is concentrated on 
\[
\exp((\log J)H_{n_0})w(\kappa_n)N^1(F)x_0\subset N(U,F)x_0,
\]
by \eqref{eq:Nufnorf} and \eqref{eq:sigmaNUF}. 
Therefore by \Cref{subsec:ChoiceF}, almost every $U$-ergodic component of $\mu$ is of the form $b\lambda_F$ for some $b\in B$, where
\[
B:=\exp((\log J)H_{n_0})w(\kappa_n) N^1(F)\subset N(U,F).
\]
So we choose a probability measure $\tilde\nu$ on $B$ such that 
\begin{equation} \label{eq:tildenu}
\mu=\int_{b\in B} b\lambda_F \,d\tilde\nu(b).
\end{equation}

Our further analysis is based on the following algebraic observation. 

\begin{lemma} \label{lemma:U}
For any $\eta>0$ and $\kappa\neq 0$,
\begin{gather}
U_1\cdot \exp((\log \eta)H_{n_0})u(\kappa\bfe_n)\Po_{n_0}=\exp((\log \eta)H_{n_0})u(\kappa\bfe_n)\Po_{n_0}, \label{eq:U-n0}\\
\Po'\cdot \exp((\log \eta)H_n)\sigma(\kappa)\Po=\exp((\log \eta)H_n)\sigma(\kappa)\Po, \label{eq:U-n}
\end{gather}
where $U_1=\{u(\zeta\bfe_1):\zeta\in\R\}$ and 
\begin{gather}
\label{def:Qprime}
\Po':=\{g\in \SL(n+1,\R): g\bfe_0=\bfe_0\}=
\begin{bsmallmatrix}
1&\R^n\\
&\SL(n,\R)
\end{bsmallmatrix}.
\end{gather}
\end{lemma}
\begin{proof}
Observe that $\exp((\log \eta)H_n)$ normalizes $U_1$, elements of $U_1$ commutes with $u(\kappa\bfe_n)$, and $U_1\subset\Po_{n_0}$. Therefore \eqref{eq:U-n0} holds.

Now let
$\{\bfe_0,\ldots,\bfe_n\}$ denote the standard basis of $\R^{n+1}$. Then
\[
\Po=\begin{bsmallmatrix}
  \SL(n,\R)\\
  \R^n&1
\end{bsmallmatrix}=\{g\in \SL(n+1,\R):g\bfe_n=\bfe_n\}.
\]
Also $\sigma(\kappa)\bfe_n=\kappa\bfe_0$ and $\exp((\log \eta)H_n)\bfe_0=\eta^n\bfe_0$.
So both sides of \eqref{eq:U-n} applied to $\bfe_n$ are equal, hence \eqref{eq:U-n} holds.
\end{proof}

\subsubsection{Choice of \texorpdfstring{$U$}{U} preserving \texorpdfstring{$\mu$}{\textmu}} \label{sec:U}
Now we specify the choice of $U$ in \Cref{subsec:SUF} to be made from the beginning of the proof as follows: If $n_0<n$, then let $U=U_1$. If $n_0=n$, then since $U_1\subset \Po'$ and $U_1$ preserves $\mu$, we let $U$ to be the subgroup generated by all unipotent one-parameter subgroups of $\Po'$ that preserve $\mu$.

Due to the above choice of $U$, by \Cref{lemma:U} we have that for all $\eta>0$ and $\kappa\neq 0$,
\begin{equation} \label{eq:NUF1}
    U\exp((\log \eta)H_{n_0})w(\kappa)\Po_{n_0}= \exp((\log \eta)H_{n_0})w(\kappa)\Po_{n_0}. 
\end{equation}

\begin{claim} \label{claim:IfQinF2}
If $\Po_{n_0}\subset F$, then $\cl{\Po_{n_0}x_0}=Fx_0$.
\end{claim}

\begin{proof}
Since $\Po_{n_0}$ is generated by unipotent one-parameter subgroups, by Ratner's orbit closure theorem~\cite{Ratner1991-Closure}, there exists a closed Lie subgroup $F_1$ of $L$ such that $\cl{\Po_{n_0}x_0}=F_1x_0$; and in this case $F_1\in\cH$ \cite[Theorem~2.3]{Shah-MathAnn91}. Since $\Po_{n_0}\subset F$ and $Fx_0$ is closed, we have that $F_1\subset F$. 

For any compact neighbourhood $\Omega$ of $e$, by \eqref{eq:limit} we get that 
\[
\supp\mu_{t_i}\subset \Omega \cdot \exp((\log J)H_{n_0})w(\kappa_n(s_0))\cl{Q_{n_0}x_0}\text{, for all large $i$.}
\]
Therefore $\supp\mu\subset \Omega \cdot \exp((\log J)H_{n_0})w(\kappa_n(s_0))F_1x_0$. Therefore we conclude that 
\begin{equation} \label{eq:suppmu}
    \supp\mu \subset \exp((\log J)H_{n_0})w(\kappa_n(s_0))F_1x_0.
\end{equation}

Since $\Po_{n_0}\subset F_1$, for any  $g\in \exp((\log J)H_{n_0})w(\kappa_n(s_0))$, by \eqref{eq:NUF1}  $UgF_1\subset gF_1$; that is $g\inv U g\subset F_1$ and so $g\in N(U,F_1)$. Thus
\[
\exp((\log J)H_{n_0})w(\kappa_n(s_0))F_1\subset N(U,F_1)F_1=N(U,F_1).
\]
Therefore by \eqref{eq:suppmu} we have $\mu(N(U,F_1)x_0)=1$. Since $F_1\subset F$, if $\dim F_1<\dim F$, then $\mu(S(U,F)x_0)\geq \mu(N(U,F_1)x_0)>0$, which is a contradiction. Therefore $F_1=F$.
\end{proof}

Our goal is to show that $\Po_{n_0}\subset F$; later we will show that this information will be sufficient to complete the proof of the proposition. 

\subsubsection{Lifting the dynamics to \texorpdfstring{$L/(N_L(F)\cap\Lambda)$}{L/(N\textunderscore{L}(F)\textcap\textLambda)}}


Let
\begin{gather*}
    X=L/\Lambda,\  \tilde X= L/(N_L(F)\cap \Lambda),\ 
    x_0=[\Lambda]\in X  \text{ and } \tilde x_0=[N_L(F)\cap \Lambda]\in \tilde X.
\end{gather*}
Consider the natural quotient map 
\begin{equation} \label{eq:injective}
g\tilde x_0\to gx_0: \tilde X\to X.
\end{equation}
For any compact set $C\subset L$, this map restricted to the closed set $CN^1(F)\tilde x_0$ is a proper map. By \eqref{eq:Nufsuf}, for any $\gamma\in \Lambda$, if  $\gamma\not\in N_L(F)$, then $N(U,F)\gamma\cap N(U,F)\subset S(U,F)$. Therefore,
\begin{equation}
    \label{eq:inj}
    \text{the map $g\tilde x_0\to gx_0: (N(U,F)\setminus S(U,F))\tilde x_0\to X$ is injective.}
\end{equation}

Let $\tilde\lambda_F\in \cP(\tilde X)$ denote the unique $F$-invariant probability measure on $F\tilde x_0$. In view of  \eqref{eq:tildenu}, define the measure
\begin{equation}
    \label{eq:tildemu}
    \tilde\mu:=\int_{b\in B} g\tilde\lambda_F \,d\tilde\nu(b) \in \cP(\tilde X),
\end{equation}
where $B:=\exp((\log J)H_{n_0})w(\kappa_n) N^1(F)\subset N(U,F)$.

We observe that 
\begin{equation} \label{eq:SUFtilde}
N(U,F)=N(U,F)N_L(F) \text{ and } S(U,F)=S(U,F)(N_L(F)\cap \Lambda),
\end{equation}
because for any $\gamma\in N_L(F)\cap \Lambda$ and $F_1\in\cH$ with $F_1\subsetneq F$, we have $\gamma F_1 \gamma\in \cH$, $\gamma F_1 \gamma\inv \subsetneq F$ and $N(U,F_1)\gamma=N(U,\gamma F_1 \gamma\inv)\subset S(U,F)$.

\begin{claim} \label{claim:unique-tildemu}
$\tilde\mu$ is the unique Borel probability measure on $N(U,F)\tilde x_0$ which projects to $\mu$ on $X$.
\end{claim} 

\begin{proof}
By definition $\tilde\mu$ is concentrated on $N(U,F)\tilde x_0$. Since $\tilde\lambda_F$ projects to $\lambda_F$, the projection of $\tilde\mu$ on $X$ is $\mu$ due to  \eqref{eq:tildenu}.

To prove uniqueness, suppose $\tilde\mu'\in \cP(\tilde X)$ is such that it is concentrated on $N(U,F)\tilde x_0$ and projects to $\mu$ on $X$.  Since $\mu(S(U,F)x_0)=0$,
we get 
\[
\tilde\mu'(S(U,F)\tilde x_0)=0.
\]
Therefore by \eqref{eq:SUFtilde}, $\tilde\mu'$ is
concentrated on 
\[
N(U,F)\tilde x_0\setminus S(U,F)\tilde x_0=(N(U,F)\setminus S(U,F))\tilde x_0.
\]
By the same reason
$\tilde\mu$ is also concentrated on $(N(U,F)\setminus S(U,F))\tilde x_0$. Since
both the measures project to $\mu$ on $X$, by the injectivity of the map in
\eqref{eq:inj} we conclude that $\tilde\mu'=\tilde\mu$. This completes the
proof of the claim. 
\end{proof}

For $t\geq 1$, let 
\[
\tilde\mu_t=(1/\abs{J})\int_{\eta\in J} a_tu(R(e^{-t}\eta))\delta_{\tilde x_0}\,d\eta.
\]

\begin{claim} \label{claim:tildemut} We have that
$\tilde\mu_{t_i}\to \tilde\mu$ as $i\to\infty$ in $\cP(\tilde X)$. 
\end{claim}

\begin{proof}
For each $i$,  $\tilde\mu_{t_i}$ projects to $\mu_{t_i}$ on $X$ and it is concentrated on the closed set $B\tilde x_0=\exp((\log J)H_{n_0})w(\kappa_n) N^1(F)\tilde x_0$, restricted to which the map in \eqref{eq:injective} is proper. 
Since $\mu_{t_i}\to \mu$ as $i\to\infty$, after passing to a subsequence, we conclude that $\tilde\mu_{t_i}$ converges
to a probability measure, say $\tilde \mu'$ on $\tilde X$, and $\tilde \mu'$ 
projects to the measure $\mu$ on $X$. Since $\tilde \mu'$ is concentrated on $B \tilde x_0\subset N(U,F)\tilde x_0$, it follows from \Cref{claim:unique-tildemu} that $\tilde \mu'=\tilde\mu$. 
\end{proof}

\subsubsection{Projection modulo $F$}
Since $\tilde\mu_{t_i}\to \tilde\mu$ in the space of probability measures $\tilde X$, given any $\epsilon>0$, there exists a compact set $K\subset \tilde X$, such that $\tilde\mu_{t_i}(\tilde X\setminus K)<\epsilon$ for all large $i$. Therefore for any bounded continuous function $f$ on $\tilde X$ we have
\[
\lim_{i\to\infty}\int f\,d\tilde{\mu}_{t_i}=\int f\,d\tilde\mu.
\]

Now let $f\in C_c(\tilde X)$. For any $g\in L$, we define
\begin{equation} \label{eq:barfF}
\bar{f}(g\tilde{x_0}):=\int_{y\in F\tilde x_0}f(gy)\,d\tilde\lambda_F(y)=\int f\, d(g\tilde\lambda_F);
\end{equation}
this map is well defined because for any $\gamma\in N_L(F)\cap\Lambda$, we have $\gamma \tilde\lambda_{F}=\tilde\lambda_F$. 
Then $\bar{f}$ is a bounded and uniformly continuous function on $\tilde X$. Also 
\[
\bar f(gx_0)=\bar f(ghx_0),\,\forall  h\in F.
\]
By \eqref{eq:tildenu},
\[
\int f\,d\tilde\mu= \int_{b\in B}\Bigl(\int_{y\in F\tilde x_0} f(by)\,d\tilde\lambda_F(y)\Bigr)\,d\tilde\nu(b)= \int_{b\in B} \bar{f}(b\tilde x_0)\,d\tilde{\nu}(b)=\int \bar f\,d\tilde\mu.
\]
Therefore we obtain that 
\begin{equation} \label{eq:modF}
\lim_{i\to\infty} \int f\,d\tilde\mu_{t_i}=\lim_{i\to\infty} \int \bar{f}\,d\tilde\mu_{t_i}.
\end{equation}

\begin{claim} \label{claim:IfQinF}
If $\Po_{n_0}\subset F$, then 
\[
\mu=\int_{\eta\in J}\exp((\log \eta)H_{n_0})w(\kappa_n)\lambda_{F}\,d\eta.
\]
\end{claim}

\begin{proof}
Let $f\in C_c(\tilde X)$. 
Since $\Po_{n_0}\subset F$,
\[
\bar f(g\Po_{n_0}\tilde x_0)=\bar f(g\tilde x_0),\,\forall g\in L.
\]
Therefore by \eqref{eq:limit},
\[
\lim_{t\to\infty} \bar f(a_{t}u(R(e^{-t}\eta))\tilde x_0)=
\bar f(\exp((\log \eta)H_{n_0})w(\kappa_n)\tilde x_0).
\]
Hence by bounded convergence theorem,
\begin{align*}
\lim_{t\to\infty} \int \bar f\,d\tilde \mu_{t_i}&=(1/\abs{J})\lim_{i\to\infty}\int_{\eta\in J}\bar f(a_{t_i}u(R(e^{-t_i}\eta))\tilde x_0)\,d\eta\\
&=(1/\abs{J})\int_{\eta\in J} \bar f(\exp((\log \eta)H_{n_0})w(\kappa_n(s_0))\tilde x_0)\,d\eta.
\end{align*}
Therefore, since $\tilde\mu_{t_i}\to\tilde\mu$, by \eqref{eq:modF} and the definition of $\bar f$, we obtain that 
\begin{align*}
\int_{\tilde X} f\,d\tilde \mu&=\lim_{i\to \infty}\int f\,d\tilde\mu_{t_i}=\lim_{i\to \infty}\int \bar f\,d\tilde\mu_{t_i}\\
&=(1/\abs{J})\int_{\eta\in J}\Bigl(\int_{y\in F\tilde x_0} f(\exp((\log \eta)H_{n_0})w(\kappa_n)y)\,d\tilde\lambda_F(y)\Bigr)\,d\eta.
\end{align*}
In other words, 
\[
\tilde\mu =(1/\abs{J})\int_{\eta\in J} \exp((\log\eta)H_{n_0})w(\kappa_n)\tilde\lambda_F\,d\eta.
\]
Hence by projecting the measures on $L/\Lambda$, we deduce the claim.
\end{proof}

\begin{claim} \label{claim:QinF}
We have that $\Po_{n_0}\subset F$.
\end{claim}

\subsubsection*{Proof of \texorpdfstring{\Cref{claim:QinF}}{Claim 4.6} for \texorpdfstring{$n_0<n$}{n\textunderscore 0 < n}} By \eqref{eq:inF}, $w(\kappa_n(s_0))\inv U_1 w(\kappa_n(s_0))\subset F$. 
In this case $w(\kappa_n(s_0))=u(\kappa_n(s_0)\bfe_n)$. So $F\supset U_1$. We know that $\Po_{n_0}\subset N^1(F)$. So $F\cap \Po_{n_0}$ is normalized by $\Po_{n_0}$ and it contains $U_1$. The smallest normal subgroup of $\Po_{n_0}$ containing $U_1$ is $\Po_{n_0}$. Therefore $\Po_{n_0}\cap F=\Po_{n_0}$, proving the \Cref{claim:QinF} in the case of $n_0<n$.

\subsubsection*{Proof of \texorpdfstring{\Cref{claim:QinF}}{Claim 4.6} for \texorpdfstring{$n_0=n$}{n\textunderscore0=n}}
In this case $\Po_{n_0}=Q$ and $w(\kappa_n(s_0))=\sigma(\kappa_n(s_0))$. For any $\zeta\in\R$ and $\kappa\neq 0$,
\begin{align}
   \sigma(\kappa) u(\zeta\bfe_1)\sigma(\kappa)^{-1}
    &=
    \begin{bsmallmatrix}
    &&\kappa\\&I_{n-1}&\\ -\kappa\inv
    \end{bsmallmatrix}
    \begin{bsmallmatrix}
    1 & \zeta  \\
      &I_{n-1}\\&&1
    \end{bsmallmatrix}
    \begin{bsmallmatrix}
    &&-\kappa\\&I_{n-1}&\\ \kappa^{-1}
    \end{bsmallmatrix} \notag\\
    &=\begin{bsmallmatrix}1 & \\ & 1\\&&  \ddots \\0 &-\kappa^{-1}\zeta&& &1 \end{bsmallmatrix}=:u_{n,1}(-\kappa\inv \zeta). \label{eq:un1}
\end{align}
We define 
\begin{equation}
\label{def:Un1}
    U_{n,1}=\{u_{n,1}(\zeta):=\begin{bsmallmatrix}1 & \\ & 1\\&&  \ddots \\ 0&\zeta&& &1 \end{bsmallmatrix}:\zeta\in\R\}\subset \Po'=\begin{bsmallmatrix}
1&\R^n\\
&\SL(n,\R)
\end{bsmallmatrix},
\end{equation}
see \Cref{lemma:U}. Then $U_{n,1}\subset F$ by \eqref{eq:inF}. Now the smallest subgroup normalized by $\Po$ containing $U_{n,1}$ is $R_u(\Po)=\begin{bsmallmatrix}
  I_n&\\\R^n&1
\end{bsmallmatrix}$, the unipotent radical of $\Po$. Therefore 
\[
F\supset R_u(\Po). 
\]


\subsubsection*{Going modulo $R_u(\Po)$}
For each $1\leq i\leq n$,  we express 
\begin{equation} \label {eq:shorthand}
c_i:=\kappa_i+\epsilon_i(h)\text{,  where } \kappa_i=\kappa_i(s_0)\neq 0,\ \epsilon_i(h)=\epsilon_i(s_0,h),
\end{equation}
and $\epsilon_i(h)$ is  a polynomial in $h$ of degree at most $k$ with no constant term (see~\eqref{eq:epsiloni}), so $\kappa_i^{-1} c_i= 1+o(h)$ as $h\to 0+$.

Fix $t\in \cT$. Fix $\eta\in J\subset (0,\infty)$. Let $h=\eta e^{-t}$. Let $\bar r_i(t)=t\inv r_i(t)$ for $1\leq i\leq n$. Then $\bar r_1\geq \ldots \bar r_n\geq 0$ and $\sum_{i=1}^n \bar r_i=n$.  Also 
\begin{equation} \label{eq:b-eta}
a_t=b(\eta)b(h\inv)\text{, where } b(\alpha):=\exp\bigl((\log\alpha)\diag(n,\bar r_1,\ldots,\bar r_n)\bigr),\, \forall \alpha>0. \end{equation}
In view of \eqref{eq:limit}, we compute
 \begin{align}
   \notag 
   &\sigma(\kappa_n)^{-1}  b(h\inv)u(R(h))
    =\sigma(-\kappa_n)\begin{bsmallmatrix}
      h^{-n}\\&h^{\bar r_1}\\&&\ddots\\&&&h^{\bar r_n}
    \end{bsmallmatrix}\cdot
    \begin{bsmallmatrix}
      1&c_1h&\hdots &c_nh^n\\
      &1\\&&\ddots\\&&&1
    \end{bsmallmatrix}\\
 \notag &= \begin{bsmallmatrix}
    &&-\kappa_n\\&I_{n-1}&\\\kappa_n^{-1}
    \end{bsmallmatrix}
    \cdot 
  \begin{bsmallmatrix}
  h^{-n} & c_1h^{-(n-1)} & \ldots &c_n\\
         & h^{\bar r_1}  \\
         &               & \ddots \\
         &               &       & h^{\bar r_n}
  \end{bsmallmatrix}
 \\
\notag &=  
    \begin{bsmallmatrix} 
  0  &     0     &  \ldots    &   0    &  -\kappa_n h^{\bar r_n}\\
    & h^{\bar r_1} &      &       &   \\
    &          &\ddots&       &   \\
    &          &      & h^{\bar r_{n-1}} & \\
\kappa_n^{-1} h^{-n} & \kappa_n^{-1} c_1 h^{-(n-1)} &\dots& \kappa_n^{-1} c_{n-1}h^{-1}  & \kappa_n^{-1} c_n
    \end{bsmallmatrix}\\
   \notag
    &=
    \begin{bsmallmatrix} 
   \kappa_n c_n^{-1} &               & -\kappa_nh^{\bar r_n} \\
    &   I_{n-1}    &   \\
 &   & \kappa_n^{-1} c_n
    \end{bsmallmatrix} 
    \underbrace{\begin{bsmallmatrix}
      h^{-n+\bar r_n}& c_1h^{-(n-1)+\bar r_n}         & \dots      &  c_{n-1}h^{-1+\bar r_n}     & 0  \\
    & h^{\bar r_1} &      &       &   \\
    &          &\ddots&       &   \\
    &          &      & h^{\bar r_{n-1}} & \\
c_n^{-1} h^{-n} & c_n^{-1} c_1 h^{-(n-1)} &\dots& c_n^{-1} c_{n-1}h^{-1} & 1 
    \end{bsmallmatrix}
    }_{:=\psi(h)} \label{def:psi}
    \\
    &=: (I_{n+1}+O(h^{\bar r_n})) \cdot \psi(h), \notag
\end{align}
and $\psi(h)\in \Po$. We further observe that 
\[
\psi(h)R_u(\Po)=\bar{a}(h\inv) u(\bar R(h))R_u(\Po),
\]
where
\begin{gather}
\bar{a}(h\inv):=\diag(h^{-(n-\bar r_n)},h^{\bar r_1},\ldots,h^{\bar r_{n-1}},1)
\notag\\
\bar{R}(h):=(c_1h^1, \ldots, c_{n-1}h^{n-1}, 0)\in\R^n.
\label{def:Rbar}
\end{gather}

Since $n_0=n$, $\lim_{t\to\infty} r_n(t)=\infty$; so $h^{\bar r_n}=\eta^{\bar r_n}e^{-r_n(t)}\to 0$ as $t\to\infty$ uniformly for $\eta\in J$, which is a compact subset of $(0,1)$.  Therefore, since $R_u(\Po)\subset F$, we have
\begin{align} 
a_{t}u(R(h))F&=b(\eta)b(h\inv)u(R(h))F \notag\\
&=(I_{n+1}+O(h^{\bar r_n}))b(\eta)\sigma(\kappa_n)\bar{a}(h^{-1})u(\bar{R}(h))F. \label{eq:modF1}
\end{align}

Without loss of generality, we replace $\cT$ by $\{t_i\}_{\{i\in\N\}}$. Therefore for any $f\in C_c(\tilde X)$, in view of \eqref{eq:barfF}, since $\bar f$ is uniformly continuous and bounded, 
\begin{align}
\lim_{t\to\infty}\int \bar f\,d\tilde\mu_t
&=\lim_{t\to\infty} \int_{\eta\in J} \bar f(a_tR(h)\tilde x_0)\,d\eta \text{, where } h=e^{-t}\eta,\notag\\
&=\lim_{t\to\infty} \int_{\eta\in J} \bar f(b(\eta)\sigma(\kappa_n)\bar{a}(h\inv)u(\bar{R}(h))\tilde x_0)\,d\nu(\eta). 
\label{eq:modRuQ}
\end{align}

\begin{subclaim} \label{claim:U2}
 $\mu$ is $U_{n,1}$-invariant.
\end{subclaim}

\begin{proof} 
We follow the the proof of \Cref{claim:u1-inv}. In view of \eqref{eq:shorthand} and \eqref{def:Rbar}, we express
\[
\bar{R}(h)=(c_1h^1, \ldots, c_{n-1}h^{n-1}, 0)=\kappa_1 h\bfe_1+h^2R_2(h),
\]
where $R_2$ is a polynomial function  taking values in $\R^n$. Take any large $t$. For $\xi\in\R$, $\eta\in J$, and $h:=e^{-t}\eta$, by arguing as in the derivation of \eqref{eq:uzetamu} we get
\begin{equation} \label{eq:pushuxi}
u(\xi \bfe_1)\bar{a}(h\inv)u(\bar{R}(h))=u(\tilde{\Delta}_{\xi,h})\bar{a}(\tilde h\inv)u(\bar{R}(\tilde h)),
\end{equation}
where $\tilde h=h+\kappa_1\inv \xi h^{(n-\bar r_n+\bar r_1)}$, $\tilde\Delta_{\xi,h}\in\R^n$, and $\norm{\tilde\Delta_{\xi,h}}=O(h)$ as $t\to\infty$.

Now let $\zeta\in\R$ and $\eta\in J$. By \eqref{eq:un1}, \eqref{def:Un1}, and \eqref{eq:b-eta} we have
\[
b(\eta)\inv u_{n,1}(\zeta)b(\eta)=u_{n,1}(\eta^{\bar r_1-\bar r_n}\zeta) \text{ and }
\sigma(\kappa_n)\inv u_{n,1}(\zeta)\sigma(\kappa_n)=u(-\kappa_n\zeta\bfe_1).
\]
Hence, putting $\xi=-\kappa_n\eta^{\bar r_1-\bar r_n}\zeta$, we get
\[
u_{n,1}(\zeta)b(\eta)\sigma(\kappa_n)=b(\eta)\sigma(\kappa_n)u(\xi\bfe_1).
\]

For any  $t\in\cT$ and $\eta\in J$, let $h=e^{-t}\eta$. Then in view of \eqref{eq:modF1},
\begin{align}
u_{1,n}(\zeta)\bigl[b(\eta)\sigma(\kappa_n)\bar{a}(h\inv)u(\bar{R}(h))\bigr]
&=b(\eta)\sigma(\kappa_n)\bigl[u(\xi\bfe_1)\bar{a}(h\inv)u(\bar{R}(h))\bigr]\notag\\
\text{\small{by \eqref{eq:pushuxi}}}\qquad &=b(\eta)\sigma(\kappa_n)\bigl[u(\Delta_{\xi,h})\bar{a}(\tilde h\inv)u(\bar{R}(\tilde{h}))\bigr] \notag\\
&=\delta_{\zeta,\eta,h}\bigl[b(\tilde\eta)\sigma(\kappa_n)\bar{a}(\tilde h\inv)u(\bar{R}(\tilde{h}))\bigr],  
\label{eq:un-del}
\end{align}
where 
\begin{align}
\tilde h&:=e^{-t}\eta+\kappa_1\inv \xi h^{(n-\bar r_n+\bar r_1)}=e^{-t}\tilde{\eta},\notag\\
\tilde{\eta}&:=\eta+\kappa_1\inv\xi \eta^{n-\bar r_n+\bar r_1} e^{-(n-\bar r_n+\bar r_1-1)t}=\eta-\kappa_1\inv \kappa_n\zeta \eta^{n}e^{-(n-\bar r_n+\bar r_1-1)t},\notag \\
\delta_{\zeta,\eta,h}&:=b(\eta\tilde{\eta}\inv) \cdot (b(\eta) \sigma(\kappa_n))u(\tilde\Delta_{\xi,h})(b(\eta) \sigma(\kappa_n))\inv=I+O(e^{-t}), \label{eq:delta-zeta}
\end{align}
because $J\subset (0,\infty)$ is a compact interval, $n\geq 2$ and $\bar r_1\geq \bar r_n$, and 
\[
u(\tilde\Delta_{\xi,h})=I+O(h)=I+O(e^{-t}) 
\text{ and } b(\eta\tilde{\eta}\inv)=I+O(e^{-(n-\bar r_n+\bar r_1-1)t})=I+O(e^{-t}).
\]
Also
\begin{equation} \label{eq:deta}
\frac{d\tilde{\eta}}{d\eta}=1-\kappa_1\inv\kappa_n \zeta n\eta^{n-1}e^{-(n-\bar r_n+\bar r_1-1)t}=1+O(e^{-t}). 
\end{equation}

Let $f\in C_c(\tilde X)$. Let $\bar f$ be as in \eqref{eq:barfF}. Let $\tilde J=\{\tilde\eta:\eta\in J\}$. Since $\bar f$ is bounded and uniformly continuous, by \eqref{eq:modRuQ}, we get
\begin{align*}
\lim_{t\to\infty} \int \bar{f}(u_{n,1}(\zeta)y)\,d\tilde\mu_t(y)
&=\lim_{t\to\infty} \int_{\eta\in J} \bar{f}(u_{n,1}(\zeta)b(\eta)\sigma(\kappa_n)\bar{a}(h\inv)u(\bar{R}(h))\tilde x_0)\,d\eta \\
\text{\small{(by \eqref{eq:un-del})} }\quad &=\lim_{t\to\infty}\int_{\tilde\eta\in \tilde J} \bar{f}(\delta_{\zeta,\eta,h}b(\tilde\eta)\sigma(\kappa_n)\bar{a}(\tilde h\inv )u(\bar{R}(\tilde{h})\tilde x_0)\frac{d\tilde{\eta}}{d\eta}\,d\tilde{\eta} \\
\text{\small{(by \eqref{eq:delta-zeta}  and \eqref{eq:deta}}}) \quad &=\lim_{t\to\infty} \int_{\tilde
\eta\in J} \bar f(b(\tilde\eta)\sigma(\kappa_n)\bar{a}(\tilde h\inv)u(\bar{R}(\tilde h))\tilde x_0)\,d\tilde{\eta}\\
&=\lim_{t\to\infty} \int \bar{f}\,d\tilde\mu_t \text{, as $\tilde h=e^{-t}\tilde \eta$.}
\end{align*}

Since the map $f\mapsto \bar{f}$ is $L$-equivariant, due to \eqref{eq:modF}, we get
\[
\lim_{t\to\infty} \int f(u_{n,1}(\zeta)y)\,d\tilde\mu_t=\lim_{t\to\infty} \int f\,d\tilde\mu_t.
\]
Hence $u_{n,1}(\zeta)
\tilde\mu=
\tilde\mu$. Therefore $u_{n,1}(\zeta)\mu=
\mu$, which proves the \Cref{claim:U2}.
\end{proof}

Since $\tilde\mu$ is invariant under the action of $U_{n,1}$ and $U_{n,1}\subset \Po'$, by the choice of $U$ as in \Cref{sec:U}, we have $U_{n,1}\subset U$. By \eqref{eq:inF}, $\sigma(\kappa_n)\inv U_{n,1} \sigma(\kappa_n)\subset F$. By \eqref{eq:un1}, $\sigma(\kappa_n)\inv U_{n,1} \sigma(\kappa_n)=U_1$. So $F\cap {\Po}$ contains $U_1$. The smallest normal subgroup of $\Po$ containing $U_1$ is $\Po$. Therefore $F\cap \Po=\Po$. Hence $F\supset \Po$. 

This completes the proof of \Cref{claim:QinF} in all the cases. 

Now combining \Cref{claim:IfQinF2}, \Cref{claim:IfQinF}, and \Cref{claim:QinF} we deduce that 
\[
\mu=(1/\abs{J})\int_{\eta\in J} \exp((\log\eta)H_{n_0})w(\kappa_n(s_0))\lambda_F\,d\eta.
\]
We showed that given any sequence $t_i\to\infty$, after passing to a subsequence, $\mu_{t_i}$ converges to the measure $\mu$ as described above. Therefore we conclude that $\mu_t\to\mu$ as $t\to\infty$. This completes the proof of \Cref{prop:main-2}. \qed

\begin{remark} \label{rem:n0xi} 
If we write $c_t=\exp(\sum_{i>n_0}^n \xi_i(t)H_i)$, then by our assumption \eqref{eq:n0xi}, we have that $c_t\to e$ as $i\to\infty$. Now $c_t\inv a_t\in G_{n_0}$ for all $t$. We will replace $a_t$ by $c_t\inv a_t$ for all $t$, and and assume that $\{a_t\}_t\subset G_{n_0}$. 
\end{remark}

\subsection{Proof of Theorem~\ref{thm:main}} \label{subsec:ThmMain} 
Since $\nu$ is absolutely continuous with respect to the Lebesgue measure, it is enough to prove this theorem is the case of $\nu$ being the normalized Lebesgue measure on every compact interval, say $J$, of positive length. Since $f$ is bounded, we may further assume that $0\not\in J$. And in view of \Cref{rem:h-negative} it is enough to consider the case of $J\subset (0,\infty)$.

Since \eqref{eq:MainLimitRight} depends only on $n_0$ and not involve the sequence $\{a_t\}_t$, it is enough
to prove the result for some subsequence of any given subsequence of $\{a_t\}_t$. 

By passing to a subsequence, we assume that for all $i<j$, $\lim_{t\to\infty} r_i(t)-r_j(t)$
exists in $[0,\infty]$ as $r_i\geq r_j$. Also $\lim_{t\to\infty} r_i(t)-r_j(t)=\infty$ for all $i\leq n_0<j$. Therefore by \eqref{eq:vinfty}, since $b_{i,j}(s)=0$ if $i>j$ and $b_{i,i}=1$,
\begin{align} 
  v_\infty(s):=\lim_{t\to\infty} v_t(s)&=\begin{bsmallmatrix} 1 & \\ 
    & \left(b^\ast_{i,j}:=b_{i,j}(s)\cdot\lim_{t\to\infty} e^{-(r_i(t)-r_j(t))}\right)
    \end{bsmallmatrix}\notag\\
    &=\begin{bsmallmatrix} 1 & \\ 
    & \left(b_{i,j}^\ast(s)\right)_{n_0\times n_0}&\\
    &&(b_{i,j})_{n-n_0\times n-n_0}
    \end{bsmallmatrix}\label{eq:vsinfty-1}\\
    &=\begin{bsmallmatrix} 1 & \\ 
    & \left(b_{i,j}^\ast(s)\right)_{n_0\times n_0}&\\
    &&I_{n-n_0}
    \end{bsmallmatrix}\cdot v(s)_{n-n_0},
    \label{eq:vinfty-2}
    \end{align}
where $v(s)_{n-n_0}$ as defined in \eqref{eq:vs}. 

For any $\eta\in J$, let $h_t=\eta e^{-t}$. We have chosen $k$ such that \eqref{def:ht} holds. Hence by \eqref{eq:poly-approx}, uniformly for $\eta\in J$, we get
\begin{align}
    a_tu(\phi(s+\eta e^{-t}))&=(I_{n+1}+\bfo_t(1))v_t(s)\inv a_tu(R(h_t))v(s)u(\phi(s)) \notag\\
    &=(I_{n+1}+\bfo_t(1))v_\infty(s)\inv a_tu(R(h_t))v(s)u(\phi(s)). \label{eq:atpoly}
\end{align}

Let $x\in L/\Lambda$ and $f\in C_c(L/\Lambda)$. Put $x_s=v(s)u(\phi(s))x$. In view of \Cref{rem:n0xi} we assume that $\{a_t\}_t\subset G_{n_0}$. Then  by \Cref{prop:main-2},
\begin{align*}
    &\lim_{t\to\infty} \int_{J} f(a_tu(\phi(s+e^{-t}\eta))x)\,d\nu(\eta)\\
    &= \lim_{t\to\infty} \int_J f(v_\infty(s)\inv a_tu(R(h_t))x_s)\,d\nu(\eta)\\
    &=\int_J \int_{y\in L/\Lambda} f(v_\infty(s)\inv \exp((\log\eta)H_{n_0})w(\kappa_n(s))y)\,d\mu_{\cl{Q_{n_0}x_s}}(y)\,d\nu(\eta)\\
   &= \int_J \int_{y\in L/\Lambda} f( \exp((\log\eta)H_{n_0})v_\infty(s)\inv w(\kappa_n(s))y)\,d\mu_{\cl{\Po_{n_0}x_s}}(y)\,d\nu(\eta),
\end{align*}
because  $v_\infty(s)$ commutes with $\exp((\log\eta)H_{n_0})$ by \eqref{eq:Hn0} and \eqref{eq:vsinfty-1}. We have
\begin{equation*}
v_\infty(s)\inv w(\kappa_n(s))\Po_{n_0}=v(s)_{n-n_0}\inv w(\kappa_n(s))\Po_{n_0},
\end{equation*}
because by \eqref{eq:vinfty-2}, $v_\infty(s)v(s)_{n-n_0}\inv$ fixes $w(\kappa_n(s))\bfe_n$.  Therefore \eqref{eq:MainLimitLeft}=\eqref{eq:MainLimitRight} follows.
This completes the proof of \Cref{thm:main}. \qed

\section{Equidistribution under slower shrinking} \label{sec:slow} 

Let the notation be as in \Cref{sec:intro}. We begin with a restricted version of \Cref{thm:slow-shrink}. Let $\bar r_1(\infty)=\limsup_{t\to\infty} t\inv r_1(t)$.

\begin{theorem} \label{thm:slow-shrink-1}
Let $s\in(0,1)$ and $k>n+\bar r_1(\infty)$. Suppose that $\phi:(0,1)\to \R^n$ is $C^k$ and ordered regular in a neighborhood of $s$. Let $t_i\to\infty$ and $\alpha_i\to\infty$ be such that 
\begin{equation} \label{eq:slowgrowth1}
    \limsup_{i\to\infty}\alpha_i^k e^{-kt_i+nt_i+r_1(t_i)}<\infty.
\end{equation}
Let $\nu$ be an absolutely continuous probability measure on $\R$. Then for any $s_i\to s$, $x\in L/\Lambda$, and $f\in C_c(L/\Lambda)$ we have that 
\begin{equation*}
    \lim_{t\to\infty} \int f(a_tu(\phi(s_i+\alpha_ie^{-t_i}\eta))g_ix)\,d\nu(\eta) 
    = \int f\, d\mu_{\cl{G_{n_0}x}},
\end{equation*}
where $g_i\to e$ in $L$ such that $g_i$ satisfies \eqref{eq:gtn0}; that is, $\cl{G_{n_0}g_ix}\subset g_i\cl{G_{n_0}x}$, $\forall\,i$.
\end{theorem}

\begin{proof} 
Without loss of generality we will assume that $\nu$ is the normalized Lebesgue measure on a compact interval $J\subset (0,\infty)$ with nonempty interior; see the explanation at the beginning of \Cref{subsec:ThmMain}. Again without loss of generality we may pass to a subsequence of $\{t_i\}_i$ and assume that \eqref{eq:vinfty-2} holds. 

For any $\eta\in J$, set  $h_{t_i}=\alpha_{t_i}e^{-t_i}\eta$. Then condition  \eqref{eq:slowgrowth1}, which is same as \eqref{eq:alphat}, implies \eqref{def:ht}. Therefore by \eqref{eq:poly-approx}, for $i\gg 0$, putting $q(s_i)=v(s_i)u(\phi(s_i))$, we get
\begin{align*}
    &a_{t_i}u(\phi(s_i+\alpha_{t_i}e^{-t_i}\eta))g_ix\\
    &=(I_{n+1}+\bfo_{t_i}(1))v_\infty(s_i)\inv a_{t_i}u(R_{s_i}(\alpha_{t_i}e^{-t_i}\eta))v(s_i)u(\phi(s_i))g_ix\\
    &=(I_{n+1}+\bfo_{t_i}(1))v_\infty(s)\inv a_{t_i}u(R_{s_i}(\alpha_{t_i}e^{-t_i}\eta))\rho_ix_{s},
\end{align*}
where $x_{s}=q(s)x$, $q(s)=v(s)u(\phi(s))$, and $\rho_i=q(s_i)g_iq(s)\inv\to e$ as $i\to\infty$. 

Also
\begin{align}
\cl{G_{n_0}\rho_ix_{s}}&=q(s_i)\cl{G_{n_0}g_ix} \qquad \text{\small{(as $\rho_ix_{s}=q(s_i)g_ix$ and $q(s_i)\in N_L(G_{n_0})$)}}\notag\\
&\subset q(s_i)g_i\cl{G_{n_0}x} \qquad \text{\small{(by \eqref{eq:gtn0})}}\notag\\
&=q(s_i)g_iq(s)\inv \cl{G_{n_0}q(s)x}\qquad \text{\small{(as $q(s)\in N_L(G_{n_0})$)}}\notag\\
&=\rho_i\cl{G_{n_0}x_{s}}. \label{eq:rhoi}
\end{align}
Hence condition \eqref{eq:gt} is satisfied for $\rho_i$ in place of $g_i$ and $x_{s}$ in place of $x$. 

In view of \Cref{rem:n0xi}, we will assume that $\{a_t\}_t\subset G_{n_0}$. Since $f$ is uniformly continuous on $L/\Lambda$,
\begin{align*}
    &\lim_{i\to\infty} 
    \int f(a_{t_i}u(s_i+\alpha_ie^{-t_i}\eta)g_ix_{s})\,d\nu(\eta)\\
    &=\lim_{i\to\infty} 
    \int f(v_\infty(s)\inv a_{t_i}u(R_{s_i}(\alpha_{t_i}e^{-t_i}\eta))\rho_ix_{s})\,d\nu(\eta) \\
    &=\int_{L/\Lambda} f(v_\infty(s)\inv y)\,d\mu_{\cl{G_{n_0}x_{s}}}(y)\text{, by \Cref{prop:main-1}}\\
   &= \int_{L/\Lambda} f\,d\mu_{\cl{[v_\infty(s)\inv G_{n_0}v(s)u(\phi(s))]x}}.
\end{align*}
By \eqref{eq:vinfty-2} it is straightforward to see that
   \begin{equation} \label{eq:Gn0-vs}
   v_\infty(s)\inv G_{n_0}v(s)u(\phi(s))=G_{n_0}.
   \end{equation}
This completes the proof of the theorem.
\end{proof}

Now we will provide a proof of \Cref{thm:smooth} by assuming that $\phi$ is $C^k$-smooth for some integer $k>n+\bar r_1(\infty)$. 

\subsection{Proof of {Theorem~\ref{thm:smooth}} for \texorpdfstring{$\mathbf{k>n+\bar r_1(\infty)}$}{k>n+r\textunderscore 1}}
The set of points where $\phi$ is not ordered regular is discrete in $(0,1)$, see \Cref{cor:ord-reg}. 
Since $\sigma$ is absolutely continuous with respect to the Lebesgue measure, it is enough to prove the theorem for each $\sigma$ which is the normalized probability measures on a compact interval $I$ of positive length and that $\phi$ is ordered regular at all $s\in I$. 
To simplify our notation, by affinely re-parametrizing $\phi$ we will assume that $I=[0,1]$.

We are only considering case of $k>n+\bar r_1(\infty)$. For $t\geq 0$, let $\ell_t:=\lfloor e^{\left(\frac{n+\bar r_1(\infty)}{k}\right)t}\rfloor$ and
$\alpha_t=e^{t}\ell_t\inv$. Then $\alpha_t\to\infty$ and $\alpha_t^k e^{-kt+nt+r_1(t)}\to 1$ as $t\to\infty$. Given $t\geq 0$, let $s_j=j/\ell_t$ for all $0\leq j\leq \ell_t$. 
Then $\ell_t\inv=\alpha_t e^{-t}$ and 
\begin{align}  
A_t:=&\int_0^1 f(a_tu(\phi(s))g_tx)\,ds=\sum_{j=0}^{\ell_t-1} \int_{s_{j}}^{s_{j+1}} f(a_tu(\phi(s))g_tx)\,ds+O(\ell_t\inv\norm{f}_\infty)\notag\\
=&\ell_t\inv \sum_{j=0}^{\ell_t-1} \int_0^1 f(a_tu(\phi(s_j+\alpha_te^{-t}\eta))g_tx)\,d\eta+O(\ell_t\inv\norm{f}_\infty).\label{eq:sum1}
\end{align}

To prove the result by contradiction, suppose there exist an $\epsilon>0$ and a sequence $t_i\to\infty$ such that 
\begin{align*}
    \abs*{A_{t_i}- 
    \int f\,d\mu_{\cl{G_{n_0}x}}}\geq\epsilon>0. 
\end{align*}
Then by \eqref{eq:sum1} there exists a sequence $(s'_i)_i\in I$ such that 
\[
\abs*{\int_0^1 f(a_{t_i}u(\phi(s'_i+\alpha_{t_i}e^{-t_i}\eta))g_{t_i}x\,d\eta - 
\int f\,d\mu_{\cl{G_{n_0}x}}}\geq \epsilon. 
\]
By passing to a subsequence, we may assume that $s'_i\to s\in I$ as $i\to\infty$. This contradicts \Cref{thm:slow-shrink-1}. This completes the proof of the theorem. \qed

\subsection{Proof of Theorem~\ref{thm:slow-shrink}}

As noted earlier, is enough to prove the result when $\nu$ is the normalized Lebesgue measure on an interval, say $(a,b)$. Then
\[
\int f(a_tu(\phi(s+\beta_te^{-t}\eta))g_tx)\,d\nu(\eta)=\int_0^1 f(a_tu(\phi(s+\beta_te^{-t}(a+(b-a)\eta)))g_tx)\,d\eta:=A_t.
\]

Due to \Cref{thm:slow-shrink-1}, it only remains to consider the limit of $A_{t_i}$ for a sequence $\{t_i\}$ such that
\[
\lim_{i\to\infty} \beta_{t_i}e^{-t_i}=0 \text{ and } \lim_{i\to\infty}\beta_{t_i}e^{-(1-\frac{n+\bar r_1(\infty)}{k})t_i}=\infty.
\]

Given $i\in\N$, let $\ell_i=\lfloor\beta_{t_i}e^{-(1-\frac{n+\bar r_1(\infty)}{k})t_i}\rfloor$, and $\alpha_i=(b-a)\beta_{t_i}\ell_i\inv$, and for each $0\leq j\leq \ell_i$, let
\[
s_j=j/\ell_i \text{ and }
z_j=s+\beta_{t_i}e^{-t_i}(a+(b-a)s_j).
\]
Then
\begin{align}  
A_{t_i}&=\sum_{j=0}^{\ell_i-1} \int_{s_j}^{s_{j+1}} f(a_{t_i}u(\phi(s+\beta_{t_i}e^{-t_i}(a+(b-a)\eta)))g_{t_i}x)\,d\eta+O(\ell_i\inv \norm{f}_{\infty})\notag\\
&=\ell_i \inv \sum_{j=0}^{\ell_i-1}\int_{0}^{1} f(u(\phi(z_j+\alpha_ie^{-t_i}\eta))g_{t_i}x)\,d\eta+ O(\ell_i\inv \norm{f}_{\infty}). \label{eq:sum2}
\end{align}
To prove the result by contradiction, suppose that after passing to a subsequence, there exists $\epsilon>0$ such that for all $i$, 
    \begin{align*}
    \abs*{A_{t_i}- 
    \int f\,d\mu_{\cl{G_{n_0}x}}}\geq\epsilon>0. 
\end{align*}
Then by \eqref{eq:sum2}, for each $i$, there exists $s'_i\in s+\beta_{t_i}e^{-t_i}[a,b]$ such that 
\begin{align}
    \abs*{\int_0^1 f(a_tu(\phi(s'_i+\alpha_ie^{-t_i}\eta))g_{t_i}x)\,d\eta - 
    \int f\,d\mu_{\cl{G_{n_0}x}}}\geq\epsilon>0. \label{eq:contra2}
\end{align}
We note that $s'_i\to s$ and $\limsup_{i\to\infty}\alpha_i^ke^{-kt_i+nt_i+r_1(t_i)}\leq b-a$. Therefore  \eqref{eq:contra2} contradicts \Cref{thm:slow-shrink-1}. This completes the proof. \qed

\section{Equidistribution for optimal shrinking outside a countable set}  \label{sec:butcountable}

First we will prove the limit distribution results for optimal shrinking at all but countably many points. And then describe limit distributions of translates of a fixed piece of a smooth curve.

\subsubsection{Countability of exceptional points}
\label{subsec:countable}
Let $\{a_t\}$ be as in \eqref{def2:ri}. We pick $k\in\N$ such that \eqref{def:k} holds; that is, 
\[
\limsup_{t\to\infty} nt+r_1(t)-kt<\infty.
\]
Let $\phi:(0,1)\to\R^n$ be a $C^k$-map which is $\ell$-nondegenerate for some $n\leq \ell\leq k$. Let 
\[
\cJ=\{s\in (0,1):\phi \text{ is ordered regular at } s\}.
\]
Then by \Cref{cor:ord-reg}, $(0,1)\setminus \cJ$ is a countable discrete subset of $(0,1)$. 

For any $s\in \cJ$, let $q(s)=v(s)u(\phi(s))$. In view of \Cref{prop:NonGen-s} and \eqref{eq:stabil}, for any $F\in\cH$, define
\[
I_F(\phi)=\{s\in \cJ: q(s)\inv \Po_{n_0}q(s) \subset N^1(F)\},
\]
and since $v(s)$ is upper triangular and $u(\phi(s))\in G_{n_0}$, we get $q(s)\in N_G(G_{n_0})$.

\begin{proposition}\label{prop:discrete}
If $I_F(\phi)$ is not discrete in $\cJ$, then $G_{n_0}\subset N^1(F)$.
\end{proposition} 

\begin{proof} Let $V$ and $p_F$ be as in \Cref{subsec:lintech}. Then the stabilizer of $p_F$ in $L$ equals $N^1(F)$. Suppose that $s\in I_F(\phi)$. Then $\Po_{n_0}$ fixes $q_sp_F$. Let $\eta>0$. Then 
\begin{align}
&\lim_{t\to\infty} \exp(tH_{n_0})u(\phi(s\pm\eta e^{-t}))p_F \notag\\
&=\lim_{t\to\infty} \exp(tH_{n_0})v(s)\inv u(R_s(\pm\eta e^{-t}))q(s)p_F \quad \text{\small{ (by
\eqref{eq:poly-approx})} }\notag\\
&=v_\infty(s)\inv \exp((\log\eta)H_{n_0}) w((\pm 1)^n\kappa_n(s))q(s)p_F, \quad\text{\small{(by \Cref{lem:Qfixed})},}
\label{eq:limit-bounded}
\end{align}
where
\begin{equation} \label{eq:Hn0vs}
v_{\infty}(s):=\lim_{t\to\infty}\exp(tH_{n_0})v(s) \exp(-tH_{n_0})\in G.
\end{equation}

Let 
\begin{align} 
V^\pm &=\{v\in V:\lim_{t\to\infty} \exp(\mp tH_{n_0})v=0\} 
\text{ and }\notag\\  
V^0&=\{v\in V: \exp(tH_{n_0})v=v,\,\forall t\}. \label{eq:V+-0}
\end{align}
Then $V=V^+\oplus V^{0}\oplus V^-$. Let $\pi_+:V\to V^+$ and  $\pi_0:V\to V^0$ denote the corresponding projections.
Fix some norm on $V$ such that $\pi_+$ is a contraction. Since $V$ is finite dimensional, there exists $\alpha>0$ such that for any $v\in V$, we have $\norm{\exp(tH_{n_0})v}\geq e^{\alpha t}\norm{\pi_+(v)}$. Also for any $\eta>0$,
\[
\norm{\pi_+(u(\phi(s\pm \eta e^{-t}))p_F)}\geq (1/2)\norm{\pi_+(u(\phi(s))p_F)},\,\forall t>\!>0.
\]
So if  $\norm{\pi_+(u(\phi(s))p_F)}\neq 0$, then  \eqref{eq:limit-bounded} does not hold.
Hence $u(\phi(s))p_F\in V^{0}+V^-$.

This shows that 
\[
I_F(\phi)\subset \{s\in (0,1): u(\phi(s))p_F\in V^{0}+V^-\}.
\]

Now suppose that $I_F(\phi)$ is not discrete in $\cJ$. Since $I_F(\phi)$ is closed in $\cJ$, there exists $s\in I_F(\phi)$ and a sequence $\{s_i\}_{i\in\N}\subset I_F(\phi)\setminus \{s\}$ such that $s_i\to s$. We have 
\[
u(\phi(s_i))p_F=\pi_0(u(\phi(s_i))p_F)+\pi_-(u(\phi(s_i))p_F).
\]
So for any $t_i\to\infty$,
\begin{equation} \label{eq:uniflimit}
\lim_{i\to\infty} \exp(t_iH_{n_0})u(\phi(s_i))p_F=\lim_{i\to\infty} \pi_0(u(\phi(s_i))p_F)=\pi_0(u(\phi(s))p_F).
\end{equation}

Let $\eta>0$. Choose $t_i\to\infty$ such that $s_i=s\pm \eta e^{-t_i}$.  Then by \eqref{eq:limit-bounded} and \eqref{eq:uniflimit},
\begin{equation*} \label{eq:compare1}
\pi_0(u(\phi(s))p_F)=v_\infty(s)\inv \exp((\log\eta)H_{n_0}) w((\pm 1)^n\kappa_n(s))q_sp_F.
\end{equation*}
Therefore for each $\eta>0$, we get 
\begin{equation}
    \label{eq:compare}
p_s:=v_\infty(s)\pi_0(u(\phi(s))p_F)=\exp((\log\eta)H_{n_0}) w((\pm 1)^n\kappa_n(s))q(s)p_F.
\end{equation}

First we consider the case of $n_0<n$. Then
$\exp(\R H_{n_0})\subset \Po_{n_0}$ and $w(\kappa)=u(\kappa \bfe_n)$ for any $\kappa\in\R$. Now
\[
\exp((\log\eta)H_{n_0})u(\kappa\bfe_n)\exp(-(\log\eta)H_{n_0})=u(\eta^n\kappa\bfe_n).
\]
Since $q(s)p_F$ is fixed by $\exp(\R H_{n_0})\subset \Po_{n_0}$, from \eqref{eq:compare} we get that 
\[
p_s=u((\pm1)^n\kappa_n(s)\eta^n\bfe_n)q(s)p_F, \,\forall \eta>0.
\]
Since $\phi$ is ordered regular at $s$, we get $\kappa_n(s)\neq 0$. Therefore varying $\eta>0$, we see that $p_s$ is fixed by $u(\R\bfe_n)$. In particular, $p_s=q(s)p_F$. Therefore $q(s)p_F$ is fixed by the group generated by $u(\R\bfe_n)$ and $\Po_{n_0}$ which is $G_{n_0}$. And since $q(s)$ normalizes $G_{n_0}$, we conclude that $p_F$ is fixed by $G_{n_0}$. This completes the proof if $n<n_0$.

Now suppose that $n=n_0$. Then $\Po_{n_0}=\Po$, and for any $\kappa\neq 0$, $w(\kappa)=\sigma(\kappa)$, and $\sigma(\kappa)\Po \sigma(\kappa)\inv=\Po'$, where $\Po'=\begin{bsmallmatrix}
1&\R^n\\
&\SL(n,\R)
\end{bsmallmatrix}$, see \eqref{eq:un1}. Since $\phi$ is ordered regular at $s$, we get $\kappa_n(s)\neq 0$.  Since $q(s)p_F$ is  fixed by $\Po_{n_0}=\Po$, putting $\eta=1$ in \eqref{eq:compare}, we get $p_s=\sigma((\pm1)^n\kappa_n(s))(q(s)p_F)$ is fixed by $\Po'$. By varying $\eta>0$ in \eqref{eq:compare} we get that $p_s$ is fixed by $\exp(\R H_n)$. Now $p_s$ is fixed by the parabolic subgroup of $G$ which is generated by $\exp(\R H_n)$ and $\Po'$. Hence $p_s$ is fixed by $G$. So $p_F$ is $G$-fixed.
\end{proof}

Let 
\[
E_{\Lambda}(\phi)=[(0,1)\setminus \cJ]\cup \left[\cup\{I_F(\phi): F\in \cH_{\Lambda},\, G_{n_0}\not\subset N^1(F)\}\right].
\]
Now $(0,1)\setminus \cJ$ is discrete in $(0,1)$, $\cH_{\Lambda}$ is countable, and by \Cref{prop:discrete}, $I_F(\phi)$ is discrete in $\cJ$ for each $F\in\cH$ such that $G_{n_0}\not\subset N^1(F)$. Therefor we have that $E_{\Lambda}(\phi)$ is a countable set. 

\begin{proof}[Proof of Theorem~\ref{thm:co-countable}] 
When $\{a_t\}_t$ is uniform, in view of \Cref{cor:ord-reg}, the result follows from \cite[Remark~1.3(2) and Theorem~1.4]{Shah2018}. 

So, we will assume that $\{a_t\}_t$ is non-uniform. 
If $x=gx_0$ for some $g\in L$, then let $E_x=E_{g\Lambda g\inv}(\phi)$. Replacing $\Lambda$ by $g\Lambda g\inv$, without loss of generality we assume that $x=x_0=e\Lambda$. Let  $s\in (0,1)\setminus E_{\Lambda}(\phi)$. Then $s\in \cJ$. Consider a sequence $s_i\to s$ in $(0,1)$. Since $\cJ$ is open, $s_i\in \cJ$ for all large $i$.  Let $t_i\to\infty$ be given. After passing to a subsequence, we may assume that \eqref{eq:vinfty-2} holds. As in \eqref{eq:poly-approx}, for all large $i$, 
\begin{equation} \label{eq:poly-approx2}
a_{t_i}u(\phi(s_i+e^{-{t_i}}\eta))=(I_{n+1}+\bfo_{t_i}(1))v_\infty(s_i)\inv a_{t_i}u(R_{s_i}(e^{-t_i}\eta))q(s_i). 
\end{equation}

Let $x_s=q(s)x_0$ and $\rho_i=q(s_i)g_iq(s)\inv$ for all large $i$. Then $\rho_i\to e$ and satisfies \eqref{eq:gt} for $x_s$, see~\eqref{eq:rhoi}. Let $J\subset (0,\infty)$ be a compact interval with nonempty interior. In view of \Cref{rem:n0xi}, without loss of generality we may assume that $\{a_t\}_t\subset G_{n_0}$. In view of \eqref{eq:mut1}, for each $i$, let
\begin{align} 
     \mu_{t_i} =
     \int_{\eta\in J} a_{t_i}u(R_{s_i}(\eta e^{-t_i}))q(s_i)g_i\delta_{x}\,d\eta
     =\int_{\eta\in J} a_{t_i}u(R_{s_i}(\eta e^{-t_i}))\rho_{i}\delta_{x_s}\,d\eta. 
     \label{eq:mui}
\end{align}

Let $F\in\cH_{\Lambda}$. Suppose that $q_s\inv \Po_{n_0} q_s\subset N_L(F)$. Since $\Po_{n_0}$ is generated by unipotent one-parameter subgroups, we obtain that $q_s\inv \Po_{n_0} q_s\subset N^1(F)$. So $s\in I_F(\phi)$. Now if $G_{n_0}\not\subset N^1(F)$, then $s\in E_{\Lambda}(\phi)$, which is a contradiction. Therefore $G_{n_0}\subset N^1(F)$. Since $q_s\in N_L(G_{n_0})$, we conclude that $q_s\inv G_{n_0}q_s\subset N_L(F)$. This verifies the condition for \Cref{prop:NonGen-s} for $q_s$ in place $g$; and hence $\mu_{t_i}\to \mu_{\cl{G_{n_0}x_s}}$ as $i\to\infty$. Now by combining \eqref{eq:poly-approx2} and \eqref{eq:mui}, we obtain 
\begin{align*}
&\lim_{i\to\infty} \abs{J}^{-1} \int_{\eta\in J} a_{t_i}u(\phi(s_i+e^{-{t_i}}\eta))g_i\delta_{x}\,d\eta\\
&=v_{\infty}(s)\inv \lim_{i\to\infty}\mu_{t_i}
=v_{\infty}(s)\inv\mu_{\cl{G_{n_0}x_s}}\quad \text{\small{(by \Cref{prop:NonGen-s})}}\\
&=\mu_{\cl{v_\infty(s)\inv G_{n_0}v(s)u(\phi(s))x}}=\mu_{\cl{G_{n_0}x}}, 
\end{align*}
because $v_{\infty}(s)\inv G_{n_0}v(s)u(\phi(s))=G_{n_0}$, see \eqref{eq:Gn0-vs}.

Therefore \eqref{eq:equidistribution} follows when $\nu$ is the normalized Lebesgue measure restricted to any given compact interval $J\subset (0,1)$ with non-empty interior. As explained in the beginning of \Cref{subsec:ThmMain}, this implies that \eqref{eq:equidistribution} holds for any absolutely continuous Borel probability measure $\nu$ on $\R$. This completes the proof of \Cref{thm:co-countable}. 
\end{proof}

\begin{proof}[Proof of Theorem~\ref{thm:smooth}] \label{subsec:thm1.1special} 
Let $E_x$ be the countable set as in Theorem~\ref{thm:co-countable}. Then $E_x$ is Lebesgue null. It is enough to show that for any given $f\in C_c(L/\Lambda)$ with $\int f\,d\mu_{\cl{G_{n_0}x}}=0$ and $\sup{|f|}\leq 1$ and a any given compact set $K\subset (0,1)\setminus E_x$ with Lebesgue measure $\abs{K}>0$, 
\[
\lim_{t\to\infty} \frac{1}{\abs{K}}\int_K f(a_{t}u(\phi(s))g_tx)\,ds =0.
\]
Suppose that this limit fails to hold. Then there exist an $\epsilon>0$ and a sequence $t_i\to\infty$ such that for each $i$, 
\[
\left\lvert\int_K f(a_{t_i}u(\phi(s))g_{t_i}x)\,ds\right\rvert>\abs{K}\epsilon .
\]

For each large $i$, we pick finitely many disjoint intervals of the form $(s,s+e^{-t_i})$ such that $(s,s+e^{-t_i})\cap K\neq \emptyset$ and the Lebesgue measure of the symmetric difference between their union and $K$ is less than $\abs{K}\epsilon/2$. And since $\abs{f}\leq 1$, for all large $i$, there exists $s_i\in (0,1)$ such that $(s_i,s_i+e^{-t_i})\cap K\neq\emptyset$ and 
\[
\left\lvert\int_{s_i}^{s_i+e^{-t_i}} f(a_{t_i}u(\phi(s))g_{t_i}x)\,ds\right\rvert>e^{-t_i}\epsilon/2.
\]
By passing a subsequence, $s_i\to s\in K\subset (0,1)\setminus E_x$, and for all large $i$, 
\[
\left\lvert\int_{0}^1 f(a_{t_i}u(\phi(s_i+e^{-t_i}\eta))g_{t_i}x)\,d\eta \right\rvert>\epsilon/2.
\]
This contradicts \eqref{eq:equidistribution}, because $\int f\,d\mu_{\cl{G_{n_0}x}} =0$. 
\end{proof}

\end{document}